\documentclass{article}
\usepackage[british]{babel}
\usepackage{amsmath,amssymb,stmaryrd,enumerate,latexsym,xcolor,hyperref,dsfont}
\usepackage[a4paper,portrait,
            left=3.5cm,
            right=3.5cm,
            top=3.5cm,
            bottom=3.5cm]{geometry}
            
\newcommand{\assign}{:=}
\newcommand{\backassign}{=:}
\newcommand{\cdummy}{\cdot}
\newcommand{\infixand}{\text{ and }}

\newcommand{\nobracket}{}
\newcommand{\nosymbol}{}

\newcommand{\tmop}[1]{\ensuremath{\operatorname{#1}}}
\newcommand{\tmstrong}[1]{\textbf{#1}}
\newcommand{\tmtextbf}[1]{\text{{\bfseries{#1}}}}
\newcommand{\tmtextit}[1]{\text{{\itshape{#1}}}}
\newenvironment{enumeratealpha}{\begin{enumerate}[a{\textup{)}}] }{\end{enumerate}}
\newenvironment{enumeratenumeric}{\begin{enumerate}[1.] }{\end{enumerate}}
\newenvironment{enumerateroman}{\begin{enumerate}[i.] }{\end{enumerate}}

\newenvironment{proof}{\noindent\textbf{Proof\ }}{\hspace*{\fill}$\Box$\medskip}
\newcommand{\NN}{\mathbb{N}}
\newcommand{\RR}{\mathbb{R}}
\DeclareMathOperator*{\esssup}{ess\,sup}
\newtheorem{definition}{Definition}
\newtheorem{theorem}[definition]{Theorem}
\newtheorem{lemma}[definition]{Lemma}
\newtheorem{remark}[definition]{Remark}
\newtheorem{proposition}[definition]{Proposition}
\newtheorem{corollary}[definition]{Corollary}

\title{Rough backward SDEs with discontinuous Young drivers}

\author{Dirk Becherer\footnote{\scriptsize Humboldt-Universit\"at zu Berlin, Berlin, Germany, \texttt{becherer(at)math.hu-berlin.de}.} \\ \and Yuchen Sun\footnote{\scriptsize Humboldt-Universit\"at zu Berlin and Technische Universität Berlin, Berlin, Germany, \texttt{yuchen.sun(at)hu-berlin.de}.} }

\numberwithin{equation}{section}
\numberwithin{definition}{section}

\begin{document}

\maketitle

\begin{abstract}
  We study solutions to backward differential equations that are driven hybridly by a deterministic discontinuous rough path $W$ of finite $q$-variation for $q \in [1, 2)$ and by Brownian motion $B$. To distinguish between integration of jumps in a forward- or Marcus-sense, we refer to these equations as forward- respectively Marcus-type rough backward stochastic differential equations \textbf{(RBSDEs)}. We establish global well-posedness by proving global apriori bounds for solutions and employing fixed-point arguments locally. Furthermore, we lift the RBSDE solution and the driving rough noise to the space of decorated paths endowed with a Skorokhod-type metric and show stability of solutions with respect to perturbations of the rough noise. Finally, we prove well-posedness for a new class of backward doubly stochastic differential equations (\textbf{BDSDEs}), which are jointly driven by a Brownian martingale $B$ and an independent discontinuous stochastic process $L$ of finite $q$-variation. We explain how our RBSDEs can be understood as conditional solutions to such BDSDEs, conditioned on the information generated by the path of $L$.  \vspace{5pt}
\end{abstract}

\tableofcontents

\section{Introduction}
We investigate backward differential equations hybridly driven by a Brownian
motion $B$ and a deterministic discontinuous rough path $W$ of finite
$q$-variation for $q < 2$.

Motivated by SDEs featuring It{\^o}'s-forward and Marcus-type jumps (see
{\cite{kurtz_stratonovich_1995}}, and also
{\cite{chevyrev_canonical_2019,friz_differential_2018}} from the rough path
literature), we differentiate between forward-type and geometric-type ways of
integration for jumps, which leads to two different notions, namely
forward-type rough backward SDEs (\tmtextbf{Forward-RBSDEs})
\begin{align}
  Y_t  =  \xi + \int_t^T f (r, Y_r, Z_r) \tmop{dr} - \int_t^T Z_r
  \tmop{dB}_r + \int_t^T g_r (Y_{r +}) \tmop{dW}_r,
  \quad t\in [0,T], \quad \mathbb{P} \text{-a.s.},
  \label{Forward-RBSDE}
\end{align}
or Marcus-type rough backward SDEs (\tmtextbf{Marcus-RBSDEs})
\begin{align}
  Y_t  =  \xi + \int_t^T f (r, Y_r, Z_r) \tmop{dr} - \int_t^T Z_r
  \tmop{dB}_r + \int_t^T g_r (Y_{r +}) \diamond \tmop{dW}_r ,
    \quad t\in [0,T], \quad \mathbb{P} \text{-a.s.}.
  \label{Marcus-RBSDE}
\end{align}
Here, $T$ is a finite time horizon, $B$ is a $d$-dimensional Brownian motion and $W$ is an $e$-dimensional
deterministic c{\`a}gl{\`a}d\footnote{While c\`adl\`ag paths are the standard choice for forward equations or forward integrals, in the present backward setting, where we work with backward Young integration (see Appendix~A), c\`agl\`ad paths are the more natural choice. } (left-continuous with right limits) path
of finite $q$-variation with jumps $\Delta W_t \assign W_{t +} - W_t$. 

While
$\int_t^T g_r (Y_{r +}) \tmop{dW}_r$ is defined as a backward Young integral
(see Appendix A) and $\int_t^T g_r (Y_{r +}) \diamond \tmop{dW}_r$ is taken to
be
\begin{equation*}
  \int_t^T g_r (Y_{r +}) \diamond \tmop{dW}_r = \int_t^T g_r (Y_{r +}) \tmop{dW}_r
  + \sum_{t \leq r < T} [\varphi (-g_r \Delta W_r, Y_{r +}) - Y_{r +} - g_r
  (Y_{r +}) \Delta W_r], 
\end{equation*}
as in {\cite{kurtz_stratonovich_1995}} with $\varphi (V, x) \assign \varphi
(V, x, 0)$ for $t \mapsto \varphi (V, x, t)$ denoting the solution to the ODE
\[ \frac{\tmop{dy}}{\tmop{dt}} (t) = V (y (t)),  \quad  y (1)
   = x, \quad 0 \leq t \leq 1.  \]
We use the notation $\int_t^T g_r (Y_{r +})
(\diamond) \tmop{dW}$ to simultaneously cover both $\int_t^T g_r (Y_{r +})  \tmop{dW}$ and $\int_t^T g_r (Y_{r +}) \diamond \tmop{dW}$, so
that both forward- and Marcus-type RBSDEs may be written in one unified form.
We say that a pair $(Y, Z)$ of multidimensional processes (of suitable dimensions) is a solution to the forward-type or
Marcus-type RBSDE with discontinuous Young drivers $W$ if it satisfies
the stochastic integral equation (\ref{Forward-RBSDE}) respectively (\ref{Marcus-RBSDE}).

\

To explain the difference in the jump dynamics in Forward-RBSDE
(\ref{Forward-RBSDE}) and Marcus-RBSDE (\ref{Marcus-RBSDE}), let us compare
the jump of the solution $Y$ at time $t$ for both RBSDEs in more detail. In
Forward-RBSDE the jump is given directly by the jump $- g_t (Y_{t +}) \Delta
W_t$ of the Young integral at time t, that we call the \tmtextbf{forward
jump}. Marcus-RBSDEs are inspired by Marcus-type SDEs (see
{\cite{kurtz_stratonovich_1995,marcus_modeling_1980}}). Intuitively, the
solution behaves at jumps as if suddenly accelerating and moving extremely
fast along the underlying vector field over an additional time interval.
Mathematically, this is accomplished by replacing $- g_t (Y_{t +}) \Delta
W_t$, the jump of the Young integral at time t, by $Y_{t +} - \varphi (-g_t \Delta
W_t, Y_{t +}) $, the so called \tmtextbf{geometric jump} or
\tmtextbf{Marcus jump}.

\

Our paper focuses on demonstrating well-posedness for RBSDEs with hybrid rough
and stochastic drivers and stability for such equations. Yet, an important motivation arises already from a purely stochastic setting, when we randomize
the rough driver $W$ to be a stochastic process $L$ that is independent of the
Brownian motion $B$. The resulting equation, formally written as
\begin{eqnarray}
  Y_t & = & \xi + \int_t^T f (r, Y_r, Z_r) \tmop{dr} - \int_t^T Z_r
  \tmop{dB}_r + \int_t^T g_r (Y_{r +}) (\diamond) \tmop{dL}, 
  \label{Marcus-BDSDE}
\end{eqnarray}
is known as backward doubly SDEs (BDSDEs).  BDSDEs have been introduced by Pardoux and Peng {\cite{pardoux_backward_1994}} for $L$ being an
independent Brownian motion, to provide a probabilistic representation for stochastic partial differential equations
(SPDEs). Diehl and Friz {\cite{diehl_backward_2012}} have shown that RBSDEs (driven
by ``frozen'' Brownian sample paths) are the conditional solutions to such
BDSDEs. We prove an analogous result for a new type of BDSDEs where $L$
instead is a possibly discontinuous stochastic process of finite $q$-variation
($q < 2$). We show how our RBSDE \eqref{Forward-RBSDE} and
\eqref{Marcus-RBSDE} can be derived from such BDSDE by freezing the sample
path $L$, and conversely, how our RBSDEs can be transformed into BDSDEs by
randomizing the rough driver $W$, see Proposition \ref{quenched-BDSDE} and
Theorem \ref{Existence-BDSDE} for details. Although this new type of BDSDE
excludes the case of $L$ being a Brownian motion, it accommodates a wide range
of other processes, including fractional Brownian motion with Hurst coefficient $H >
\frac{1}{2}$, as studied by Jing {\cite{jing_nonlinear_2012}}, or
pure-jump L\'evy processes (see Ch.\ref{solution-bdsde} for details), or a linear combination of independent processes of either kind. We note that our work has implications for the analysis of classes of non-linear stochastic partial differential equations with noise of the above kind, such SPDEs can be understood as rough PDEs for a fixed realization of the noise, a topic discussed further in the companion paper \cite{becherer_2026}.

\

In the pioneering work on RBSDEs by Diehl and Friz
{\cite{diehl_backward_2012}}, the key idea has been to approximate the rough driver $W$ by
a sequence of smooth drivers $W^n$ and to prove that the respective BSDE solutions
$(Y^n, Z^n)$ converge to some limiting process $(Y, Z)$, using classical
BSDE stability results. It is a natural but open question, whether $(Y, Z)$
solves the formal limiting RBSDE equation. Indeed, to make sense of the rough
integral term $\int g (Y) \tmop{dW}$ one needs some regularity of $Y$ (in
the rough path sense, cf.\ {\cite[ch.4]{friz_course_2020}}), which is however lacking from classical BSDE theory as employed in \cite{diehl_backward_2012}. In this sense, the limiting RBSDE there
has only a formal (but not an intrinsic) meaning, in that it cannot be
understood as an integral equation. Our work contributes to more recent developments
{\cite{diehl_backward_2017,liang_multidimensional_2025,song_backward_2025}} to obtain an intrinsic notion for RBSDE solutions by fixed point
methods, wherein regularity analysis for the solution is proven as an invariance
property for the fixed point map. Such arguments require rough drivers with
higher regularity (as in {\cite{diehl_backward_2017,song_backward_2025}} and
our paper) or function $g$ being linear (as in
{\cite{liang_multidimensional_2025}}).

\

Our paper extends beyond the continuous setting in
{\cite{diehl_backward_2017}}. Moreover, by deriving apriori bounds through
direct estimation, we can drop their restriction on $Y$ being one-dimensional,
which arises (naturally) in {\cite{diehl_backward_2017}} from applying
classical BSDE comparison result to obtain bounds on $Y$. We also allow for
more general integrands of the form $g (t, \omega, y)$ instead of $g (y)$.
Most importantly, we permit discontinuous rough drivers $W$ to be integrated
in both the forward and Marcus sense. We show well-posedness for both types
of RBSDEs and by lifting the equations to the space of decorated paths endowed
with a new type of ($p$-variation) Skorokhod metric (see
{\cite{chevyrev_superdiffusive_2024}}), generalizing the existing J1/M1
metric, we obtain the stability of solutions, particularly with respect to
discontinuous rough drivers $W$. We emphasize that all existing literature
about RBSDE is concerned with continuous systems. Even in the broader area of
rough stochastic differential equations (RSDE, see {\cite{friz_rough_2024}}),
we are only aware of the work by Allan and Pieper {\cite{allan_rough_2026}},
who study RSDEs with forward jumps, but they show stability analysis with
respect to the rough drivers $W$ measured in $p$-variation metric, which is
stronger than any kind of Skorokhod metric.

\

Regarding the discontinuous nature of the differential equation, note that discontinuities in the RBSDE solution arise solely from the jumps in the rough
driver $W$. We draw inspiration from the theory of rough differential equations (RDEs) with jumps, which has
been studied for forward-type RDEs in {\cite{friz_differential_2018}} and for
Marcus-type RDEs in
{\cite{chevyrev_canonical_2019,friz_general_2017,williams_path-wise_2001}}.
For Marcus-type RDEs, the well-known ``time-stretching'' method (cf. 
{\cite{cohen_singular_2021}}) has been used, while for forward-type RDEs the
well-posedness is typically proven by direct fixed-point arguments. In the
former approach, one adds \ at each jump of the driver $\phi$ fictitious time
intervals (of total length $\delta$) during which the jump is linearly
interpolated; we denote the resulting object $\phi^{\delta}$ and the
additional path segments ``linear excursion''; one can solve the continuous
RDE driven by $\phi^{\delta}$ and remove the fictitious time to recover a
discontinuous process that is then taken as the solution to the original RDE.
Compared to the fixed-point approach, the ``time-stretching'' method benefits
from building on already existing solution theory on continuous RDEs, thereby
naturally inheriting properties therefrom. In contrast, we directly prove the
well-posedness of both forward-type and Marcus-type RBSDEs through a fixed
point approach. In this sense, our approach is more ``intrinsic'', does not
rely on any previous results from continuous RBSDEs, and the global convergence of a sequence of Picard iterations to the solution (see Theorem
\ref{Picard-iteration}) is not only a useful result on its own, but also crucial for proving how RBSDE solutions
depend stably on $W$. Since the ``time-stretching'' idea provides intuition
and is still important, we show in Theorem \ref{fictioustime_Marcus-RBSDE} how
fixed-point solutions of Marcus-RBSDEs coincide with the solutions obtained
from the ``time-stretching'' method.

\

A key aspect in the theory of RDEs with jumps is to understand how small
perturbations in the rough driver $W$ affect the solution $(Y, Z)$ and whether
Wong--Zakai-type results can be achieved to approximate the solution by
those for smoothed drivers. The choice of suitable topologies for the solution
space and for the rough path space is crucial for such analysis. In our work,
we build on the theory of decorated paths recently introduced by Chevyrev et
al. {\cite{chevyrev_superdiffusive_2024}}. In essence, a decorated path $\Phi$
is a high-dimensional object consisting of a discontinuous rough path $\phi :
[0, T] \to \mathbb{R}^d$ together with additional information at each jump to
characterize the dynamics at jumps. This information is referred to as
``excursions'' and could, for instance, be simply ``linearly connecting the
jump'' or something more sophisticated to describe the trajectory that leads
to a Marcus jump. From such information, one can construct a path
$\phi^{\delta} : [0, T + \delta] \rightarrow \mathbb{R}^d$ similar as to the
``time-stretching'' method by adding fictitious time (of total length
$\delta$) to the jumps and interpolate the jumps according to the additional
information encoded in $\Phi$.

Let $\Lambda_I$ denote the set of strictly increasing bijections from an
interval $I$ onto itself, we define on the space of decorated paths the
Skorokhod-type metric:
\begin{equation}
  \alpha_{p ; [0, T]} (\Phi^1, \Phi^2) \assign \lim_{\delta \rightarrow 0}
  \inf_{\lambda \in \Lambda_{[0, T + \delta]}} (\| \lambda - \tmop{id}
  \|_{\infty} \vee \| \phi^{1, \delta} \circ \lambda - \phi^{2, \delta} \|_{p ;
  [0, T + \delta]}) . \label{heuristic-skorokhod}
\end{equation}
Readers familiar with Skorokhod metrics may recognize that if one takes $p =
\infty$ and chooses linear (or constant) excursions, then $\alpha_{p ; [0,
T]}$ coincides with the classical Skorokhod M1 (or respectively J1) metric.
This generalization is essential; for instance, as shown in Example 1.4 of
{\cite{chevyrev_superdiffusive_2020}}, even for simple two-dimensional
Marcus-type ODEs (with a bounded variation driver), the convergence of drivers
in the M1 norm is not sufficient to ensure convergence of the corresponding
solutions in J1 or M1. However, if the driver and the solution are embedded in
the space of decorated paths and measured by metric $\alpha_{\infty}$, then
convergence of the solutions can be ensured.

In Theorem \ref{stability-RBSDE-decorated} we show how for both forward-type
and Marcus-type RBSDEs the rough driver and the solution can be embedded into
the space of decorated paths. Consequently, for any sequence $W^n \rightarrow
W$ converging in the $\alpha_q$ topology, the corresponding solutions $(Y^n,
Z^n)$ converge $Y^n \rightarrow Y$ in $\alpha_p$ in probability, and also $Z^n
\rightarrow Z$ in $L^2  (\tmop{dt} \otimes \mathbb{P})$. To our best
knowledge, this is the first time stability in Skorokhod-type topologies has
been studied in a hybrid driver setting. To prove it, \ we apply a
doubly-indexed Picard scheme also used in the literature on the convergence of
filtrations in BSDEs {\cite{coquet_weak_2001,papapantoleon_stability_2023}}.
To this end, let $(Y^{n, k}, Z^{n, k})$ denote the $k$-th Picard iterate for
the $n$-th RBSDE. Rather than proving directly that $(Y^n, Z^n) \rightarrow
(Y, Z)$, we first show by induction that $(Y^{n, k}, Z^{n, k}) \rightarrow
(Y^k, Z^k)$ for every $k \in \mathbb{N}$ in the $\alpha_p$ topology. Then, by
establishing a uniform bound on these iterations, we obtain the desired
convergence by letting $k \rightarrow \infty$. This proof differs
significantly from the proof of stability for solutions in deterministic RDE
case (cf. Theorem 5.3 in {\cite{friz_differential_2018}}, Theorem 3.13 in
{\cite{chevyrev_canonical_2019}} or Proposition 8.16 in
{\cite{chevyrev_superdiffusive_2024}}). In the deterministic context,
stability proofs are naturally tied to contraction arguments in fixed point
methods, relying implicitly on the fact that the latter norm in
(\ref{heuristic-skorokhod}) (the $p$-variation norm) is the same norm as under
which the fixed point has been proven. In our setting, this would correspond
to replacing $\| \cdummy \|_p$ with the $\| \cdummy \|_{p, 2}$ norm on
$\mathcal{B}^p$ (see Definition \ref{YZ-statespace}) in the definition
\eqref{heuristic-skorokhod} of $\alpha_p$. But as one can see from the
definition of $\| \cdummy \|_{p, 2}$, this norm depends on the choice of
filtration and when measuring the distance between $Y^n$ and $Y$ in the above
Skorokhod norm $\alpha_p$, one encounters a term $Y^{n, \delta} \circ
\lambda$, which is adapted to a different filtration from $Y^{\delta}$ due to
the time reparametrization $\lambda$. This discrepancy makes it hard to see how one could interpret or bound $\|Y^{n, \delta} \circ \lambda - Y^{\delta}
\|_{p, 2}$. The new approach with the double-indexed Picard scheme avoids
directly estimating the RBSDE solutions in the Skorokhod metric and instead
only estimates the Picard iterations in it, thereby circumventing the
complications associated with the $\| \cdummy \|_{p, 2}$ norm.

\

The paper is organized as follows. Chapter 2 recalls 
preliminaries and introduces notations used throughout. The focus of
Chapter \ref{Chapter 3} is to prove well-posedness of the RBSDE. We establish
both the local and global existence and uniqueness of the solution for slightly generalized
forward- and Marcus-type RBSDEs, see (\ref{rBSDE-Marcus}-\ref{rBSDE-forward}). In Chapter~4.1, we introduce the space of decorated paths and the associated
Skorokhod metric, and we also prove a new alternative characterization of
convergence in this metric. Chapter \ref{Chapter 4} demonstrates 
stability of RBSDE solutions with respect to perturbations in $\xi, f, g$ and
$W$. Finally, Chapter 5 shows well-posedness for a new BDSDE, which corresponds
to the RBSDE theory developed, and explains how such an RBSDE can be seen as a 
conditional solution to such a BDSDE. Appendix A recalls backward Young
integration while Appendix B provides a version of It{\^o}'s formula, which applies to processes that are sums of a (continuous) local martingale and a
c{\`a}gl{\`a}d process of finite $q$-variation, $q < 2$. 

\section{Preliminaries}
Unless otherwise stated, all definitions, spaces, norms, and results introduced
on $[0,T]$ extend in the obvious way to arbitrary closed intervals
$[a,b]\subset [0,T]$.

\tmtextbf{Frequently used inequalities:} Throughout the paper, we often
use inequalities $\tmop{ab} \lesssim \lambda a^2 + \frac{1}{\lambda} b^2$
for $\lambda > 0$ or $(a + b)^2 \lesssim a^2 + b^2$ without further mentioning. The notation
$\lesssim$ means less or equal up to a constant factor.

\tmtextbf{Rough Paths:} Let $\mathcal{P}([0,T])$ denote the collection of all finite partitions
$\pi=(0=t_0<t_1<\cdots<t_n=T)$ of $[0,T]$.
For a normed space $(E,|\cdot|_E)$ and a path $x:[0,T]\to E$, writing
$x_{s,t}\assign x_t-x_s$, we define the \tmtextit{$p$-variation}
(semi-)norm by
\[
\|x\|_{p;[0,T];E}
\assign
\bigg(
\sup_{(t_i)\in\mathcal{P}([0,T])}
\sum_i |x_{t_i,t_{i+1}}|_E^p
\bigg)^{1/p}.
\]
We omit the target space and write simply $\|x\|_{p;[0,T]}$ whenever
$E$ is clear from the context. Similarly, we write
\[
\|x\|_{\infty;[0,T];E}
\assign \sup_{t\in[0,T]}|x_t|_E,
\]
again omitting $E$ when the target space is clear. We denote by
$D^p([0,T],E)$ the space of c{\`a}gl{\`a}d paths $x:[0,T]\to E$ of finite
$p$-variation, abbreviated as $D^p$ or $D^p([0,T])$ when the target space is
clear from the context. \\
We call a function $\omega$ from $\{(s,t)\,:\, 0 \leq s \leq t \leq T\}$ into $[0,
\infty)$ a \tmtextit{control} if it is null on the diagonal and
super-additive, i.e. $\omega (s, t) + \omega (t, u) \leq \omega (s, u)$ for
all $0 \leq s \leq t \leq u \leq T$. A control is called \tmtextit{regular} if it is
moreover continuous. For properties of controls not discussed here we refer  to {\cite{friz_multidimensional_2010}}. Notice
that most properties are stated therein  for regular controls (their setting being for continuous processes), but one can easily check
that they also hold for non-regular controls.

\tmtextbf{Stochastic objects:} We fix a filtered probability space $(\Omega,
\mathcal{F}, (\mathcal{F}_t)_{t \in [0, T]}, \mathbb{P})$, which supports a
$d$-dimensional continuous martingale $M$. The filtration $(\mathcal{F}_t)_{t
\in [0, T]}$ is given by the usual filtration of $M$. We further assume $M = B \circ c$ to be a time-changed Brownian motion with $B$ being a $d$-dimensional Brownian motion and $c:[0,T] \to [0,T]$ being a \textbf{deterministic} continuous non-decreasing surjective function. Naturally, for every local martingale $L$ on
$(\mathcal{F}_t)_{t \in [0, T]}$ there exists a predictable process $H$ in
$L^2 (\tmop{dc} \otimes \mathbb{P})$ with
\[ L_t = L_0 + \int_0^t H_r \tmop{dM}_r . \]
Throughout most of the paper, $M$ is simply taken to be the Brownian motion $B$. However, this slight generalization beyond the Brownian scheme is necessary. As we will see in Chapter \ref{rBSDEasDecoratedPath}, we apply ``time-stretching'' to the RBSDE, and the Brownian motion becomes a process $M$ whose trajectories are described by $M = B \circ c$. \\
For a random variable $X$ on $(\Omega, \mathcal{F}, \mathbb{P})$, 
$\| X \|_{L^\infty}$ denotes its $L^{\infty}$ norm.

\tmtextbf{Solution Space of RBSDE:} The following spaces were introduced in \cite{diehl_backward_2017} and play an essential role in our subsequent analysis. The reason for using this particular choice of solution spaces, rather than some of the alternative norms appearing in related works, is explained in Remark~\ref{Reason-p2Norm}. See also \cite{allan_rough_2026,liang_multidimensional_2025} for related choices in the literature.

\begin{definition}
  \label{YZ-statespace}For $p \geq 0$, define $\mathcal{B}_p ([0, T],
  \mathbb{R}^h) \backassign \mathcal{B}^p$ to be the space of adapted
  c{\`a}gl{\`a}d process $Y : \Omega \times [0, T] \rightarrow \mathbb{R}^h$
  with
  \[ \|Y\|_{p, 2 ; [0, T]} \assign \sup_{t \in [0, T]} \esssup_{\omega \in \Omega} \mathbb{E}_t [\|Y\|^2_{p ; [t, T]}]^{1 / 2} <
     \infty . \]
  Denote by $\tmop{BMO} ([0, T], \mathbb{R}^{h \times d}) \backassign
  \tmop{BMO}$ the space of all progressively measurable $Z : \Omega \times [0,
  T] \rightarrow \RR^{h \times d}$ with
  \[ \|Z\|_{\tmop{BMO} ; [0, T]} \assign \sup_{t \in [0, T]} \esssup_{\omega \in \Omega} \mathbb{E}_t \bigg[ \int_t^T |Z_r |^2
     \tmop{dc}_r \bigg]^{1 / 2} < \infty, \]
  where $| Z_r |$ denotes the Frobenius norm $| Z_r | \assign \sqrt{\tmop{tr} (Z_r^{\top} Z_r)}$. \\
  We further introduce a (semi)norm $\interleave Y, Z \interleave_{[0, T]}$
  on $\mathcal{B}^p \times \tmop{BMO}$ defined by  the sum
  \[ \interleave Y, Z \interleave_{[0, T]} \assign \|Y\|_{p, 2 ; [0, T]} +
     \|Z\|_{\tmop{BMO} ; [0, T]} . \]
  Notice that $\| \cdot \|_{p, 2}$ and therefore $\interleave \cdot
  \interleave$ is only a seminorm. One could make it into a
  norm by either adding $\| Y_T \|_{L^\infty}$ to the definition as in
  {\cite{diehl_backward_2017}}, or by restricting the (semi)norm domain to $Y_T=0$, by applying it only to differences $Y_T^1-Y_T^2=0$) with the same terminal $Y$-values, as later in
  (\ref{B-R-T-def}). In both variants, $(\mathcal{B}^p \times \tmop{BMO}, \interleave \cdot
  \interleave_{[T - \varepsilon, T]})$ becomes a Banach space.
\end{definition}

In the next lemma, we record some useful properties of the norms $\| \cdot
\|_{p, 2}$ and $\| \cdot \|_{\tmop{BMO}}$.

\begin{lemma} (and Definition.)
  \label{p2Norm-property} For $p \geq 0$, $Y \in \mathcal{B}^p$ and $Z \in
  \tmop{BMO}$, we have inequalities  
  \begin{enumeratealpha}
    \item $\sup_{t \in [0, T]} \|Y_t \|_{L^\infty} \leq \|Y_T \|_{L^\infty} +
    \|Y\|_{p, 2 ; [0, T]} ;$
    
    \item $\|Y\|_{p, 2 ; [t, T]} \leq \|Y\|_{p, 2 ; (t, T]} + \| \Delta
    Y_t \|_{L^\infty}$ and
    $\|Z\|_{\tmop{BMO} ; [t, T]} = \|Z\|_{\tmop{BMO} ; (t, T]}$ for $t \leq T$, where $\|Y\|_{p, 2 ; (t, T]} \assign \lim_{\varepsilon
    \rightarrow 0} \|Y\|_{p, 2 ; [t + \varepsilon, T]}$ and $\|Z\|_{\tmop{BMO} ; (t, T]} \assign \lim_{\varepsilon
    \rightarrow 0} \|Z\|_{\tmop{BMO} ; [t + \varepsilon, T]}$;
    
    \item $\frac{1}{2} (\|Y\|_{p, 2 ; [a, b]} +\|Y\|_{p, 2 ; [b, c]})
    \leq \|Y\|_{p, 2 ; [a, c]} \leq 2^{p - 1} (\|Y\|_{p, 2 ; [a, b]}
    +\|Y\|_{p, 2 ; [b, c]})$, $a < b < c$.
  \end{enumeratealpha}
\end{lemma}

\begin{proof}~\\
  a) $|Y_t | =\mathbb{E}_t [|Y_t |] \leq \|Y_T \|_{L^\infty} +\mathbb{E}_t [|Y_T
  - Y_t |] \leq \|Y_T \|_{L^\infty} +\mathbb{E}_t [\|Y\|_{p ; [t, T]}] \leq
  \|Y_T \|_{L^\infty} + \|Y\|_{p, 2 ; [0, T]}$.
  b) Using Minkowski's inequality in the first step, and dominated convergence
together with the tower property in the second step, we obtain
\begin{align*}
  \mathbb{E}_t [\|Y\|^2_{p ; [t,T]}]^{1/2}
  \leq {} &
  \mathbb{E}_t \Big[
  \lim_{\varepsilon \to 0}
  \|Y\|^2_{p ; [t+\varepsilon,T]}
  \Big]^{1/2}
  +
  \|\Delta Y_t\|_{L^\infty} \\
  \leq {} &
  \lim_{\varepsilon \to 0}
  \mathbb{E}_t \Big[
  \mathbb{E}_{t+\varepsilon}
  \big[
  \|Y\|^2_{p ; [t+\varepsilon,T]}
  \big]
  \Big]^{1/2}
  +
  \|\Delta Y_t\|_{L^\infty} \\
  \leq {} &
  \|Y\|_{p,2 ; (t,T]}
  +
  \|\Delta Y_t\|_{L^\infty},
\end{align*}
where for the last inequality we used
\[
\mathbb{E}_{t+\varepsilon}
[\|Y\|^2_{p ; [t+\varepsilon,T]}]^{1/2}
\leq
\|Y\|_{p,2 ; [t+\varepsilon,T]}
\leq
\|Y\|_{p,2 ; (t,T]}
\]
for every $\varepsilon>0$. It only remains to take $\sup_t \esssup_\omega$ to get the desired result. Similarly, one can show
\begin{align*}
  \mathbb{E}_t \Big[
  \int_t^T |Z_r|^2 \, dc_r
  \Big]^{1/2}
  = 
  \lim_{\varepsilon \to 0}
  \mathbb{E}_t \Big[
  \int_{t+\varepsilon}^T |Z_r|^2 \, dc_r
  \Big]^{1/2} 
  \leq
  \|Z\|_{\mathrm{BMO};(t,T]}.
\end{align*}
Together with $\mathbb{E}_{t+\varepsilon} \Big[
\int_{t+\varepsilon}^T |Z_r|^2 \, dc_r
\Big]^{1/2}
\leq
\|Z\|_{\mathrm{BMO};(t,T]}$ and $\|Z\|_{\mathrm{BMO};[t,T]}
\geq
\|Z\|_{\mathrm{BMO};(t,T]}$ implies the identity. \\
c) This follows from Lemma 4.6 in {\cite{r_m_dudley_introduction_nodate}}.
\end{proof}

\tmtextbf{Vector field:} For some finite dimensional Banach spaces $\mathbb{W},
\mathbb{V}$, we say a vector field $g : \mathbb{W} \rightarrow \mathbb{V}$ is
in $C^2_b  (\mathbb{W}, \mathbb{V})$ if $|g|_{C^2_b} \assign |g|_{\infty} + |
\tmop{Dg} |_{\infty} + |D^2 g|_{\infty} < \infty$, where $|g|_\infty$ denotes the supremum norm, and
$|
\tmop{Dg} |_{\infty}:= \sup_{x \in \mathbb{W}} \|\mathrm{D}g(x)\|_{\mathcal{L}(\mathbb{W}, \mathbb{V})}$ and
$|D^2 g|_{\infty}:= \sup_{x \in \mathbb{W}} \|\mathrm{D}^2 g(x)\|_{\mathcal{L}^2(\mathbb{W}, \mathbb{V})}$ denote the
corresponding operator norms. \\
We are interested in time-dependent random vector fields
\[
g:[0,T]\times\Omega\longrightarrow C^2_b(\mathbb{W},\mathbb{V}).
\]
We write $g(\omega)$ for the path $t\mapsto g(t,\omega)$ and, with a slight
abuse of notation, write $g(t,\omega,y)$ for $g(t,\omega)(y)$. We introduce
  the notation
\begin{align*}
  {}[[g]]_{p,2;[0,T]}
  &\assign \sup_{t\in[0,T]}\esssup_{\omega\in\Omega}
  \mathbb{E}_t
  \left[\|g(\omega)\|_{p;[t,T];C_b(\mathbb{W},\mathbb{V})}^2\right]^{\frac12}\\
  &= \sup_{t\in[0,T]}\esssup_{\omega\in\Omega}
  \mathbb{E}_t
  \left[
    \sup_{(t_i)\in\mathcal{P}([t,T])}
    \left(
      \sum_{i}
      \sup_{y\in\mathbb{W}}
      |g_{t_{i+1}}(\omega,y)-g_{t_{i}}(\omega,y)|^p
    \right)^{\frac{2}{p}}
  \right]^{\frac12}.
\end{align*}
Inspired by the notion of stochastic controlled vector field in
{\cite{friz_rough_2024}}, we 
define that a
time-dependent random vector field
$g:[0,T]\times\Omega\rightarrow C^2_b(\mathbb{W},\mathbb{V})$
is of class $D^{p, 2} C^2_b (\mathbb{W}, \mathbb{V})$
if
\begin{enumeratealpha}
  \item $g$ is progressively measurable with respect to the filtration
  $(\mathcal{F}_t)_{t \in [0, T]}$,
  
  \item $\sup_{t\in[0,T]}\||g_t|_{C^2_b}\|_{L^\infty}<\infty$,
  
  \item $g(\cdot,\omega,y)$ is a continuous path with finite $p$-variation
  for a.e. $\omega$ and every $y\in\mathbb{W}$,
  
  \item we have finiteness of  $[[g]]_{p, 2 ; [0, T]}
   < \infty$ and  $[[\tmop{Dg}]]_{p, 2 ; [0, T]} < \infty$.
\end{enumeratealpha}
The next lemma states multiple norm estimates  for composed maps of $g \in D^{p, 2} C_b^2
(\mathbb{W}, \mathbb{V})$, where $\mathbb{W}=\mathbb{R}^h$ and
$\mathbb{V}=\mathcal{L} (\mathbb{R}^e, \mathbb{R}^h)$, and an $x \in
\mathcal{B}_p ([0, T], \mathbb{R}^h)$,  
to be used later.

\begin{lemma}
  \label{composition}Let $x \in \mathcal{B}_p ([0, T], \mathbb{R}^h)$ and $g
  \in D^{p, 2} C_b^2 (\mathbb{R}^h, \mathcal{L} (\mathbb{R}^e,
  \mathbb{R}^h))$, then for almost every $\omega$ and every $t \in [0,
  T]$ it holds
  \[ \|g (\omega, x (\omega))\|_{p ; [t, T]} \leq \sup_{s \in [t, T]} |
     \tmop{Dg}_s (\omega) |_{\infty} \|x (\omega) {\|_{p ; [t,
     T]}}_{\nosymbol} + \|g(\omega)\|_{p ; [t,T];C_b}, \]
  this implies
  \[ \|g (x)\|_{p, 2 ; [0, T]} \leq \sup_{t \in [0, T]} \| |g_t|_{C^2_b}
     \|_{L^\infty} \| {x\|_{p, 2 ; [0, T]}}_{\nosymbol} + [[g]]_{p, 2 ; [0, T]}
     . \]
  For $i = 1, 2$, let $x^i \in
  \mathcal{B}_p$ and $g^i \in D^{p, 2} C_b^2$, and let 
  $x^{\Delta} \assign x^1 - x^2$ and $g^{\Delta} \assign g^1 - g^2$  denote the differences. Then, for
  a.e. $\omega$ and every $t \in [0, T]$, we have
  \begin{flalign}
    & \|g^1 (\omega, x^1 (\omega)) - g^2 (\omega, x^2 (\omega))\|_{p ; [t,
    T]} \nonumber\\
    \lesssim & \sup_{s \in [t, T]} | \tmop{Dg}^1_s (\omega) |_{\infty} 
    \|x^{\Delta} (\omega)\|_{p ; [t, T]} + \sup_{s
    \in [t, T]} |g^{\Delta}_s (\omega) |_{\tmop{Lip}} \|x^2 (\omega) \|_{p ; [t, T]}\nonumber \\
    & + \sup_{s \in [t, T]} |\tmop{D}^2 g^1_s (\omega) |_{\infty}  (\|x^1
    (\omega)\|_{p ; [t, T]} +\|x^2 (\omega)\|_{p ; [t, T]}) \sup_{r \in [0,
    T]} |x^{\Delta}_r (\omega) | \label{Reason-p2-Norm-1}\\
    & + \|\tmop{Dg}^1(\omega)\|_{p;[t,T];C_b}
    \sup_{r \in [0,T]} |x^\Delta_r(\omega)|
    +\|g^\Delta(\omega)\|_{p;[t,T];C_b} \nonumber.
  \end{flalign}
  If, moreover, $x^\Delta_T \in L^\infty$, then
  \begin{eqnarray*} 
    &  & \|g^1 (x^1) - g^2 (x^2)\|_{p, 2 ; [0, T]}\\
    & \lesssim & \sup_{t \in [0, T]} \||g^1_t |_{C^2_b} \|_{L^\infty} 
    (\|x^{\Delta} \|_{p, 2 ; [0, T]} + (\|x^1 \|_{p, 2 ; [0, T]} +\|x^2 \|_{p,
    2 ; [0, T]}) \sup_{t \in [0, T]} \| x^{\Delta}_t \|_{L^\infty})\\
    &  & + [[\tmop{Dg}^1]]_{p, 2 ; [0, T]} \sup_{t \in [0, T]} \|x^{\Delta}_t
    \|_{L^\infty} + [[g^{\Delta}]]_{p, 2 ; [0, T]} + \sup_{s \in [0, T]}
    \||g^{\Delta}_s (\omega) |_{\tmop{Lip}} \|_{L^\infty} \|x^2 \|_{p, 2 ; [0,
    T]} .
  \end{eqnarray*}
\end{lemma}

\begin{remark} \label{Reason-p2Norm}
    Let us explain the use of $\| \cdot \|_{p,2}$ instead of other potential norms on $Y$. A natural alternative would be the weaker norms given by $\| \cdot \|_{L^m_p; [0,T]} \assign \mathbb{E}[\| \cdot \|_{p; [0,T]}^{m}]^\frac{1}{m}$, $m\in \NN$. However, when trying to estimate the term $\|g(x^1) - g(x^2)\|_{p ; [t,T]}$, the product $(\|x^1\|_{p ; [t, T]} +\|x^2\|_{p ; [t, T]}) \sup_{t \in [0, T]} |x^{\Delta}t|$ appears in \eqref{Reason-p2-Norm-1}. This is a consequence of the non-linearity of $g$. Next, to take $\mathbb{E}[(\cdot)^m]^{\frac{1}{m}}$ on both sides and apply Hölder's inequality to separate the factors, we need the $L^{2m}$-integrability of $\|x^1\|_{p ; [t, T]}$, $\|x^2\|_{p ; [t, T]}$, and $\sup_{t \in [0, T]} |x^{\Delta}t|$. This means that, if we were to work with the norm $\| \cdot \|_{L^m_p; [0,T]}$, each iteration of the fixed-point map in Theorem~\ref{existenceuniqueness} would lead to a loss of half of the integrability, see \eqref{Reason-p2-Norm-2} and \eqref{Reason-p2-Norm-3}. When using $\| \cdot \|_{p,2}$ instead, this problem does not occur, since we have $\sup_{t \in [0, T]} \|x^{\Delta}t \|_{L^\infty} < \infty$ by Lemma~\ref{p2Norm-property}, and therefore we do not need to apply Hölder's inequality to separate this factor.\\
    Other alternatives are the norms from the recently developed theory of stochastic sewing \cite{le_stochastic_2020} and rough SDEs \cite{allan_rough_2026,friz_controlled_2024}, where the authors work with norms on stochastic processes $x$ that come from applying either $1/p$-Hölder or $p$-variation norm on the two-parameter process 
    \begin{align} \label{m-infty-norm}
        (s,t) \to \|x_{s,t} \mid \mathcal{F}_s\|_{m, \infty} := \esssup_\omega \mathbb{E}[| x_{s,t} |^{m} \mid \mathcal{F}_s]^\frac{1}{m}.
    \end{align}
    Unfortunately, while under $\| \cdot \|_{p,2}$, an estimate on the martingale part of the RBSDE the form $\| \int Z dB \|_{p,2} \lesssim \|Z\|_{\tmop{BMO}}$ follows directly by applying the Burkholder-Davis-Gundy inequality of {\cite[Theorem 14.12]{friz_multidimensional_2010}}, for reasons that go beyond the framework of this paper, the same estimate becomes much more difficult under the above norm. The work \cite{liang_multidimensional_2025} uses a similar norm, but with $\|x_{s,t} \mid \mathcal{F}_s\|_{m, \infty}$ in \eqref{m-infty-norm} replaced by the $L^m$ norm $\|x_{s,t}\|_{L^m}$. This leads to a simpler estimate on the martingale part, allowing this type of norm to be used in the RBSDE setting. However, the resulting norm suffers from the loss of integrability described earlier in this remark. This may be one reason why the authors consider RBSDEs with linear functions $g$, while allowing a more irregular driver $W$.
\end{remark}

\begin{proof}
  For $t \leq s < r \leq T$, we have
  \begin{align*}
    |g_r(\omega,x_r(\omega))-g_s(\omega,x_s(\omega))|
    \leq {}&
    |g_r(\omega,x_r(\omega))-g_r(\omega,x_s(\omega))|+|g_r(\omega,x_s(\omega))-g_s(\omega,x_s(\omega))|\\
    \leq {}&
    |\tmop{Dg}_r(\omega)|_\infty|x_{s,r}(\omega)|
    +\|g_r(\omega)-g_s(\omega)\|_\infty.
  \end{align*}
  Taking $p$-variation on both sides of the
  above inequality proves the first pathwise estimate. Applying $\sup_{t \in [0, T]} \esssup_{\omega \in \Omega}
  \mathbb{E}_t [(\cdummy)^2]^{\frac{1}{2}}$ to the pathwise estimate then yields the $\|\cdot\|_{p,2}$ estimate. For the difference estimate, taking $p$-variation on both sides of the
  following inequality yields the stated pathwise bound:
  \begin{align*}
    & |g^1_t (x_t^1) - g^2_t (x_t^2) - (g^1_s (x_s^1) - g^2_s (x_s^2)) |\\
    \leq & |g^1_t (x_t^1) - g^1_t (x_t^2) - g^1_t (x_s^1) + g^1_t (x_s^2)
    | + |g^1_t (x^1_s) - g^1_s (x^1_s) - g^1_t (x^2_s) + g^1_s (x^2_s) |\\
    & + |g^1_t (x^2_t) - g^2_t (x_t^2) - g^1_s (x^2_t) + g^2_s (x_t^2) | +
    |g^1_s (x^2_t) - g^2_s (x_t^2) - g^1_s (x^2_s) + g^2_s (x_s^2) |\\
    \leq & |D^2 g^1_t |_{\infty}  (|x^1_{s, t} | + |x^2_{s, t} |)
    |x^{\Delta}_t | + | \tmop{Dg}^1_t |_{\infty}  |x^{\Delta}_{s, t} |\\
    & + |g^1_{s, t} (x^1_s) - g^1_{s, t} (x^2_s) | + |g^{\Delta}_{s, t}
    (x_t^2) | + | g^\Delta_s (x_t^2) - g^\Delta_s (x_s^2) |\\
    \leq & |D^2 g^1_t |_{\infty}  (|x^1_{s, t} | + |x^2_{s, t} |)
    |x^{\Delta}_t | + | \tmop{Dg}^1_t |_{\infty}  |x^{\Delta}_{s, t} |
    + \|\tmop{Dg}^1_{s, t}\|_{\infty} |x^{\Delta}_s | + \|g^{\Delta}_{s, t}
    \|_{\infty} + |g^{\Delta}_s |_{\tmop{Lip}} |x^2_{s, t} |,
  \end{align*}
  where we have used Lemma 1 from {\cite{diehl_backward_2017}} in the second inequality. \\
  Next, under the additional assumption, we have $\sup_{t \in [0, T]}
  \|x^{\Delta}_t \|_{L^\infty} \leq \|x^{\Delta} \|_{p, 2 ; [0, T]} + \|
  x^{\Delta}_T \|_{L^\infty} < \infty$ by Lemma \ref{p2Norm-property}. Thus, we
  can simply apply $\sup_{t \in [0, T]} \esssup_{\omega \in \Omega}
  \mathbb{E}_t [(\cdummy)^2]^{\frac{1}{2}}$ to the pathwise difference estimate
  to obtain the corresponding conditional estimate.
\end{proof}

\section{Well-posedness of the RBSDE}\label{Chapter 3}

In this section, we prove the existence and uniqueness for solutions to RBSDEs of 
forward-type (\ref{Forward-RBSDE}) and  of Marcus-type
(\ref{Marcus-RBSDE}), in a slightly more general form. 

This section is structured as follows.
In Theorem \ref{existenceuniqueness} we show that the solution of the RBSDE
exists on small time intervals, where the length of the interval depends in
particular on the terminal condition and the $q$-variation of $W$. Then we
concatenate the local solutions to a global solution in Theorem
\ref{global_existenceuniqueness}. Of course, this is only possible if there is
no explosion. Therefore we start the section by first deriving an apriori bound
for solutions to the RBSDE in Theorem \ref{apriori-bound-YZ}. In addition, we
also show that  Picard iterations converge globally to the solution of the RBSDE
in Theorem \ref{Picard-iteration}. Such an iterative approximation scheme is a natural result being of interest in its
own. Moreover, it also turns out to be crucial for our proofs of stability for RBSDE solutions in Section~\ref{Stability-solution}.

For readers familiar with the theory of BSDE, it should not come as a surprise
that in many of the following proofs we need to apply It\^{o}'s formula to
$|Y|^2$, where $Y$ is a sum of a local martingale and a process of finite
$q$-variation (with $q < 2$), see e.g. (\ref{ItoY2}) or (\ref{Y2Con}). Yet, such clearly
is outside the scope of the classical It{\^o}'s formula, where $Y$ is required
be a semimartingale.
But by exploiting that the process
still exhibits finite pathwise quadratic variation in the sense of
F{\"o}llmer, one can adapt his ideas for a pathwise proof of It{\^o}'s formula from
{\cite{follmer_calcul_1981}} to our setting, 
see Appendix B for details.

\subsection{Apriori Bound}

In {\cite{diehl_backward_2017}}, the authors show that 
(continuous) RBSDE solutions, if they exist, are bounded from above and below by
the solutions of Young ODEs with drift, which is achieved by a limit argument and using a comparison theorem from classical BSDE theory. However, the existence of global
solutions for these Young ODEs is not proven and is unknown, at least to us. Instead of proving (or assuming) global
existence for the Young ODEs, we prove global apriori bounds here for RBSDE
solutions directly in Theorem \ref{apriori-bound-YZ}. Doing so also offers the benefit
that we do not need to restrict to BSDEs with one-dimensional $Y$, as would naturally be required when using classical BSDE comparison, as in  \cite{diehl_backward_2017}.\\

\noindent \tmtextbf{Assumption A:} We assume that
\begin{enumeratealpha}
  \item $q \in [1, 2)$, $p > 2$ with $\frac{1}{p} + \frac{1}{q} > 1$;
  
  \item the rough path $W$ is in $ D^q  ([0, T], \mathbb{R}^e)$; $M$ and $c$ are as in Chapter 2; $\xi$ is in $L^{\infty} (\mathcal{F}_T)$;
  
  \item the generator function $f : \Omega \times [0, T] \times \mathbb{R}^h \times \mathbb{R}^{h
  \times d} \rightarrow \mathbb{R}^h$ is progressively measurable (with respect to
   $(\mathcal{F}_t)_{t \in [0, T]}$). There exists some constant $C_f
  > 0$ such that
  \begin{eqnarray}
    \tmop{ess} \sup_{\omega} \sup_{t \in [0, T]} |f (t, 0, 0) | & \leq & C_f,
    \nonumber\\
    \tmop{ess} \sup_{\omega} \sup_{t \in [0, T]} |f (t, y, z) - f (t, y', z')
    | & \leq & C_f  (|y - y' | + |z - z' |) ;  \label{fLip}
  \end{eqnarray}
  \item $g $ is in $ D^{p, 2} C_b^2 (\mathbb{R}^h, \mathcal{L} (\mathbb{R}^e,
  \mathbb{R}^h))$, and there exists some constant $C_g > 0$ such that
  \[ \sup_{t \in [0, T]} \||g|_{C^2_b} \|_{L^\infty} \leq C_g, \quad
     [[g]]_{p, 2 ; [0, T]} \leq C_g, \quad [[\tmop{Dg}]]_{p, 2 ; [0, T]}
     \leq C_g . \] 
\end{enumeratealpha}
\begin{remark}
We work in a relatively general framework for the vector field $g$, allowing
it to depend on both $t$ and $\omega$. If $g$ is independent of these
variables, then Assumption A.d reduces to the classical condition $g \in C_b^2(\mathbb{R}^h,
\mathcal{L}(\mathbb{R}^e,\mathbb{R}^h))$, which is standard in Young and rough integration theory.
\end{remark}

\begin{theorem}
  \label{apriori-bound-YZ}Let $(Y, Z) \in \mathcal{B}^p \times \tmop{BMO}$ be
  a solution to the RBSDE with Marcus jumps
  \begin{align}
      Y_t = & \xi + \int_t^T f (r, Y_r, Z_r) \tmop{dc}_r + \int_t^T g_r
      (Y_{r +}) \tmop{dW}_r - \int_t^T Z_r \tmop{dM}_r \label{rBSDE-Marcus}\\
      & + \sum_{t \leq r < T} \varphi (-g_r \Delta W_r, Y_{r
      +}) - Y_{r +} - g_r (Y_{r +}) \Delta W_r, \nonumber
  \end{align}
or, respectively, to the RBSDE with forward jumps
\begin{equation}
    Y_t = \xi + \int_t^T f (r, Y_r, Z_r) \tmop{dc}_r + \int_t^T g_r (Y_{r
    +}) \tmop{dW}_r - \int_t^T Z_r \tmop{dM}_r .\label{rBSDE-forward}
\end{equation}
Provided that Assumption A holds, the values of $\|Y\|_{p, 2 ; [0, T]}$ and
$\|Z\|_{\tmop{BMO} ; [0, T]}$ are bounded by some constant $L$, whose choice
only depends on $C_f, C_g, | c_T |, \|W\|_{q ; [0, T]}$ and $\| \xi
\|_{L^\infty}$. In particular, we have $L_{\tmop{apriori}} \assign \sup_{t \in [0, T]} \| Y_t \|_{L^\infty} \leq
  \| \xi \|_{L^\infty} + \|Y\|_{p, 2 ; [0, T]} < \infty$ by Lemma \ref{p2Norm-property}.
\end{theorem}

\begin{proof}
  Instead of showing the apriori bound directly on the whole time interval, we
  start by showing it on $[T - \varepsilon, T]$ for some small $\varepsilon >
  0$. We also assume that $| c_{T - \varepsilon, T} | \leq
  \bar{\varepsilon}$ and $\|W\|_{2 ; [T - \varepsilon, T]} \leq \|W\|_{q
  ; [T - \varepsilon, T]} < \bar{\varepsilon}$ for some $\bar{\varepsilon} >
  0$. The choice of $\varepsilon, \bar{\varepsilon}$ will be specified later. \\
  In order to derive a bound for $\mathbb{E}_t [\|Y\|^2_{p ; [t, T]}]^{1 /
  2}$, we bound each term in (\ref{rBSDE-Marcus}) separately. \\
  By the Lipschitz property of $f$, Proposition 5.3 in
  {\cite{friz_multidimensional_2010}} and H{\"o}lder inequality, it follows
  \begin{align}
     & \mathbb{E}_t \bigg[ \bigg\| \int_t^. f (r, Y_r, Z_r) \tmop{dc}_r
    \bigg\|^2_{p ; [t, T]} \bigg]^{\frac{1}{2}} \label{fInv}\\
    \leq & \mathbb{E}_t \bigg[ \bigg\| \int_t^. f (r, Y_r, Z_r) \tmop{dc}_r
    \bigg\|^2_{1 ; [t, T]} \bigg]^{\frac{1}{2}} 
     \leq  \mathbb{E}_t \bigg[ \bigg( \int_t^T |f (r, Y_r, Z_r) | 
    \tmop{dc}_r \bigg)^2 \bigg]^{\frac{1}{2}} \nonumber\\
    \lesssim & C_f \mathbb{E}_t \bigg[ \bigg( \int_t^T |f (r, 0, 0) | 
    \tmop{dc}_r \bigg)^2 + {\bar{\varepsilon}^2}  \|Y\|^2_{\infty ; [t, T]} +
    \bar{\varepsilon} \int_t^T |Z_r {|^2}  \tmop{dc}_r \bigg]^{\frac{1}{2}}
    \nonumber\\
    \lesssim & \bar{\varepsilon} + \bar{\varepsilon} \|Y\|_{p,2 ; [t, T]} +
    \bar{\varepsilon} \|Y_T \| _{\infty} + \bar{\varepsilon}^{\frac{1}{2}} \|
    Z \|_{\tmop{BMO} ; [T - \varepsilon, T]}.  \nonumber
  \end{align}
  By Proposition \ref{BackwardsYoung} it follows
  \[ \bigg\| \int_t^. g_r (Y_{r +}) \tmop{dW}_r \bigg\|_{p ; [t, T]} \leq
     \bigg\| \int_t^. g_r (Y_{r +}) \tmop{dW}_r \bigg\|_{q ; [t, T]} \leq (|g_T
     (Y_T) | +\|g (Y)\|_{p ; [t, T]}) \|W\|_{q ; [t, T]} . \]
  Applying Lemma \ref{composition} and Corollary \ref{integral_q-var} yields
  \begin{align}
    \bigg\| \int_t^. g_r (Y_{r +}) \tmop{dW}_r \bigg\|_{p, 2 ; [T -
    \varepsilon, T]} \leq & (C_g  (1 + \| Y \|_{p, 2 ; [T - \varepsilon,
    T]}) + C_g) \|W\|_{q ; [T - \varepsilon, T]} .  \label{gInv}
  \end{align}
  By the Burkholder-Davis-Gundy inequality of {\cite[Theorem 14.12]{friz_multidimensional_2010}} we obtain
  \begin{eqnarray}
    \mathbb{E}_t \bigg[ \bigg\| \int_t^. Z_r \tmop{dM}_r \bigg\|^2_{p ; [t,
    T]} \bigg]^{1 / 2} & \lesssim & \mathbb{E}_t \bigg[ \int_t^T |Z_r |^2
    \tmop{dc}_r \bigg]^{1 / 2} .  \label{ZInv}
  \end{eqnarray}
  Applying Taylors formula to $u \mapsto \varphi (-g_r \Delta W_r, x, u)$ we
  get for some $\theta \in [0, 1]$
  \begin{align} \label{JumpInv}
    & \bigg| \sum_{k \leq r \leq l} \varphi (-g_r \Delta W_r, Y_{r +})
    - Y_{r +} - g_r (Y_{r +}) \Delta W_r \nobracket \bigg| \\
    \leq & \sum_{k \leq r \leq l} \bigg| \frac{1}{2} \bigg(g_r (\cdot) \tmop{Dg}_r
    (\cdot)\bigg) (\varphi (-g_r \Delta W_r, Y_{r +}, \theta)) \nobracket  \bigg|
    (\Delta W_r)^2 \nonumber\\
    \leq & \sum_{k \leq r \leq l} \frac{1}{2} (\sup_{t \in [0, T]}
    \||g|_{C^2_b} \|_{L^\infty})^2  (\Delta W_r)^2     \leq 
\frac{1}{2} \sup_{t \in [0, T]} \||g_t|_{C^2_b} \|^2_{\infty}
    \|W\|_{2 ; [k, l]}^2 . \nonumber
  \end{align}
  Notice that $\omega (k, l) \assign \|W\|_{2 ; [k, l]}^2$ defines a control
  (see {\cite{friz_multidimensional_2010}}, Proposition 5.8), so taking
  $p$-variation yields
  \begin{align}
    & \bigg\| \sum_{t \leq r < \cdot} \varphi (-g_r \Delta W_r, Y_{r
    +}) - Y_{r +} - g_r (Y_{r +}) \Delta W_r \nobracket \bigg\|^2_{p ; [t,
    T]} \label{pvarJumpInv} \\
    \leq & \bigg\| \sum_{t \leq r < \cdot} \varphi (-g_r \Delta W_r,
    Y_{r +}) - Y_{r +} - g_r (Y_{r +}) \Delta W_r \nobracket \bigg\|^2_{1 ;
    [t, T]} \leq \frac{1}{2} C_g^4 \|W\|_{2 ; [t, T]}^4. \nonumber
  \end{align}
  By combining the estimates (\ref{fInv}-\ref{pvarJumpInv}) and applying
  H{\"o}lder inequality, we have 
  \begin{align}
    \|Y\|_{p, 2 ; [T - \varepsilon, T]} \lesssim & \;\bar{\varepsilon} +
    \bar{\varepsilon} \|Y\|_{p, 2 ; [T - \varepsilon, T]} + \bar{\varepsilon} \|Y_T
    \|_{L^\infty} + \|Z\|_{\tmop{BMO} ; [T - \varepsilon, T]} + \|W\|_{2 ; [T -
    \varepsilon, T]}^2 \label{Y2} 
\\
    & + \bar{\varepsilon}^{\frac{1}{2}} \|Z\|_{\tmop{BMO} ; [T -
    \varepsilon, T]} + (C_g  (1 + | Y |_{p, 2 ; [T - \varepsilon, T]}
    \nobracket) \|W\|_{q ; [T - \varepsilon, T]} 
\nonumber\\
    \lesssim &  \; \bar{\varepsilon} + \bar{\varepsilon} \|Y_T \|_{L^\infty} +
    \bar{\varepsilon}^2 + \bar{\varepsilon} \|Y\|_{p, 2 ; [T - \varepsilon,
    T]} + \big( 1 + \bar{\varepsilon}^{\frac{1}{2}} \big) \|Z\|_{\tmop{BMO}
    ; [T - \varepsilon, T]},  \nonumber
  \end{align}
  This is not a satisfying final result yet, since the estimation of
  $\|Y\|_{p, 2 ; [T - \varepsilon, T]}$ depends on $\|Z\|_{\tmop{BMO} ; [T -
  \varepsilon, T]}$. A common technique from BSDE theory is to derive a second
  estimation by applying It{\^o}'s formula (Proposition \ref{Ito}) to $| Y_t |^2$, together with associativity of Young integral
  (Lemma \ref{associativityYoung}) we get
  \begin{align}
    | Y_t |^2 = & 2 \int_t^T Y^{\top}_r f (r, Y_r, Z_r) \tmop{dc}_r + 2
    \int_t^T Y^{\top}_{r +} g_r (Y_{r +}) \tmop{dW}_r - 2 \int_t^T Y^{\top}_r
    Z_r \tmop{dM}_r  \label{ItoY2}\\
    & + 2 \sum_{t \leq r < T} Y^{\top}_{r +} (\varphi (-g_r \Delta W_r, Y_{r +})
    - Y_{r +} - g_r (Y_{r +}) \Delta W_r ) \nonumber \\
    & - \int_t^T | Z_r |^2 \tmop{dc}_r + \sum_{t \leq r < T} [| Y_{r +}
    |^2 - | Y_r |^2 - 2 Y_r^{\top} (\Delta Y_r)] . \nonumber
  \end{align}
  Taking conditional expectations on both sides and making use of $| Y_t |^2>0$, yields 
  \begin{align*}
    \mathbb{E}_t \bigg[ \int_t^T | Z_r |^2
    \tmop{dc}_r \bigg]\leq & \mathbb{E}_t [| \xi |^2] + 2\mathbb{E}_t \bigg[ \int_t^T Y^{\top}_r
    f (r, Y_r, Z_r) \tmop{dc}_r \bigg] + 2\mathbb{E}_t \bigg[ \int_t^T
    Y^{\top}_{r +} g_r (Y_{r +}) \tmop{dW}_r \bigg]\\
    & + 2 \mathbb{E}_t \bigg[  \sum_{t \leq r < T} Y^{\top}_{r +} (\varphi (-g_r  \Delta W_r, Y_{r +})
    - Y_{r +} - g_r (Y_{r +}) \Delta W_r ) \nonumber \bigg]\\
    & +\mathbb{E}_t \bigg[ \sum_{t \leq r < T} [| Y_{r +} |^2 - | Y_r |^2
    - 2 Y_r^{\top} (\Delta Y_r)] \bigg] .
  \end{align*}
  We again bound the terms separately. Using Lipschitz continuity of $f$ and the
  basic inequality $\tmop{ab} \lesssim \lambda a^2 + \frac{1}{\lambda} b^2$
  for $\lambda > 0$, we obtain
  \begin{align*}
    \mathbb{E}_t \bigg[ \int_t^T Y^{\top}_r f (r, Y_r, Z_r) \tmop{dc}_r
    \bigg] \lesssim & C_f \mathbb{E}_t \bigg[ \int_t^T |Y_r | (|f (r, 0, 0) |
    + |Y_r | + |Z_r |) \tmop{dc}_r \bigg]  \\
    \lesssim & \mathbb{E}_t \bigg[ \int_t^T \frac{1}{\lambda}  | Y_r | ^2 +
    \lambda | f (r, 0, 0) | ^2 + \lambda | Y_r | ^2 + \lambda | Z_r |^2
    \tmop{dc}_r \bigg] . 
  \end{align*}
  Applying Lemma \ref{p2Norm-property} yields
  \begin{align}
    & \mathbb{E}_t \bigg[ \int_t^T Y^{\top}_r f (r, Y_r, Z_r) \tmop{dc}_r
    \bigg] \label{f2Inv}\\
    \lesssim & \lambda \bar{\varepsilon} \mathbb{E}_t [|f (r, 0, 0) |^2] +
    \lambda \mathbb{E}_t \bigg[ \int_t^T |Z_r |^2 \tmop{dc}_r \bigg] +
    \bar{\varepsilon}  (\lambda + 1 / \lambda)  (\mathbb{E}_t  [\|Y\|_{p ; [t,
    T]}^2] + |Y_T |^2) \nonumber\\
    \lesssim & \lambda (\bar{\varepsilon} +\|Z\|^2_{\tmop{BMO} ; [T -
    \varepsilon, T]} + \bar{\varepsilon} \|Y\|^2_{p, 2 ; [T - \varepsilon, T]}) +
    \bar{\varepsilon}  (\lambda + 1 / \lambda) (\|Y_T \|_{L^\infty}^2 +\|Y\|^2_{p,
    2 ; [T - \varepsilon, T]}).  \nonumber
  \end{align}
  By Corollary \ref{integral_q-var} and Lemma 2 from
  {\cite{diehl_backward_2017}}, we get
  \begin{align*}
    & \bigg| \int_t^T Y^{\top}_{r +} g_r (Y_{r +}) \tmop{dW}_r \bigg|\\
    \leq & (\|Y^{\top}_{r +} g_r (Y_{r +})\|_{p ; [t, T]} + C_g |Y_T |)
    \|W\|_{q ; [t, T]}\\
    \lesssim & (\|Y\|_{p ; [t, T]} \sup_{r \in [t, T]} |g_r (Y_{r +}) |
    +\|g_r (Y_{r +})\|_{p ; [t, T]} \sup_{r \in [t, T]} |Y_{r +} | + C_g |Y_T
    |) \|W\|_{q ; [t, T]}\\
    \lesssim & (\|Y\|_{p ; [t, T]} \sup_{r \in [t, T]} |g_r (\cdummy)
    |_{\infty} +\|g_r (Y_{r +})\|_{p ; [t, T]}^2 + (\|Y\|_{p ; [t, T]} + | Y_T
    |)^2) \|W\|_{q ; [t, T]} .
  \end{align*}
  Applying Lemma \ref{composition} yields and using that $\|W\|_{q ; [T -
  \varepsilon, T]} < \bar{\varepsilon}$, yields
  \begin{flalign}
     & \mathbb{E}_t \bigg[ \bigg| \int_t^T Y^{\top}_{r +} g_r (Y_{r +})
    \tmop{dW}_r \bigg| \bigg] \label{g2Inv}\\
    \lesssim & C_g \|Y\|_{p, 2 ;
    [T - \varepsilon, T]}  \bar{\varepsilon} + (C_g \| Y \|^2_{p, 2 ;
    [T - \varepsilon, T]} + [[g]]^2_{p, 2 ; [T - \varepsilon, T]} + \| Y
    \|^2_{p, 2 ; [T - \varepsilon, T]} + \| Y_T \|^2_{\infty}) 
    \bar{\varepsilon}. \nonumber
  \end{flalign}
  For the third term, we apply Lemma \ref{p2Norm-property} and reuse the estimate \eqref{JumpInv} to get
  \begin{align}
      &\mathbb{E}_t \bigg[  \sum_{t \leq r < T} Y^{\top}_{r +} (\varphi (-g_r \Delta W_r, Y_{r +})
    - Y_{r +} - g_r (Y_{r +}) \Delta W_r ) \nonumber \bigg] \\
     \leq & \frac{1}{2} (\| Y
    \|_{p, 2 ; [T - \varepsilon, T]} + \| Y_T \|_{L^\infty}) ) C_g^2  \|W\|_{2 ; [t, T]}^2.
  \end{align}
  For the last term, it holds by Remark \ref{ItoJump}
  \begin{align}
    \mathbb{E}_t \bigg[ \sum_{t \leq r < T} [| Y_{r +} |^2 - | Y_r |^2 - 2
    Y_r^{\top} (\Delta Y_r)] \bigg] & \leq 2\mathbb{E}_t  \bigg[
    \sum_{t \leq r < T} | \varphi (-g_r \Delta W_r, Y_{r +}) - Y_{r +}
    \nobracket |^2 \bigg] \label{phi2Inv}\\
    & \leq 2 \sup_{t \in [T - \varepsilon, T]} \||g_t|_{C^2_b}
    \|^2_{\infty}  \sum_{t \leq r < T} | \Delta W_r |^2 \leq 2 C^2_g  \bar{\varepsilon}^2, \nonumber    
  \end{align}
  where the second inequality follows by Taylor's formula.\\
  Combining the above estimates and using $|a| \leq 1 + |a|^2$ imply for some constant $c$ that
  \begin{eqnarray}
    \|Z\|^2_{\tmop{BMO}, [T - \varepsilon, T]} & \leq & c \lambda
    \|Z\|^2_{\tmop{BMO} ; [T - \varepsilon, T]} + c (\bar{\varepsilon} \lambda +
    \bar{\varepsilon} / \lambda + \bar{\varepsilon}^2) \|Y\|^2_{p, 2 ; [T - \varepsilon, T]}
    \nonumber\\
    &  & + c (\lambda \bar{\varepsilon} + \bar{\varepsilon} +
    \bar{\varepsilon}^2) + c (\bar{\varepsilon} \lambda + \bar{\varepsilon} /
    \lambda + \bar{\varepsilon}^2)  \| Y_T \|^2_{\infty} + \| Y_T \|^2_{\infty} . \nonumber
  \end{eqnarray}
  Choosing $\lambda$ small enough such that $c \lambda \leq \frac{1}{2}$, we obtain
  \begin{align}
    \|Z\|_{\tmop{BMO}, [T - \varepsilon, T]} \leq & c (\bar{\varepsilon}
    \lambda + \frac{\bar{\varepsilon}}\lambda + \bar{\varepsilon}^2)^{\frac{1}{2}} \|Y\|_{p, 2 ; [T -
    \varepsilon, T]} + c (\lambda \bar{\varepsilon} + \bar{\varepsilon} +
    \bar{\varepsilon}^2)^{\frac{1}{2}} \label{Z-apriori}\\
     & + c(\bar{\varepsilon} \lambda + \frac{\bar{\varepsilon}}{\lambda}+ \bar{\varepsilon}^2)^{\frac{1}{2}} \| Y_T\|_{L^\infty} +\| Y_T\|_{L^\infty}. \nonumber
  \end{align}
  Now substitute the term $\|Z\|_{\tmop{BMO}, [T - \varepsilon, T]}$ in (\ref{Y2}) with (\ref{Z-apriori}) to get
  \begin{align*}
      \|Y\|_{p, 2 ; [T - \varepsilon, T]} \leq & c \bigg(
    \bar{\varepsilon} + \bigg( 1 + \bar{\varepsilon}^{\frac{1}{2}} \bigg)
    (\bar{\varepsilon} \lambda + \bar{\varepsilon} / \lambda+ \bar{\varepsilon}^2)^{\frac{1}{2}} \bigg)
    \|Y\|_{p, 2 ; [T - \varepsilon, T]} \\
    & + c \bigg( \bar{\varepsilon} + \bar{\varepsilon}^2 + \bigg( 1 +
    \bar{\varepsilon}^{\frac{1}{2}} \bigg) (\lambda \bar{\varepsilon} +
    \bar{\varepsilon} + \bar{\varepsilon}^2)^{\frac{1}{2}}\bigg) \\
    & + c \bigg((\bar{\varepsilon} + 1) \| Y_T \|_{L^\infty} + \bigg( 1 +
    \bar{\varepsilon}^{\frac{1}{2}} \bigg) (\bar{\varepsilon} \lambda +
    \bar{\varepsilon} / \lambda + \bar{\varepsilon}^2)^{\frac{1}{2}}  \| Y_T \|_{L^\infty} \bigg).
  \end{align*}
  Now, we fix $\bar{\varepsilon}$ to be small enough such that $c \bigg( \bar{\varepsilon} + \bigg( 1 + \bar{\varepsilon}^{\frac{1}{2}}
    \bigg) (\bar{\varepsilon} \lambda + \bar{\varepsilon} / \lambda+ \bar{\varepsilon}^2)^{\frac{1}{2}}
    \bigg) \leq \frac{1}{2}.$ By   {\cite[Lem.4.7,4.8]{r_m_dudley_introduction_nodate}}, there exists a finite
  partition $\pi = \{0 = t_0 < t_1 < \cdots < t_N = T\}$ such that
  \begin{eqnarray*}
      \max_{i =1, \cdots, N} | c_{t_{i - 1}, t_i} |  \leq \bar{\varepsilon}, & \max_{i = 1, 
  \cdots, N} \|W\|_{q ; (t_{i - 1}, t_i]} \leq \bar{\varepsilon}, & N
  \leq 1 + \max \{ | c_T |, \|W\|_{q ; [0, T]} \} / \bar{\varepsilon}.
  \end{eqnarray*}
  Notice that the choice of $\lambda$ and $\bar{\varepsilon}$ only depends on $C_f, C_g, p$ and is therefore uniform for all intervals, so it holds for all $i = 1, \ldots, N$ that 
  \begin{eqnarray}
    \|Y\|_{p, 2 ; (t_{i - 1}, t_i]} \assign \lim_{\delta \rightarrow 0}
    \|Y\|_{p, 2 ; [t_{i - 1} + \delta, t_i]} & \leq & C_1 \| Y_{t_i}
    \|_{L^\infty} + C_2  \label{Y-localbound-first}
  \end{eqnarray}
  with $C_1 = c ( ( 1 + \bar{\varepsilon}^{\frac{1}{2}} )
  (\bar{\varepsilon} \lambda + \bar{\varepsilon} / \lambda+ \bar{\varepsilon}^2)^{\frac{1}{2}} +
  \bar{\varepsilon} + 1 )$ and $C_2 = c ( \bar{\varepsilon} +
  \bar{\varepsilon}^2 + ( 1 + \bar{\varepsilon}^{\frac{1}{2}} )
  (\lambda \bar{\varepsilon} + \bar{\varepsilon} +
  \bar{\varepsilon}^2)^{\frac{1}{2}} )$.
  Applying Taylor's formula yields for all $i = 1, \ldots, N$ that
  \begin{align}
    \| \Delta Y_{t_{i - 1}} \|_{L^\infty} =\|(\varphi (g_{t_{i - 1}}  \Delta
    W_{t_{i - 1}}, Y_{t_{i - 1} +}) - Y_{t_{i - 1} +} \|_{L^\infty} \leq
    C_g  | \Delta W_{t_{i - 1}} |
    \leq C_3, \label{JumpDynamic}
  \end{align}
  where $C_3 = C_g \|W\|_{q ; [0, T]}$. Applying Lemma \ref{p2Norm-property} allows us to conclude that
\(
    \|Y\|_{p, 2 ; [t_{i - 1}, t_i]} 
\) 
 is dominated by
\(
 C_1 \|Y_{t_i} \|_{L^\infty} + C_2 + C_3
\).
  This implies $\|Y_{t_{i - 1}} \|_{L^\infty} \leq \|Y_{t_i} \|_{L^\infty} + \|Y\|_{p,
     2 ; [t_{i - 1}, t_i]} \leq (C_1 + 1) \|Y_{t_i} \|_{L^\infty} + C_2 + C_3$ for all $i = 1, \ldots, N$. So, even though it is not apriori clear whether $\| Y_{t_i} \|_{L^\infty}$, $i = 1, \ldots, N - 1$, is bounded, we can derive
  a bound iteratively. We can, in fact, show by induction
  \begin{align*}
      \|Y_{t_i} \|_{L^\infty} 
       \leq & (C_1 + 1)^{N - i} \| Y_T \|_{L^\infty} + (C_2 + C_3)  \sum^{N
       - i}_{j = 1} (C_1 + 1)^{N - i - j}, & i=1,\cdots,N.
  \end{align*}
  Overall, this yields
  \begin{align*}
    \|Y\|_{p, 2 ; [0, T]} \leq & \sum_{i = 1}^N \|Y\|_{p, 2 ; [t_{i -
    1}, t_i]}\\
    \leq & \sum_{i = 1}^N C_1  \bigg( (C_1 + 1)^{N - i} \|Y_T
    \|_{L^\infty} + (C_2 + C_3) \sum^{N - i}_{j = 1} (C_1 + 1)^{N - i - j}
    \bigg) + C_2 + C_3\\
    \leq & N (C_1 + 1)^N \| Y_T \|_{L^\infty} + N^2  (C_2 + C_3)  (C_1 +
    1)^N.
  \end{align*}
  One can also attain a global apriori bound for $\|Z\|_{\tmop{BMO} ; [0, T]}$ by adding up (\ref{Z-apriori}). 
  
  \tmtextbf{$(Y, Z)$ being solution to (\ref{rBSDE-forward}):} The proof is even simpler, since the term
(\ref{pvarJumpInv}) no longer appears. We do obtain a different jump
dynamics in (\ref{phi2Inv}, \ref{JumpDynamic}), namely
$\Delta Y_t = - g_t (Y_{t+}) \Delta W_t$, but this can still be bounded by
$\|\Delta Y_t\|_{L^\infty} \leq C_g |\Delta W_t|$, and therefore does not
affect the subsequent estimates.
\end{proof}

\subsection{Existence and Uniqueness}

We start by showing the existence and uniqueness of the solution on a small
time interval of length $\varepsilon$ (see Theorem \ref{existenceuniqueness})
and specify the dependencies of the interval length in Remark
\ref{epsilon-dependency}. We then explain in Theorem
\ref{global_existenceuniqueness} how to construct the global solution by
iteratively ``gluing'' local solutions, similar to
{\cite{friz_course_2020}} for RDEs or in {\cite{zhang_backward_2017}} for
BSDEs.

Observe that in the proof below, it is essential to verify that the
fixed point map preserves the closed ball $B^{T, \varepsilon}_R$ with
radius $R > 0$ defined by
\begin{equation}
  B^{T, \varepsilon}_R \assign \left\{ (Y, Z) \in \mathcal{B}^p \times
  \mathrm{BMO} \hspace{0.17em} | \hspace{0.17em} Y_T = \xi, \hspace{0.17em}
  \interleave Y, Z \interleave_{[T - \varepsilon, T]} \leq R \right\} .
  \label{B-R-T-def}
\end{equation}
This invariance property, which is also commonly found in the rough path literature, ensures a
sufficient path regularity of $Y$, which is critical for the well-posedness of
the backward Young integral $\int g (Y) \hspace{0.17em} \mathrm{d} W$. The explicit bounds on the fixed-point map's image moreover serve to enable the subsequent
contraction argument.

\begin{theorem}[Local existence and uniqueness]
  \label{existenceuniqueness}Provided that Assumption A holds, there exists a
  sufficiently small $\bar{\varepsilon} > 0$ such for all $\varepsilon > 0$ with
  \begin{equation}
    | c_{T - \varepsilon, T} | < \bar{\varepsilon}  \infixand \| W \|_{q ; [T
    - \varepsilon, T]} < \bar{\varepsilon}, \label{barepsilon}
  \end{equation}
  the integral equation (\ref{rBSDE-Marcus}), and respectively the integral equation
  (\ref{rBSDE-forward}) each has a unique respective solution $(Y, Z)$ in $\mathcal{B}_p ([T -
  \varepsilon, T]) \times \tmop{BMO} ([T - \varepsilon, T])$ on $[T -
  \varepsilon, T]$.
\end{theorem}

\begin{proof}
  We will only show the proof for (\ref{rBSDE-Marcus}), the proof for
  (\ref{rBSDE-forward}) is essentially the same, and in fact slightly simpler, since it does not contain the infinite sums of Marcus jump terms.\\
  Fix some $\bar{\varepsilon} > 0$ that will be specified later in the proof.
  Notice that due to the left-continuity of $c$ and $W$, we have by Lemma 7.1
  in {\cite{friz_differential_2018}} that there always exists an $\varepsilon >
  0$ such that the condition \eqref{barepsilon} is satisfied. \\
  For $Y, Z \in B^{T, \varepsilon}_R$, we define the fixpoint map $\mathcal{M}^{T,
  \varepsilon} (Y, Z) = (\bar{Y}, \bar{Z})$  with
  \begin{align*}
    \bar{Y}_t \assign & \mathbb{E}_t  \bigg[ \nobracket \xi + \int_t^T f (r,
    Y_r, Z_r) \tmop{dc}_r + \int_t^T g_r (Y_{r +}) \diamond \tmop{dW}_r
    \bigg]
  \end{align*}
  with $\bar{Z}$ being defined by It\^o's  martingale representation 
  \begin{align*}
    & \bar{Y}_t + \int_{T - \varepsilon}^t f (r, Y_{r +}, Z_{r +}) \tmop{dc}_r +
    \int_{T - \varepsilon}^t g_r (Y_{r +}) \diamond \tmop{dW}\\
    = & \mathbb{E}_t  \bigg[ \nobracket \xi + \int_{T - \varepsilon}^T f (r,
    Y_{r +}, Z_{r +}) \tmop{dc}_r + \int_{T - \varepsilon}^T g_r (Y_{r +})
    \diamond \tmop{dW}_r \bigg] = \int_{T - \varepsilon}^t \bar{Z} \tmop{dM}_r + \bar{Y}_{T - \varepsilon},
  \end{align*}
  on the time interval  $t \in [T - \varepsilon, T]$. It follows by construction that $(\bar{Y},\bar{Z})$ satisfies 
  \begin{eqnarray}
    \bar{Y}_t & = & \xi + \int_t^T f (r, Y_r, Z_r) \tmop{dc}_r + \int_t^T g_r
    (Y_{r +}) \diamond \tmop{dW}_r - \int_t^T \bar{Z}_r \tmop{dM}_r . 
    \label{Ybar}
  \end{eqnarray}
  The following proof is divided into two major parts. First we show
  that the fixpoint map $\mathcal{M}^{T, \varepsilon}$ has an invariance property, meaning that
  for $R$ sufficiently large, $\mathcal{M}^{T, \varepsilon}$ maps $B^{T,
  \varepsilon}_R$ to itself. Afterwards, we will show a contraction property of
  $\mathcal{M}^{T, \varepsilon}$ in the second part.
  
  {\tmstrong{Invariance:}} Let $(Y, Z)$ be elements in $B^{T, \varepsilon}_R$ and
  denote $(\bar{Y}, \bar{Z}) =\mathcal{M}^{T, \varepsilon} (Y, Z)$.
  This proof is very similar to the proof of the apriori bound. We therefore omit some steps in the estimations. In order to derive a bound for $\mathbb{E}_t [\| \bar{Y} \|^2_{p ; [t,
  T]}]^{1 / 2}$, we will bound each term in (\ref{Ybar}). Similarly to (\ref{fInv}-\ref{pvarJumpInv}), we obtain the four inequalities
  \begin{align}
    \mathbb{E}_t \bigg[ \bigg\| \int_t^. f (r, Y_r, Z_r) \tmop{dc_r} \bigg\|_{p
    ; [t, T]} \bigg] & \lesssim \bar{\varepsilon} + \bar{\varepsilon} R + \bar{\varepsilon} 
    \| Y_T \|_{L^\infty} + \bar{\varepsilon}^{\frac{1}{2}} R ; \nonumber\\
    \bigg\| \int_t^. g_r (Y_{r +}) \tmop{dW_r} \bigg\|_{p, 2 ; [T -
    \varepsilon, T]} & \leq C_g  (1 + R) \|W\|_{q ; [T - \varepsilon,
    T]} ; \nonumber\\
    \mathbb{E}_t \bigg[ \bigg\| \int_t^. \bar{Z}_r \tmop{dM}_r \bigg\|^2_{p ;
    [t, T]} \bigg]^{1 / 2} & \lesssim \mathbb{E}_t \bigg[ \int_t^T |
    \bar{Z}_r |^2 \tmop{dc_r} \bigg]^{1 / 2} ; \nonumber\\
    \bigg\| \sum_{t \leq r < \cdot} \varphi (-g_r  (\cdot) \Delta W_r, Y_{r +})
    - Y_{r +}& - g_r (Y_{r +}) \Delta W_r \bigg\|^2_{p ; [t, T]} 
    \leq \frac{1}{2} C_g^4 \|W\|_{2 ; [t, T]}^4 .\nonumber
  \end{align}
  By combining the above estimates and H{\"o}lder's inequality, we finally have
  the bound
  \begin{eqnarray}
    \| \bar{Y} \|_{p, 2 ; [T - \varepsilon, T]} & \lesssim & \bar{\varepsilon}
    \big( 1 + \big\| {Y_T}  \big\|_{L^\infty} \big) + \bar{\varepsilon}^2
    + \big( \bar{\varepsilon} + \bar{\varepsilon}^{\frac{1}{2}} \big) R +
    \| \bar{Z} \|_{\tmop{BMO} ; [T - \varepsilon, T]}  \label{YInv}
  \end{eqnarray}
  In order to derive a second estimation, we again apply It{\^o}'s formula to $|
  \bar{Y}_t |^2$, take conditional expectations on both sides and make use of $| \bar{Y}_t |^2 \geq 0$, to conclude that
  \begin{align*}
    \mathbb{E}_t \bigg[ \int_t^T | \bar{Z}_r |^2
    \tmop{dc_r} \bigg] \leq& | \bar{Y}_T |^2 + 2\mathbb{E}_t \bigg[ \int_t^T
    \bar{Y}^{\top}_r f (r, Y_r, Z_r) \tmop{dc}_r \bigg] + 2\mathbb{E}_t
    \bigg[ \int_t^T \bar{Y}^{\top}_{r +} g_r (Y_{r +}) \tmop{dW}_r \bigg]\\
    & + \mathbb{E}_t \bigg[  \sum_{t \leq r < T} \bar{Y}^{\top}_{r +} (\varphi (-g_r \Delta W_r, Y_{r +})
    - Y_{r +} - g_r (Y_{r +}) \Delta W_r ) \nonumber \bigg] \\
    & +\mathbb{E}_t \bigg[ \sum_{t \leq r < T} [| \bar{Y}_{r +} |^2 - |
    \bar{Y}_r |^2 - 2 \bar{Y}_r^{\top} (\Delta \bar{Y}_r)] \bigg] .
  \end{align*}
  We can repeat the same estimation as in (\ref{f2Inv} -
  \ref{phi2Inv}), now of course also carefully distinguish the terms
  containing $\bar{Y}, \bar{Z}$ and $Y, Z$, we obtain the three inequalities\footnote{We also
  distinguish between $\bar{Y}_T$ and $Y_T$, which are actually both equal to $\xi$ here,
  but they are going to be different when we are going to apply the estimates derived here again to prove  Theorem
  \ref{Picard-iteration}.}
  \[ \mathbb{E}_t \big[ \int_t^T \bar{Y}^{\top}_r f (r, Y_r, Z_r)
     \tmop{dc}_r \big] \lesssim \lambda (\bar{\varepsilon} + R^2 +
     \bar{\varepsilon} R^2) + \frac{\bar{\varepsilon}}{\lambda} (\| \bar{Y}_T
     \|^2_{\infty} +\| \bar{Y} \|^2_{p, 2 ; [T - \varepsilon, T]}) +
     \bar{\varepsilon} \lambda (\|Y_T \|^2_{\infty} + R^2) ; \]
  \[ \mathbb{E}_t \big[ \big| \int_t^T \bar{Y}^{\top}_{r +} g_r (Y_{r +})
     \tmop{dW}_r \big| \big] \lesssim C_g \bar{\varepsilon} {\| \bar{Y} \|
     _{p, 2 ; [T - \varepsilon, T]}}_{\nosymbol} + (C_g R^2 + C_g + \| \bar{Y}
     \|^2_{p, 2 ; [T - \varepsilon, T]} + \| \bar{Y}_T \|^2_{\infty})
     \bar{\varepsilon} ; \]
  \[ \mathbb{E}_t \bigg[  \sum_{t \leq r < T} \bar{Y}^{\top}_{r +} (\varphi (-g_r  \Delta W_r, Y_{r +})
    - Y_{r +} - g_r (Y_{r +}) \Delta W_r ) \nonumber \bigg] \leq \frac{1}{2} C_g^2(\| \bar{Y}
    \|_{p, 2 ; [T - \varepsilon, T]} + \| \bar{Y}_T \|_{L^\infty}) )   \bar{\varepsilon}^2;
  \]
  \[ \mathbb{E}_t \bigg[ \sum_{t \leq r < T} [\Delta |\bar{Y}_r|^2 - 2
     \bar{Y}_r^\top (\Delta \bar{Y}_r) ] \bigg] \leq 2 C_g 
     \bar{\varepsilon}^2 . \]
  Combining the above estimates and using $|a| \leq 1 + |a|^2$ imply for some
  constant $c$ that
  \begin{align*}
    \mathbb{E}_t \big[ \int_t^T | \bar{Z} |^2 \tmop{dc} \big] \leq & 
    \| \bar{Y}_T \|^2_{\infty} + c \lambda (\bar{\varepsilon} + R^2 +
    \bar{\varepsilon} R^2) +  \frac{c \bar{\varepsilon}}{\lambda} \| \bar{Y}_T
    \|^2_{\infty} + c \bar{\varepsilon} \lambda \|Y_T \|^2_{\infty} +
    \frac{c \bar{\varepsilon}}{\lambda} \| \bar{Y} \|^2_{p, 2 ; [T - \varepsilon, T]}\\
    & + c \bar{\varepsilon} (1 + \| \bar{Y} \|^2_{p, 2 ; [T - \varepsilon,
    T]}) + c (R^2 +\| \bar{Y} \|^2_{p, 2 ; [T - \varepsilon, T]} + \|
    \bar{Y}_T \|^2_{\infty})  \bar{\varepsilon} \\ 
    & + (c+\| \bar{Y}
    \|_{p, 2 ; [T - \varepsilon, T]} + \| \bar{Y}_T \|_{L^\infty}) \bar{\varepsilon}^2 .
  \end{align*}
  Then for some constant $C_z$ (depending only on $C_f, C_g, p$) the following
  follows directly
  \begin{align}
    \| \bar{Z} \|_{\tmop{BMO} ; [T - \varepsilon, T]} \leq & C_z \lambda^{\frac{1}{2}} R  + C_z
    \bar{\varepsilon}^{\frac{1}{2}} ( 1 + \lambda^{^{\frac{1}{2}}} +
    \bar{\varepsilon}^{^{\frac{1}{2}}} )  +
    (1 + C_z \bar{\varepsilon}^{\frac{1}{2}} \lambda^{- \frac{1}{2}} + \bar{\varepsilon}) \| \bar{Y}_T
    \|_{L^\infty} \label{ZInv-2}  \\
    & + C_z
    \bar{\varepsilon}^{\frac{1}{2}} \lambda^{\frac{1}{2}} \| Y_T \|_{L^\infty}+ C_z \bar{\varepsilon}^{\frac{1}{2}} ( 1 +
    \lambda^{^{\frac{1}{2}}} ) R + C_z (\bar{\varepsilon} / \lambda +
    \bar{\varepsilon})^{\frac{1}{2}} \| \bar{Y} \| _{p, 2 ; [T - \varepsilon,
    T]} .\nonumber
  \end{align}
  Substituting (\ref{ZInv-2}) in (\ref{YInv}), yields for a constant $C_y$
  (depending only on $C_f, C_g, p$) that
  \begin{align}
    \| \bar{Y} \|_{p, 2 ; [T - \varepsilon, T]} \leq & C_y
    \lambda^{\frac{1}{2}} R + C_y \bigg( \bar{\varepsilon}^{\frac{1}{2}}
    \lambda^{\frac{1}{2}} + \bar{\varepsilon}^{\frac{1}{2}} \bigg) \|Y_T
    \|_{L^\infty} + C_y (1 + \bar{\varepsilon}^{\frac{1}{2}} \lambda^{- \frac{1}{2}}+\bar{\varepsilon}) \| \bar{Y}_T \|_{L^\infty} \nonumber\\
    & + F_y (\bar{\varepsilon})  \bigg( 1 + \lambda^{\frac{1}{2}} + \bigg( 1
    + \lambda^{\frac{1}{2}} \bigg) R + \bigg( 1 + \frac{1}{\lambda} \bigg)
    \| \bar{Y} \| _{p, 2 ; [T - \varepsilon, T]} \bigg),  \label{YInv-detail}
  \end{align}
  for some function $F_y$, which is right continuous at $0$, i.e. $F_y (0 +) =
  0$.\\
  We define $C = C_z \vee C_y$ and fix some $m \in \mathbb{N}$. We can choose
  $\lambda$ small such that $C \lambda^{\frac{1}{2}} \leq \frac{1}{m}$ and
  then choose $\bar{\varepsilon}$ small such that
  \begin{align}
     C  \bar{\varepsilon}^{\frac{1}{2}}  \bigg( 1 + \lambda^{\frac{1}{2}}
    + \bar{\varepsilon}^{\frac{1}{2}} \bigg) \leq \frac{1}{m},\,\,\,  C
    \bar{\varepsilon}^{\frac{1}{2}}  \bigg( 1 + \lambda^{\frac{1}{2}} \bigg)
    \leq \frac{1}{m}, & \, \, C  (\bar{\varepsilon} / \lambda +
    \bar{\varepsilon})^{\frac{1}{2}} \leq \frac{1}{m}, \,\,\,\, C
    \bar{\varepsilon}^{\frac{1}{2}} \lambda^{- \frac{1}{2}} +\bar{\varepsilon}\leq
    \frac{1}{m},  \nonumber\\
     C \bigg( \bar{\varepsilon}^{\frac{1}{2}} \lambda^{\frac{1}{2}} +
    \bar{\varepsilon}^{\frac{1}{2}} \bigg) \leq \frac{1}{m},\,  F_y
    (\bar{\varepsilon})  \bigg( 1 + \lambda^{\frac{1}{2}} \bigg) \leq
    \frac{1}{m}, &\, \, F_y (\bar{\varepsilon})  \bigg( 1 + \frac{1}{\lambda} \bigg)
    \leq \frac{1}{m} .   \label{condition-barepsilon}
  \end{align}
  This implies for (\ref{ZInv-2}-\ref{YInv-detail}) that
  \begin{align}
    \| \bar{Z} \|_{\tmop{BMO} ; [T - \varepsilon, T]} \leq &
    \frac{2}{m} R  + \frac{1}{m} \| Y_T \|_{L^\infty} + \bigg( 1 + \frac{1}{m}
    \bigg) \| \bar{Y}_T \|_{L^\infty} + \frac{1}{m} + \frac{1}{m} \| \bar{Y} \|
    _{p, 2 ; [T - \varepsilon, T]}  \label{ZInv-abstact}\\
    \| \bar{Y} \|_{p, 2 ; [T - \varepsilon, T]} \leq & \frac{2}{m} R +
    \frac{1}{m} \| Y_T \|_{L^\infty} + \bigg( \frac{1}{m} + C_y \bigg) \|
    \bar{Y}_T \|_{L^\infty} + \frac{1}{m} + \frac{1}{m} \| \bar{Y} \| _{p, 2 ;
    [T - \varepsilon, T]}  \label{YInv-abstact}
  \end{align}
  Choosing $R \geq m C_y  \| \xi  \|_{L^\infty} \vee 1$ and $m \geq
  11$ leads to
  \begin{equation}
    \| \bar{Y} \|_{p, 2 ; [T - \varepsilon, T]} \leq \frac{5}{m - 1} R
    \leq \frac{1}{2} R \quad \infixand \quad \| \bar{Z} \|_{\tmop{BMO} ;
    [T - \varepsilon, T]} \leq \frac{5}{m} R \leq \frac{1}{2} R.
    \label{YInv-final}
  \end{equation}
  We have thus shown that $\mathcal{M}^{T, \varepsilon} (B^{T, \varepsilon}_R) \subseteq B^{T,
  \varepsilon}_R$.
  
  {\tmstrong{Contraction:}}
  Let $\varepsilon, \bar{\varepsilon}$ be small enough such that the
  invariance property holds. \ .
  
  We define $(\bar{Y}, \bar{Z}) =\mathcal{M}^{T, \varepsilon} (Y, Z)$ and
  $(\overline{Y'}, \overline{Z'}) =\mathcal{M}^{T, \varepsilon} (Y', Z')$ for
  $(Y, Z), (Y', Z') \in B^{T, \varepsilon}_R$ and denote the differece by
  $H^{\Delta} \assign H - H'$ for $H = Y, Z, \bar{Y}, \bar{Z}$. For any $t \in [T - \varepsilon, T]$ holds
  \begin{align}
    \bar{Y}^{\Delta}_t  = & \bar{Y}^{\Delta}_T + \int_t^T [f (r, Y_r, Z_r) -
    f (r, Y'_r, Z'_r)] \tmop{dc}_r + \int_t^T [g_r (Y_{r +}) - g_r (Y'_{r +})]
    \tmop{dW}_r \nonumber\\
     & - \int_t^T \bar{Z}^{\Delta}_r \tmop{dM}_r + \sum_{t \leq r < T} \mathcal{J}_r,   \label{contraction}
  \end{align}
  where $\mathcal{J}_r$ is defined as
  \begin{align*}
    \mathcal{J}_r:= \varphi (-g_r \Delta W_r, Y_{r +}) - \varphi (-g_r \Delta W_r, Y'_{r
    +}) - Y_{r +} + Y'_{r +} + g_r (Y_{r +}) \Delta W_r - g_r (Y'_{r +})
    \Delta W_r.
  \end{align*}
  By definition of the fixed point map we have $Y^{\Delta}_T =
  \bar{Y}^{\Delta}_T = 0$, but in order to reuse the estimates for the proof
  of Theorem \ref{Picard-iteration}, we shall keep this term as a dummy
  variable.
  
  By Lipschitz continuity of $f$ and applying H{\"o}lder inequality, we get
  \begin{align}
   &\bigg\| \int_t^. [f (r, Y_r, Z_r) - f (r, Y'_r, Z'_r)] \tmop{dc}_r
    \bigg\|_{p ; [t, T]} 
    \leq 
    \int_t^T |f (r, Y_r, Z_r) - f (r, Y'_r,
    Z'_r) | \tmop{dc}_r 
    \nonumber\\
& \qquad 
    \leq 
      C_f  \int_t^T |Y^{\Delta}_r | \tmop{dc}_r 
     + C_f  \int_t^T
    |Z^{\Delta}_r | \tmop{dc}_r 
    \lesssim 
    \; \bar{\varepsilon} \sup_{r \in [T - \varepsilon, T]} |
    Y^{\Delta}_r | + \bar{\varepsilon}^{1 / 2} \bigg( \int_t^T |Z^{\Delta}_r
    |^2 \tmop{dc}_r \bigg)^{\frac 1 2}. \nonumber
  \end{align}
  This implies with Lemma \ref{p2Norm-property} that
  \begin{eqnarray}
    &  & \bigg\| \int_t^. [f (r, Y_r, Z_r) - f (r, Y'_r, Z'_r)] \tmop{dc}_r
    \bigg\|_{p, 2 ; [T - \varepsilon, T]} \nonumber\\
    & \lesssim & \bar{\varepsilon} (\| Y^{\Delta} \|_{p, 2 ; [T -
    \varepsilon, T]} + \| Y^{\Delta}_T \|_{L^\infty}) + \bar{\varepsilon}^{1 /
    2} \| Z^{\Delta} \|_{_{\tmop{BMO} ; [T - \varepsilon, T]}}.  \label{fCon}
  \end{eqnarray}
  We get by Corollary \ref{integral_q-var} and Lemma \ref{composition} that
  \begin{align}
     & \bigg\| \int_t^. g (Y_{r +}) - g (Y'_{r +}) \tmop{dW}_r \bigg\|_{p, 2
    ; [T - \varepsilon, T]} \nonumber\\
    \leq &  (\| g (Y) - g (Y') \|
  _{p,2 ; [t, T]} + | g (Y_T) - g (Y'_T) |_\infty) \| W \| _{q ; [T- \varepsilon, T]} \label{Reason-p2-Norm-2}\\
    \lesssim & \sup_{t \in [0, T]} \||g^1 |_{C^2_b} \|_{L^\infty}  \|Y^{\Delta}
    \|_{p, 2 ; [T - \varepsilon, T]}  \| W \| _{q ; [T - \varepsilon, T]}
    \nonumber\\
    & + \sup_{t \in [0, T]} \||g^1 |_{C^2_b} \|_{L^\infty} (\|Y\|_{p, 2 ; [T
    - \varepsilon, T]} +\|Y' \|_{p, 2 ; [T - \varepsilon, T]}) \sup_{t \in [T
    - \varepsilon, T]} \| Y^{\Delta}_t \|_{L^\infty} \| W \| _{q ; [T -
    \varepsilon, T]} \nonumber\\
    &  + [[\tmop{Dg}^1]]_{p, 2 ; [0, T]} \sup_{t \in [T - \varepsilon, T]}
    \|Y^{\Delta}_t \|_{L^\infty} \| W \| _{q ; [T - \varepsilon, T]} + C_g  \|Y^{\Delta}_T \|_{L^\infty} \| W \| _{q ; [T - \varepsilon, T]}\nonumber\\
    \lesssim & C_g (\|Y^{\Delta} \|_{p, 2 ; [T - \varepsilon, T]} + \|
    Y^{\Delta}_T \|_{L^\infty}) (1 + R) \| W \| _{q ; [T - \varepsilon, T]} . 
    \label{gCon}
  \end{align}
  The Burkholder-Davis-Gundy inequality {\cite[Theorem 14.12]{friz_multidimensional_2010}}
  implies that
  \begin{eqnarray}
    \mathbb{E}_t \bigg[ \bigg\| \int_t^. \bar{Z}^{\Delta}_r \tmop{dM}_r
    \bigg\|^2_{p ; [t, T]} \bigg]^{1 / 2} & \leq & C_p \mathbb{E}_t
    \bigg[ \int_t^T | \bar{Z}^{\Delta}_r |^2 \tmop{dc_r} \bigg]^{1 / 2} 
    \label{ZCon}
  \end{eqnarray}
  for a constant $C_p$ depending only on $p$. \\
  Using Taylor approximation of $u \mapsto \varphi (-g_r \Delta W_r, Y_{r +}, u) - \varphi (-g_r \Delta W_r, Y'_{r +}, u)$ yields 
  \begin{align*}
    | \mathcal{J}_r | \leq \frac{1}{2}  
    \left|g_r \cdot \tmop{Dg}_r (\varphi (-g_r \Delta W_r, Y_{r
    +}, \theta)) - g_r \cdot \tmop{Dg}_r (\varphi (-g_r \Delta W_r, Y'_{r +},
    \theta)) \right| \cdot |\Delta W_r|^2, 
  \end{align*}
  for some $\theta \in [0, 1]$. 
  By Lipschitz continuity of $g_r \cdot \tmop{Dg_r}$ and Gronwall's inequality  
  {\cite[Thm.3.15]{friz_multidimensional_2010}} we get 
  \begin{align}
    \mathcal{J}_r \leq & | g_r |_{C_b^2}  | \varphi (-g_r \Delta W_r, Y_{r +}, \theta) -
    \varphi (-g_r \Delta W_r, Y'_{r +}, \theta) || \Delta W_r |^2 \nonumber\\
     \leq & | g_r |_{C_b^2} |Y^{\Delta}_{r +} | e^{| \tmop{Dg}_r |_{\infty}  |
    \Delta W_r |}  | \Delta W_r |^2, \nonumber
  \end{align}
  and summing up over jump times $r$ of $W$ yields
  \begin{align} \label{J-estimate}
    \sum_{t \leq r < \cdot} | \mathcal{J}_r |
    \leq & c (\|Y^{\Delta} \|_{p ; [t, T]} + \| Y^{\Delta}_T \|_{L^\infty})
    e^{c \|W\| _{q ; [T - \varepsilon, T]}}  \sum_{t \leq r < \cdot} | \Delta
    W_r |^2 \\
    \leq & c (\|Y^{\Delta} \|_{p ; [t, T]} + \| Y^{\Delta}_T \|_{L^\infty})
    e^{c \bar{\varepsilon}} \|W\|^2_{2 ; [t, .]}, \nonumber
  \end{align}
  and then, by a similar argument as in equation (\ref{pvarJumpInv}) and
  taking p-variation, conditional expectation and $\esssup_\omega \sup_t$, we
  obtain
  \begin{align}
    \big\| \sum_{t \leq r < \cdot} \mathcal{J}_r  \big\| _{p ; 2, [T - \varepsilon,
    T]} \leq c (\|Y^{\Delta} \|_{p, 2 ; [T - \varepsilon, T]} + \| Y^{\Delta}_T
    \|_{L^\infty}) e^{c \bar{\varepsilon}} \|W\|^2_{2 ; [T - \varepsilon, T]} . 
    \label{JumpCon}
  \end{align}
  By combining (\ref{fCon}-\ref{JumpCon}) and noticing that $R > 1$, we get
  \begin{align}
    \| \bar{Y}^{\Delta} \|_{p, 2 ; [T - \varepsilon, T]} \leq & c (\|Y^{\Delta}
    \|_{p, 2 ; [T - \varepsilon, T]} + \| Y^{\Delta}_T \|_{L^\infty})  (R
    \bar{\varepsilon} + e^{c\|W\|q ; [0, T]}  \bar{\varepsilon}^2) \nonumber\\
    & + c \bar{\varepsilon}^{1 / 2} \|Z^{\Delta} \|_{\tmop{BMO} ; [T -
    \varepsilon, T]} + C_p  \| \bar{Z}^{\Delta} \|_{\tmop{BMO} ; [T - \varepsilon,
    T]}, \nonumber
  \end{align}
  For some suitable function $F$, being right continuous at $0$,  we can simplify the above as 
  \begin{align}
    \| \bar{Y}^{\Delta} \|_{p, 2 ; [T - \varepsilon, T]} \leq & F
    (\bar{\varepsilon})  (R (\|Y^{\Delta} \|_{p, 2 ; [T - \varepsilon, T]} + \|
    Y^{\Delta}_T \|_{L^\infty}) +\|Z^{\Delta} \|_{\tmop{BMO} ; [T - \varepsilon,
    T]})  \label{YCon}\\
    & + C_p  \| \bar{Z}^{\Delta} \|_{\tmop{BMO} ; [T - \varepsilon, T]}, \nonumber
  \end{align}
  Similar to the invariance proof, we apply Itô's
  formula on $| \bar{Y}^{\Delta}_t |^2$, take conditional expectations on both sides and make use of $| \bar{Y}_t |^2 \geq 0$ to get
  \begin{align}
    \mathbb{E}_t \bigg[\int_t^T | \overline{Z}^{\Delta}_r |^2
    \tmop{dc}_r \bigg] \leq &  \| \bar{Y}^{\Delta}_T \|_\infty^2 + 2 \mathbb{E}_t \bigg[\int_t^T \bar{Y}^{\Delta}_r
    (f (r, Y_r, Z_r) - f (r, Y'_r, Z_r')) \tmop{dc}_r \bigg] \label{Y2Con} \\
    & + 2 \mathbb{E}_t \bigg[ \int_t^T \bar{Y}^{\Delta}_r  (g_r (Y_{r +}) - g_r (Y'_{r +}))
    \tmop{dW}_r \bigg ] + 2 \mathbb{E}_t \bigg[ \sum_{t \leq r < T} \bar{Y}_{r +}^{\Delta} \mathcal{J}_r \bigg ]  \nonumber \\
    &+ \mathbb{E}_t \bigg[ \sum_{t \leq r < T} [| \bar{Y}_{r +}^{\Delta} |^2 - |
    \bar{Y}_r^{\Delta} |^2 - 2 \Delta \bar{Y}^{\Delta}_r (\bar{Y}^{\Delta}_r)] \bigg ]   \nonumber
  \end{align}
  By Lipschitz continuity of $f$ and using Lemma \ref{p2Norm-property} for the last inequality, we see that
  \begin{align}
    & \mathbb{E}_t \bigg[ \bigg| \int_t^T \bar{Y}^{\Delta}_r (f (r, Y_r, Z_r)
    - f (r, Y'_r, Z_r')) \tmop{dc}_r \bigg| \bigg] \label{fY2Con}\\
    \lesssim & \mathbb{E}_t \bigg[ \int_t^T | \bar{Y}^{\Delta}_r |
    (|Y^{\Delta}_r | + |Z^{\Delta}_r |) \tmop{dc}_r \bigg] \nonumber\\
    \lesssim & \mathbb{E}_t \bigg[ \int_t^T \bigg( \frac{1}{\lambda} |
    \bar{Y}^{\Delta}_r |^2 + \lambda |Y^{\Delta}_r |^2 + \lambda |Z^{\Delta}_r
    |^2 \bigg) \tmop{dc}_r \bigg] \nonumber\\
    \lesssim & \mathbb{E}_t \bigg[ \lambda \int_t^T |Z^{\Delta}_r |^2
    \tmop{dc}_r + \lambda \bar{\varepsilon} \sup_{r \in [t, T]} | Y^{\Delta}_r 
    |^2 + \frac{\bar{\varepsilon}}{\lambda} \sup_{r \in [t, T]} |
    \bar{Y}^{\Delta}_r  |^2 \bigg] \nonumber\\
    \lesssim & \lambda \| Z^{\Delta} \|_{\tmop{BMO} ; [T - \varepsilon, T]}^2
    + \lambda \bar{\varepsilon}  (\| Y^{\Delta} \|^2_{p, 2 ; [T - \varepsilon,
    T]} + \| Y^{\Delta}_T \|_{L^\infty}^2) + \frac{\bar{\varepsilon}}{\lambda}
    (\| \bar{Y}^{\Delta} \|^2_{p ; [t, T]} + \| \bar{Y}^{\Delta}_T
    \|_{L^\infty}^2) . \nonumber
  \end{align}
  For the following step, we use Lemma 2 from {\cite{diehl_backward_2017}} in the
first inequality, Lemma \ref{composition} in the second and $\tmop{ab}
\lesssim a^2 + b^2$ in the third, to get

\begin{align}
  & \mathbb{E}_t \big[\| \bar{Y}^{\Delta}_{r +} (g_r (Y_{r +}) - g_r (Y'_{r
  +}))\|_{p ; [t, T]} \big] \nonumber\\
  \leq & \mathbb{E}_t \big[\| \bar{Y}^{\Delta} \|_{p ; [t, T]} \sup_{r \in [t, T]}
  | g_r |_{C_b^2} | Y^{\Delta}_r | + \sup_{r \in [t, T]} | \bar{Y}^{\Delta}_r
  | \|g (Y) - g (Y')\|_{p ; [t, T]}\big] \nonumber\\
  \lesssim & \mathbb{E}_t \big[\| \bar{Y}^{\Delta} \|_{p ; [t, T]} \sup_{r \in [t,
  T]} | Y^{\Delta}_r | + \sup_{r \in [t, T]} | \bar{Y}^{\Delta}_r | \|Y^{\Delta} \|_{p ; [t, T]} \big] \nonumber\\
  & + \mathbb{E}_t \big[\sup_{r \in [t, T]} | \bar{Y}^{\Delta}_r |
  (1 + (\| Y \|_{p ; [t, T]} + \| Y' \|_{p ; [t,
  T]}) \sup_{r \in [t, T]} | Y^{\Delta}_r |) \big] \nonumber\\
  \lesssim & \mathbb{E}_t \big[\| \bar{Y}^{\Delta} \|_{p ; [t, T]}^2 + \sup_{r \in
  [t, T]} | Y^{\Delta}_r |^2 + \sup_{r \in [t, T]} | \bar{Y}^{\Delta}_r |^2
  +\|Y^{\Delta} \|_{p ; [t, T]}^2 \big] \nonumber\\
  & +\mathbb{E}_t \big[(1+\| Y \|_{p ; [t, T]} + \| Y' \|_{p ; [t, T]}) (\sup_{r \in [t, T]} | \bar{Y}^{\Delta}_r |^2 + \sup_{r \in
  [t, T]} | Y^{\Delta}_r |^2) \big] \label{Reason-p2-Norm-3}
\end{align}
Thereby, using the above estimate, Corollary, Lemma \ref{p2Norm-property} and $R > 1$,
we obtain
\begin{align}
  & \mathbb{E}_t \bigg[ \bigg| \int_t^T \bar{Y}^{\Delta}_{r +} (g_r (Y_{r +})
  - g_r (Y'_{r +})) \tmop{dW}_r \bigg| \bigg] \label{gY2Con}\\
  \lesssim & R (\|Y^{\Delta} \|_{p, 2 ; [T - \varepsilon, T]}^2 + \|
  Y^{\Delta}_T \|_{L^\infty}^2 +\| \bar{Y}^{\Delta} \|_{p, 2 ; [T - \varepsilon,
  T]}^2 + \| \bar{Y}^{\Delta}_T \|^2_{\infty}) \| W \|_{q ; [0, T]}. \nonumber
\end{align}
Applying Lemma \ref{p2Norm-property} and use the estimate \eqref{J-estimate} we get 
\begin{align}\label{JY2Con}
    \mathbb{E}_t \bigg[ \sum_{t \leq r < T} \bar{Y}_{r +}^{\Delta} \mathcal{J}_r \bigg ] &\lesssim (\|\bar{Y}^{\Delta} \|_{p, 2 ; [T - \varepsilon, T]} + \|
  \bar{Y}^{\Delta}_T \|_{L^\infty}) (\|Y^{\Delta} \|_{p, 2 ; [T - \varepsilon, T]} + \| Y^{\Delta}_T \|_{L^\infty})
    e^{c \bar{\varepsilon}} \|W\|^2_{2 ; [t, T]} \\
    &\lesssim (\|\bar{Y}^{\Delta} \|_{p, 2 ; [T - \varepsilon, T]}^2 + \|
  \bar{Y}^{\Delta}_T \|_{L^\infty}^2 + \|Y^{\Delta} \|_{p, 2 ; [T - \varepsilon, T]}^2 + \| Y^{\Delta}_T \|_{L^\infty}^2)
    e^{c \bar{\varepsilon}} \|W\|^2_{2 ; [t, T]}, \nonumber
\end{align}
here we again used $\tmop{ab}
\lesssim a^2 + b^2$ for the second estimate. \\
  Using Taylor approximation for $u \mapsto \varphi (-g_r \Delta W_r, Y_{r
  +}, u) - \varphi (-g_r \Delta W_r, Y_{r +}, u)$, Lipschitz
  continuity of $g$ in the first and Gronwall in the second inequality, we get
  \begin{align}
    \lefteqn{ \sum_{t \leq r < T} | \nobracket \varphi (g \Delta W_r, Y_{r +}) -
    \varphi (g \Delta W_r, Y'_{r +}) - Y_{r +} + Y'_{r +} | \nobracket^2 
    \label{Jump-contraction-Y2}
    }\\
    \leq & C_g \sum_{t \leq r < T} | \Delta W_r |^2 | \nobracket
    \varphi (g \Delta W_r, Y_{r +}, \theta) - \varphi (g \Delta W_r, Y'_{r +},
    \theta) |^2 \nobracket \nonumber\\
    \leq & C_g \sum_{t \leq r < T} | \Delta W_r |^2  |Y^{\Delta}_{r +}
    |^2 e^{2 C_g  | \Delta W_r |} 
    \lesssim 
(\|Y^{\Delta} \|^2_{p ; [t, T]} + | Y^{\Delta}_T |^2) e^{c
    \|W\|_{q ; [T - \varepsilon, T]}} \|W\|^2_{2 ; [t, T]} \nonumber
  \end{align}
   for some $\theta \in [0, 1]$.
  It then follows, using Remark \ref{ItoJump}, that
  \begin{align}
    & \sum_{t \leq r < T} [| \bar{Y}_{r +}^{\Delta} |^2 - |
    \bar{Y}_r^{\Delta} |^2 - 2 \Delta \bar{Y}^{\Delta}_r (\bar{Y}^{\Delta}_r)]
    \label{JumpY2Con}\\
    \leq & 2 \sum_{t \leq r < T} | \nobracket \varphi (g \Delta W_r, Y_{r +})
    - \varphi (g \Delta W_r, Y'_{r +}) - Y_{r +} + Y'_{r +} | \nobracket^2
    \nonumber\\
    \lesssim & (\|Y^{\Delta} \|^2_{p ; [t, T]} + | Y^{\Delta}_T |^2) e^{c
    \|W\|_{q ; [T - \varepsilon, T]}} \|W\|^2_{2 ; [t, T]}. \nonumber 
  \end{align}
  By combining the estimates (\ref{fY2Con}-\ref{JumpY2Con}) and Corollary \ref{integral_q-var}, the inequality (\ref{Y2Con}) becomes
  \begin{align*}
    \mathbb{E}_t \bigg[ \int_t^T | \bar{Z}^{\Delta}_r |^2 \tmop{dc}_r \bigg]
    \leq & \| \bar{Y}^{\Delta}_T \|_{L^\infty}^2 + c \lambda \|Z^{\Delta}
    \|^2_{ \tmop{BMO} ; [T - \varepsilon, T]} + c \lambda \bar{\varepsilon}  (\|
    Y^{\Delta} \|^2_{p, 2 ; [T - \varepsilon, T]} + \| Y^{\Delta}_T
    \|_{L^\infty}^2)\\
    & + c (R  \bar{\varepsilon} +  e^{c \bar{\varepsilon}}  \bar{\varepsilon}^2)   (\| Y^{\Delta} \|^2_{p, 2 ; [T - \varepsilon,
    T]} + \| Y^{\Delta}_T \|_{L^\infty}^2 +\| \bar{Y}^{\Delta} \|_{p, 2 ; [T -
    \varepsilon, T]}^2 + \| \bar{Y}^{\Delta}_T \|^2_{\infty}) \\
    & + c (\| Y^{\Delta} \|^2_{p, 2 ; [T - \varepsilon, T]} + \| Y^{\Delta}_T
    \|_{L^\infty}^2) e^{c \bar{\varepsilon}}  \bar{\varepsilon}^2 + c
    \frac{\bar{\varepsilon}}{\lambda} (\| \bar{Y}^{\Delta} \|^2_{p ; [t, T]} +
    \| \bar{Y}^{\Delta}_T \|_{L^\infty}^2) .
  \end{align*}
  Next, by taking $ \sup_t \esssup_{\omega}(\cdot)^{\frac{1}{2}}$ on both sides, we obtain
  \begin{eqnarray}
    \| \bar{Z}^{\Delta} \|_{\tmop{BMO} ; [T - \varepsilon, T]} & \leq & \|
    \bar{Y}^{\Delta}_T \|_{L^\infty} + c \bigg(
    \frac{\bar{\varepsilon}}{\lambda} + R \bar{\varepsilon} + e^{c \bar{\varepsilon}}  \bar{\varepsilon}^2 \bigg)^{\frac{1}{2}}
    \| \bar{Y}^{\Delta} \|_{p, 2 ; [T - \varepsilon, T]} \label{ZCon-2}\\
    &  & + c (\lambda \varepsilon + R \bar{\varepsilon} + e^{c \bar{\varepsilon}} 
    \bar{\varepsilon}^2)^{\frac{1}{2}}  \|Y^{\Delta} \|_{p, 2 ; [T - \varepsilon,
    T]} + c \lambda^{\frac{1}{2}} \|Z^{\Delta} \|_{\tmop{BMO} ; [T - \varepsilon,
    T]}  \nonumber\\
    &  & + c \bigg( \frac{\bar{\varepsilon}}{\lambda} + R \bar{\varepsilon} + e^{c \bar{\varepsilon}}  \bar{\varepsilon}^2
    \bigg)^{\frac{1}{2}} \| \bar{Y}^{\Delta}_T \|_{L^\infty}  + c (\lambda
    \bar{\varepsilon} + R \bar{\varepsilon} + e^{c \bar{\varepsilon}} 
    \bar{\varepsilon}^2)^{\frac{1}{2}} \| Y^{\Delta}_T \|_{L^\infty}.  \nonumber
  \end{eqnarray}
  Combining (\ref{YCon},\ref{ZCon-2}) yields for some function $F$, which is right continuous at $0$, that
  \begin{align}
    \| \bar{Z}^{\Delta} \|_{\tmop{BMO} ; [T - \varepsilon, T]} \leq & \|
    \bar{Y}^{\Delta}_T \|_{L^\infty} + F (\bar{\varepsilon}) (\lambda +
    R)^{\frac{1}{2}} \|Y^{\Delta} \|_{p, 2 ; [T - \varepsilon, T]} 
    \label{ZBMOCon}\\
    & + \bigg( F (\bar{\varepsilon}) + c \lambda^{\frac{1}{2}} \bigg)
    \|Z^{\Delta} \|_{\tmop{BMO} ; [T - \varepsilon, T]} + F (\bar{\varepsilon})
    (\frac{1}{\lambda} + R)^{\frac{1}{2}} \| \bar{Z}^{\Delta} \|_{\tmop{BMO} ;
    [T - \varepsilon, T]} \nonumber\\
    & + F (\bar{\varepsilon}) (\frac{1}{\lambda} + R)^{\frac{1}{2}} \|
    \bar{Y}^{\Delta}_T \|_{L^\infty}  + F (\bar{\varepsilon}) (\lambda +
    R)^{\frac{1}{2}} \| Y^{\Delta}_T \|_{L^\infty}  . \nonumber
  \end{align}
  Now we first choose $\lambda$ small enough such that $c
  \lambda^{\frac{1}{2}} \leq \frac{1}{6}$ and then $\varepsilon$ small enough
  such that $F (\bar{\varepsilon}) (\frac{1}{\lambda} \vee \lambda + R)^{\frac{1}{2}} \leq
  \frac{1}{6}$ and $F (\bar{\varepsilon}) \leq \frac{1}{6}$, then we have
  (recall that $Y^{\Delta}_T = 0$)
  \[ \| \bar{Z}^{\Delta} \|_{\tmop{BMO} ; [T - \varepsilon, T]} \leq \frac{2}{5}
     \|Y^{\Delta} \|_{p, 2 ; [T - \varepsilon, T]} + \frac{2}{5} \|Z^{\Delta}
     \|_{\tmop{BMO} ; [T - \varepsilon, T]} . \]
  Then, by substituting (\ref{ZCon-2}) into (\ref{YCon}), one obtains
  \begin{align}
    & \| \bar{Y}^{\Delta} \|_{p, 2 ; [T - \varepsilon, T]} 
    \; \leq\;  C_p \|
    \bar{Y}^{\Delta}_T \|_{L^\infty} + F (\bar{\varepsilon}) R \|Y^{\Delta} \|_{p,
    2 ; [T - \varepsilon, T]} + F (\bar{\varepsilon}) R \| \bar{Y}^{\Delta}
    \|_{p, 2 ; [T - \varepsilon, T]}  \label{Yp2Con}\\
    &\qquad  
    + \bigg( F (\bar{\varepsilon}) R + \lambda^{\frac{1}{2}} \bigg) 
    \|Z^{\Delta} \|_{\tmop{BMO} ; [T - \varepsilon, T]} + F (\bar{\varepsilon})
    (\lambda + R) \| Y^{\Delta}_T \|_{L^\infty} \nonumber 
    + F (\bar{\varepsilon}) \bigg(
    \frac{1}{\lambda} + R \bigg)  \| \bar{Y}^{\Delta}_T \|_{L^\infty}  .
    \nonumber
  \end{align}
  and can again by choosing $\lambda$ and $\bar{\varepsilon}$ carefully (and
  recalling $\bar{Y}^{\Delta}_T = Y^{\Delta}_T = 0$), we have
  \begin{eqnarray*}
    \| \bar{Y}^{\Delta} \|_{p, 2 ; [T - \varepsilon, T]} & \leq & \frac{2}{5} 
    \|Y^{\Delta} \|_{p, 2 ; [T - \varepsilon, T]} + \frac{2}{5} \|Z^{\Delta}
    \|_{\tmop{BMO} ; [T - \varepsilon, T]} .
  \end{eqnarray*}
  Overall, we thus have
  \[ \| \bar{Y}^{\Delta} \|_{p, 2 ; [T - \varepsilon, T]} + \| \bar{Z}^{\Delta}
     \|_{\tmop{BMO} ; [T - \varepsilon, T]} \leq \frac{4}{5}  (\|Y^{\Delta}
     \|_{p, 2 ; [T - \varepsilon, T]} +\|Z^{\Delta} \|_{\tmop{BMO} ; [T -
     \varepsilon, T]}) . \]
  Therefore $\mathcal{M}^{T, \varepsilon}$ admits a unique fixpoint $(Y, Z) \in
  B^{T, \varepsilon}_R$, which is the unique solution of the Marcus-RBSDE
  (\ref{rBSDE-Marcus}) on the interval $[T - \varepsilon, T]$.
\end{proof}

\begin{remark}
  \label{epsilon-dependency}The choice of $R$ in the invariance part of the
  proof, see \eqref{YInv-final}, depends only on $C_f, C_g, p$ and the norm of
  the terminal condition $\| \xi  \|_{L^\infty}$. As for $\bar{\varepsilon}$, it is
  chosen to be small enough depending only on $C_f, C_g, p$ such that the
  invariance property holds. They have to be chosen even smaller for the
  contraction property to hold and there their choice also depends on $R$, and
  therefore directly on $\| \xi  \|_{L^\infty}$. One can further bound $\| \xi
  \|_{L^\infty}$ by the apriori bound $L_{\tmop{apriori}}$ from Theorem \ref{apriori-bound-YZ}, and obtain new quantities $R_L \geq R$ and
  $\bar{\varepsilon}_L \leq \bar{\varepsilon}$, which are still going to play crucial roles in proving of the next theorem. Finally, the
  choice of $\varepsilon$ is determined by condition \eqref{barepsilon},
 showing that it depends on $\bar{\varepsilon}, c$ and $W$. 
\end{remark}

The general idea of proving existence and uniqueness of the RBSDE solution on
the whole interval $[0, T]$ is to find a finite time partition $\pi = \{0 =
t_0 < t_1 < \cdots < t_N = T\}$ with small meshsize $| \pi | \leq
\varepsilon$, so that we can apply Theorem \ref{existenceuniqueness} to solve
the RBSDE on the intervals $[t_i, t_{i + 1}]$ and ``glue'' the local solutions
together. There are two obstacles where this approach faces problems. The first one is
that, at least under Assumption A, the choice of $\bar{\varepsilon}$ and
therefore of $\varepsilon$ depend on norm of terminal condition $\| Y_{t_i}
\|_{L^\infty}$, which now differs for every interval. Fortunately, whenever
we obtain a new terminal condition $Y_{t_i}$ from solving the RBSDE on $[t_i,
t_{i + 1}]$, it is automatically bounded by the apriori bound
$L_{\tmop{apriori}}$ and we could work with the quantity $\bar{\varepsilon}_L$
defined in Remark \ref{epsilon-dependency}, which serves the same puropose as
$\bar{\varepsilon}$, but is additionally uniform for all terminal conditions
obtained in the above form. The second challenge occurs from large jumps of the rough driver $W$ at times $\tau \in [0, T]$, since regardless of the time
partition we choose, there will always be an interval $[t_i, t_{i + 1}]$
containing such a $\tau$, thus causing  $\| W \|_{q ; [t_i, t_{i + 1}]}$ to be too
large for the condition \eqref{barepsilon} to hold,  presenting direct application of Theorem
\ref{existenceuniqueness}  on $[t_i, t_{i + 1}]$. Fortunately, since there are only finitely many jumps larger than
$\bar{\varepsilon}>0$, one can ``take them out'' of the local fix point construction, but instead define the solution at such large jumps 
``by hand''.
This sketch of ideas is elaborated in the proof of the next theorem.

\begin{theorem}[Global existence and uniqueness]
  \label{global_existenceuniqueness}Provided that Assumption A holds, there
  exists a unique $(Y, Z) \in \mathcal{B}_p \times \tmop{BMO}$ that satisfies
  the integral equation (\ref{rBSDE-Marcus}) (resp. (\ref{rBSDE-forward})) on
  the whole interval $[0, T]$.
\end{theorem}

\begin{proof}
  Recall the definition of $\bar{\varepsilon}_L$ from Remark
  \ref{epsilon-dependency}. \ By Lemma 4.7 and 4.8 in
  {\cite{r_m_dudley_introduction_nodate}}, there exists a finite time
  partition $\pi = \{0 = t_0 < t_1 < \cdots < t_N = T\}$ such that 
  \begin{align*}
      \max_{i =
  1, \cdots, N} | c_{t_{i - 1}, t_i} |  \leq \bar{\varepsilon}_L, &\max_{i =
  1, \cdots, N} \|W\|_{q ; (t_{i - 1}, t_i]} \leq \bar{\varepsilon}_L, &N \leq 1 + \max \{ | c_T |, \|W\|_{q ; [0, T]} \} / \bar{\varepsilon}_L.
  \end{align*}  
  By Theorem \ref{existenceuniqueness}, for every $t>t_{N-1}$ there exists a
unique solution $(Y,Z)$ to the RBSDE on $[t,T]$. Moreover, by uniqueness, these
solutions agree on overlapping intervals. Hence they define a unique solution
$(Y,Z)$ on $(t_{N-1},T]$. It remains to define the value of the solution at $t_{N-1}$. For Marcus jumps,
we set
\[
Y_{t_{N-1}} \assign
\varphi \bigl(g_{t_{N-1}} \Delta W_{t_{N-1}}, Y_{t_{N-1}+}\bigr),
\qquad
Z_{t_{N-1}} \assign Z_{t_{N-1}+}.
\]
For forward jumps, we instead set
\[
Y_{t_{N-1}} \assign
Y_{t_{N-1}+}
+ g_{t_{N-1}}(Y_{t_{N-1}+}) \Delta W_{t_{N-1}},
\qquad
Z_{t_{N-1}} \assign Z_{t_{N-1}+}.
\]
In both cases, this defines the value of $(Y,Z)$ immediately after the jump in
reverse time. It is straightforward to verify that we
  have thereby constructed an unique solution of the RBSDE on $[t_{N - 1}, T]$. By
  the apriori bounds (from Theorem \ref{apriori-bound-YZ}), we have that $\| Y_{t_{N -
  1}} \|_{L^\infty} \leq L_{\tmop{apriori}}$, and one sees that we next can apply the 
  Theorem \ref{existenceuniqueness} again on $(t_{N - 2}, t_{N - 1}]$.
  Iterating the above contruction (for $n=N-1,N-2,\ldots$) until reaching $t_{n-1}=0$,  yields a global solution to
  the rough BSDE. The uniqueness follows from the local uniqueness at each step.
\end{proof}

Let $(Y^n,Z^n)_{n\in\mathbb N}$ denote the (global) Picard iteration associated with
the Marcus-RBSDE \eqref{rBSDE-Marcus} or the Forward-RBSDE
\eqref{rBSDE-forward}, defined recursively as follows:
\begin{enumeratenumeric}
  \item $Y^0 \equiv 0$, $Z^0 \equiv 0$;
  \item given $(Y^n,Z^n)$, define
  \[
  Y^{n+1}_t
  \assign
  \mathbb E_t \bigg[
  \xi
  +
  \int_t^T f(r,Y^n_r,Z^n_r)\,dc_r
  +
  \int_t^T g_r(Y^n_{r+})\,(\diamond)\,dW_r
  \bigg],
  \]
  and let $Z^{n+1}$ be given through the martingale representation
  \[
\int_0^t Z^{n+1}_r\,dM_r
=
\mathbb E_t\bigg[
\xi+\int_0^T f(r,Y^n_r,Z^n_r)\,dc_r
+\int_0^T g_r(Y^n_{r+})\,(\diamond)\,dW_r
\bigg] .
\]
\end{enumeratenumeric}
It is straightforward to verify that $(Y^{n+1},Z^{n+1})$ defined above
satisfies
\[
Y^{n+1}_t
=
\xi
+
\int_t^T f(r,Y^n_r,Z^n_r)\,\mathrm{d}c_r
+
\int_t^T g_r(Y^n_{r+})\,(\diamond)\,\mathrm{d}W_r
-
\int_t^T Z^{n+1}_r\,\mathrm{d}M_r,
\qquad t\in[0,T],\ n\in\mathbb N .
\]
In the next theorem, we show that $(Y^n,Z^n)_{n\in\mathbb N}$ is a Cauchy
sequence in $\mathcal{B}_p\times\mathrm{BMO}$. Combined with the uniqueness
of the global solution $(Y,Z)$ established in Theorem
\ref{global_existenceuniqueness}, this implies that
$(Y^n,Z^n)\to(Y,Z)$ in $\mathcal{B}_p\times\mathrm{BMO}$ as $n\to\infty$. This result provides a constructive way of obtaining global solutions to
RBSDEs. Its usefulness also shows in its crucial role in
the proof of Theorem \ref{stability-RBSDE-decorated} and Theorem
\ref{measurable-RBSDE-solution}. We like to  emphasize that showing the convergence of the global Picard iterations is non-trivial. Indeed, by the (local) contraction proof of Theorem
\ref{existenceuniqueness}, we concluded that, on a sufficiently small time interval $[s, t]$, it holds
\begin{align*}
    & \interleave Y^{n + 1} - Y^n, Z^{n + 1} - Z^n \interleave_{[s, t]} 
    \\
    & \qquad \leq  C
\interleave Y^n - Y^{n - 1}, Z^n - Z^{n - 1} \interleave_{[s, t]} + (C + \eta)
\|Y^{n + 1}_t - Y^n_t \|_{L^\infty} 
+ C \| Y^n_t - Y^{n - 1}_t \|_{L^\infty}
\end{align*}
for some constants $C < 1$ and $C + \eta > 1$. Since the terms $\|Y^{n + 1}_t - Y^n_t
\|_{L^\infty}$ and $\| Y^n_t - Y^{n - 1}_t \|_{L^\infty}$ are in general in \emph{not}
zero except for $t = T$ (as in Theorem \ref{existenceuniqueness}), this explains why 
that proof there does \emph{not} establish a global contraction property. In fact, we do not believe that a global contraction holds in general. Instead, we only show $(Y^n,Z^n)$ to be a Cauchy sequence in the subsequent theorem. Here we can use arguments similar to those used for Theorems \ref{existenceuniqueness}
and \ref{global_existenceuniqueness}, and we are able to build on several estimates from there.

\begin{theorem}
  \label{Picard-iteration}Provided that Assumption A holds, then for the 
  Marcus-RBSDE \eqref{rBSDE-Marcus} (or, respectively, for the Forward-RBSDE (\ref{rBSDE-forward})), the Picard-iterates $(Y^n,Z^n)_{n \in \NN}$, as defined above, form a Cauchy sequence in $\mathcal{B}_p \times \tmop{BMO}$. Its limit is the (respective) global RBSDE solution $(Y,Z)$ from Theorem~\ref{global_existenceuniqueness}, in other words it holds
  \[ \lim_{n \rightarrow \infty} (\| Y^n - Y \|_{p, 2 ; [0, T]} + \| Z^n - Z
     \|_{\tmop{BMO} ; [0, T]}) = 0. \]
\end{theorem}

\begin{proof}
  We show that there exists an $R > 0$ such that for all $n \in
  \mathbb{N}$ we have the bounds
  \begin{equation}
    \interleave Y^n, Z^n \interleave_{[0, T]} \assign \|Z^n \|_{\tmop{BMO} ;
    [0, T]} + \|Y^n \|_{p, 2 ; [0, T]} \leq R \label{Picard-Inv} .
  \end{equation}
  Furthermore, there exists some $\varepsilon > 0$ such that for any partition
  $\Pi = \{0 = t_0 < t_1 < \cdots < t_N = T\}$ satisfying\footnote{The existence of such partitions is provided by Lemma 4.7 and Lemma 4.8 in
  {\cite{r_m_dudley_introduction_nodate}}.}
  \begin{eqnarray}
    \max_{j = 1, \cdots, N} |c_{t_{i - 1}, t_i} | \leq \varepsilon, \; \max_{j
    = 1, \cdots, N} \|W\|_{q ; (t_{i - 1}, t_i]} \leq \varepsilon, \; N
    \leq 1 + \max \{|c_T |, \|W\|_{q ; [0, T]} \} / \varepsilon, 
    \label{partition-pi-bar}
  \end{eqnarray}
  we have for some $C < 1$ and constants $\kappa, \eta > 0$ (see
  (\ref{DefEta}) and (\ref{DefKi}) for details) that  
  \begin{equation}
    \interleave Y^{n, n + 1}, Z^{n, n + 1} \interleave_{[t_{N - j}, T]}
    \leq 2^{(p + 2) (j - 1)} C^{n - j}  (C + \kappa)^j  (C + \eta)^{j -
    1} n^{j - 1} \interleave Y^1, Z^1 \interleave_{[t_{N - j}, T]},
    \label{Picard-contraction}
  \end{equation}
  holds for all $n \in \mathbb{N}$ and $j =
  1, \cdots, N$,
  where $Y^{n, n + 1}$ and \ $Z^{n, n + 1}$ denote the differences $Y^{n + 1} -
  Y^n$ and $Z^{n + 1} - Z^n$, respectively.
  
  For better readability, we postpone proving (\ref{Picard-Inv}) and
  (\ref{Picard-contraction}) to continue at first by completing the main line of proof of the theorem:
  For $j = N$ we have for all $n \in \mathbb{N}$
  \begin{equation}
    \interleave Y^{n, n + 1}, Z^{n, n + 1} \interleave_{[0, T]} \leq C_N
    C^n n^{N - 1} \interleave Y^1, Z^1 \interleave_{[0, T]},
    \label{Picard-estimates}
  \end{equation}
  with $C_N \assign 2^{(p + 2) (N - 1)}  (C + \kappa)^N  (C + \eta)^{N - 1}
  C^{- N}$. We get the boundedness of $\interleave Y^1, Z^1 \interleave_{[0,
  T]}$ directly from (\ref{Picard-Inv}).
  Hence, for every $n \in \mathbb{N}$ we have
  \begin{equation}
    \interleave Y - Y^n, Z - Z^n \interleave_{[0, T]} \leq \interleave
    Y^1, Z^1 \interleave_{[0, T]} C_N  \sum_{l = n}^{\infty} l^{N - 1} C^l .
    \label{Picard-bound}
  \end{equation}
  It is straightforward to show that $\sum_{l = 0}^{\infty} l^N C^l < \infty$ is finite, implying
  that
  \[ \lim_{n \rightarrow \infty} \|Y^n - Y\|_{p, 2 ; [0, T]} + \|Z^n -
     Z\|_{\tmop{BMO} ; [0, T]} = 0. \]
  \tmtextbf{Proof of inequality} \eqref{Picard-Inv}: By the same
  calculation as in (\ref{ZInv-abstact}-\ref{YInv-abstact}), we have that there exists
  some $\bar{\varepsilon} > 0$, such that for any time partition $\pi = \{0 =
  s_0 < s_1 < \cdots < s_M = T\}$ satisfying
  \begin{eqnarray}
    \max_{i = 1, \cdots, M} |c_{s_{i - 1}, s_i} | \leq \bar{\varepsilon}, &
    \max_{i = 1, \cdots, M} \|W\|_{q ; (s_{i - 1}, s_i]} \leq
    \bar{\varepsilon}, & M \leq 1 + \max \{|c_T |, \|W\|_{q ; [0, T]} \}
    / \bar{\varepsilon}, \nonumber
  \end{eqnarray}
  we have that for all $i = 0, \ldots, M - 1$ and $n \in \mathbb{N}$ it holds
  \begin{align*}
    \|Z^n \|_{\tmop{BMO} ; (s_{i - 1}, s_i]} \leq & \frac{2}{m}
    \interleave Y^{n - 1}, Z^{n - 1} \interleave_{(s_{i - 1}, s_i]} +
    \frac{1}{m} \|Y^{n - 1}_{s_i}  \|_{L^\infty} + \bigg( \frac{1}{m} + 1
    \bigg) \| Y^n_{s_i} \|_{L^\infty} \\ &+ \frac{1}{m} + \frac{1}{m} \|Y^n \|_{p,
    2 ; (s_{i - 1}, s_i]}, \\
    \|Y^n \|_{p, 2 ; (s_{i - 1}, s_i]} \leq & \frac{2}{m} \interleave
    Y^{n - 1}, Z^{n - 1} \interleave_{(s_{i - 1}, s_i]} + \frac{1}{m}  \|Y^{n
    - 1}_{s_i} \|_{L^\infty} + \bigg( \frac{1}{m} + C_y \bigg) \|Y^n_{s_i}
    \|_{L^\infty} + \frac{1}{m} \\&+ \frac{1}{m} \|Y^n \|_{p, 2 ; (s_{i - 1}, s_i]}
    . 
  \end{align*}
  Notice that the choice of $\bar{\varepsilon}$ is specified by the
  conditions (\ref{condition-barepsilon}). We observe, that it depends only on $C_y$
  (and $C_z$), which in turn depends only on $C_f, C_g, p$ and is invariant for
  different intervals and different $n \in \mathbb{N}$. At any time $t \in [0,
  T]$ and for any $n \in \mathbb{N}$, we have for both the forward jump
  $\Delta Y^{n + 1}_{s_i} = - g_{s_i} (Y^n_{s_i +}) \Delta W_{s_i}$ and Marcus
  jumps $\Delta Y^{n + 1}_{s_i} = - (\varphi (g_{s_i} \Delta W_{s_i}, Y^n_{s_i
  +}) - Y^n_{s_i +})$ the uniform bound $\| \Delta Y^{n + 1}_{s_i} \|_{L^\infty}
  \leq C_g \|W\|_{q ; [0, T]}$, here we use Taylor formula for estimating
  the Marcus jumps. This gives us the estimate $\|Y^n \|_{p, 2 ; [s_{i - 1},
  s_i]} \leq \|Y^n \|_{p, 2 ; (s_{i - 1}, s_i]} + C_g \|W\|_{q ; [0,
  T]}$, combining this with the above two inequalities and using the
  continuity of the $\tmop{BMO}$ norm we get
  \begin{align}
    \frac{m - 2}{m} \interleave Y^n, Z^n \interleave_{[s_{i - 1}, s_i]}
    \leq & \frac{4}{m} \interleave Y^{n - 1}, Z^{n - 1}
    \interleave_{(s_{i - 1}, s_i]} + \frac{2}{m}  \|Y^{n - 1}_{s_i}
    \|_{L^\infty} \nonumber \\
    & + \bigg( \frac{2}{m} + 1 + C_y \bigg) \|Y^n_{s_i} \|_{L^\infty} + \frac{2}{m} + C_g \|W\|_{q ; [0, T]} . \nonumber
  \end{align}
  Given some $m \in \mathbb{N}$, to  be specified later, we fix a
  value
  \begin{equation}
    R \assign mM^{p-1}  (4 C_y)^{M + 1}  (1 \vee \| \xi  \|_{L^\infty} \vee C_g
    \|W\|_{q ; [0, T]}) \label{DefR}
  \end{equation}
  and show by induction over $i$ and $n$ that
  \[ \interleave Y^n, Z^n \interleave_{[s_{i - 1}, s_i]} \leq M^{1 - p} 
     (4 C_y)^{- i} R. \]
  We begin by showing it for the cases $n = 0, i = 1, \ldots, M$ and $i = M, n
  \in \mathbb{N}$. The first case with $Y^0 = 0, Z^0 = 0$ being trivial, we
  just show the second case by induction over $n$. Assume $\interleave Y^{n-1},
  Z^{n-1} \interleave_{[t_{M - 1}, t_M]} \leq M^{1 - p}  (4 C_y)^{- M} R$
  holds for some $n-1$, we show for $n$ that
  \begin{align*}
    & \frac{m - 2}{m} \interleave Y^n, Z^n \interleave_{[t_{M - 1}, t_M]} \\
    \leq & \frac{4}{m} \interleave Y^{n - 1}, Z^{n - 1}
    \interleave_{(t_{M - 1}, t_M]} + \bigg( \frac{4}{m} + 1 + C_y \bigg) \|
    \xi  \|_{L^\infty} + \frac{2}{m} + C_g \|W\|_{q ; [0, T]}\\
    \leq & \frac{4}{m} M^{1 - p}  (4 C_y)^{- M} R + \frac{1}{m}  \bigg(
    \frac{4}{m} + 1 + C_y \bigg) M^{1 - p}  (4 C_y)^{- M - 1} R\\
    & + \frac{2}{m^2} M^{1 - p}  (4 C_y)^{- M - 1} R + \frac{1}{m} M^{1 - p} 
    (4 C_y)^{- M - 1} R,
  \end{align*}
  implying $\interleave Y^n, Z^n \interleave_{[t_{M - 1}, t_M]} \leq
  \frac{5}{m - 2} M^{1 - p}  (4 C_y)^{- M} R$. So for $m \geq 7$ we get the desired result.
  Next, assuming that for some $n$ and $i$ we have $\interleave Y^n, Z^n
  \interleave_{[s_j, s_{j + 1}]} \leq M^{1 - p}  (4 C_y)^{- j - 1} R$ for
  all $j = i, \ldots, M - 1$ and $\interleave Y^{n - 1}, Z^{n - 1}
  \interleave_{[s_{j - 1}, s_j]} \leq M^{1 - p}  (4 C_y)^{- j} R$ for all
  $j = i, \ldots, M$, we show that $\interleave Y^n, Z^n \interleave_{[s_{i -
  1}, s_i]} \leq M^{1 - p}  (4 C_y)^{- i} R$ follows. We have
  \begin{align*}
    & \frac{m - 2}{m} \interleave Y^n, Z^n \interleave_{[s_{i - 1}, s_i]}\\
    \leq & \frac{4}{m} \interleave Y^{n - 1}, Z^{n - 1}
    \interleave_{(s_{i - 1}, s_i]} + \frac{2}{m} \|Y^{n - 1}_{s_i} 
    \|_{L^\infty} + (\frac{2}{m} + 1 + C_y)\| Y^n_{s_i} \|_{L^\infty} +
    \frac{2}{m} + C_g \|W\|_{q ; [0, T]}\\
    \leq & \frac{4}{m} \interleave Y^{n - 1}, Z^{n - 1}
    \interleave_{(s_{i - 1}, s_i]} + \frac{2}{m}  \bigg( \sum_{j = i}^{M - 1}
    \|Y^{n - 1} \|_{p, 2 ; [s_j, s_{j + 1}]} +\| \xi  \|_{L^\infty} \bigg)\\
    & + (\frac{2}{m} + 1 + C_y)  \bigg( \sum_{j = i}^{M - 1} \|Y^n \|_{p, 2 ;
    [s_j, s_{j + 1}]} +\| \xi  \|_{L^\infty} \bigg) + \frac{2}{m} + C_g \|W\|_{q
    ; [0, T]}.
  \end{align*}
  By the induction assumption, it holds
  \[ \sum_{j = i}^{M - 1} \|Y^{n - 1} \|_{p, 2 ; [s_j, s_{j + 1}]} \leq
     RM^{1 - p}  \sum_{j = i}^{M - 1} (4 C_y)^{- j - 1} < RM^{1 - p}  (4
     C_y)^{- i - 1}  \sum_{j = 0}^{\infty} (4 C_y)^{- j} \]
  Without loss of generality, let $C_y > 1$. Using geometric series limits, we obtain
  \[ \sum_{j = i}^{M - 1} \|Y^{n - 1} \|_{p, 2 ; [s_j, s_{j + 1}]} \leq
     \frac{RM^{1 - p}  (4 C_y)^{- i - 1} }{1 - (4 C_y)^{- 1}} \leq
     \frac{4}{3} RM^{1 - p}  (4 C_y)^{- i-1} = \frac{1}{3 C_y} R M^{1 - p}  (4
     C_y)^{- i} \]
  and by the same estimation we also have $\sum_{j = i}^{M - 1} \|Y^n \|_{p, 2
  ; [s_j, s_{j + 1}]} \leq \frac{1}{3 C_y} R M^{1 - p}  (4 C_y)^{- i}$.
  Together with the \eqref{DefR}, we conclude
  \[ \frac{m - 2}{m} \interleave Y^n, Z^n \interleave_{[s_{i - 1}, s_i]}
     \leq \bigg( \frac{5}{m} + \frac{1}{3 C_y} + \frac{1}{3} \bigg) M^{1
     - p}  (4 C_y)^{- i} R, \]
  implying $\interleave Y^n, Z^n \interleave_{[s_{i - 1}, s_i]} \leq M^{1
  - p}  (4 C_y)^{- i} R$ for any $m \geq 21$. This completes the
  induction argument. Applying Lemma \ref{p2Norm-property} and the above in particular implies that
  \[ \interleave Y^n, Z^n \interleave_{[0, T]} \leq M^{p - 1}  \sum_{i =
     1}^M \interleave Y^n, Z^n \interleave_{[s_{i - 1}, s_i]} \leq
     \frac{(4 C_y)^{- 1}}{1 - (4 C_y)^{- 1}} R \leq \frac{1}{3} R. \]
  \tmtextbf{Proof of the inequalities in} \eqref{Picard-contraction}: 
  By exactly the same calculation as (\ref{ZBMOCon}-\ref{Yp2Con}), we get for
  all $n \in \mathbb{N}$ and $i = 0, \ldots, N - 1$, that
  \begin{align*}
    \|Z^{n, n + 1} \|_{\tmop{BMO} ; (t_i, t_{i + 1}]} \leq & \|Y^{n, n +
    1}_{t_{i + 1}} \|_{L^\infty} + F (\varepsilon)  (\frac{1}{\lambda} +
    R)^{\frac{1}{2}} \|Y^{n - 1, n} \|_{p, 2 ; (t_i, t_{i + 1}]}\\
    & + (F (\varepsilon) + c \lambda^{\frac{1}{2}}) \|Z^{n - 1, n}
    \|_{\tmop{BMO} ; (t_i, t_{i + 1}]} + F (\varepsilon) (\frac{1}{\lambda} +
    R)^{\frac{1}{2}} \| Y^{n, n + 1}_{t_{i + 1}} \|_{L^\infty} \\
    & + F (\varepsilon)  (\lambda + R)^{\frac{1}{2}} \|Y^{n - 1, n}_{t_{i + 1}}
    \|_{L^\infty} + F (\varepsilon)  (\frac{1}{\lambda} + R)^{\frac{1}{2}}
    \|Z^{n, n + 1} \|_{\tmop{BMO} ; (t_i, t_{i + 1}]},
  \end{align*}
  \begin{align}
    \|Y^{n, n + 1} \|_{p, 2 ; (t_i, t_{i + 1}]} \leq & C_p \|Y^{n, n +
    1}_{t_{i + 1}} \|_{L^\infty} + F (\varepsilon) R\|Y^{n - 1, n} \|_{p, 2 ;
    (t_i, t_{i + 1}]}  \nonumber\\
    & + F (\varepsilon) R \|Y^{n, n + 1} \|_{p, 2 ; (t_i, t_{i
    + 1}]} + (F (\varepsilon) R + \lambda^{\frac{1}{2}}) \|Z^{n - 1, n}
    \|_{\tmop{BMO} ; (t_i, t_{i + 1}]} \nonumber\\
    & + F (\varepsilon) (\lambda + R) \|Y^{n - 1, n}_{t_{i + 1}} \|_{L^\infty}
    + F (\varepsilon) \bigg( \frac{1}{\lambda} + R \bigg)  \|Y^{n, n +
    1}_{t_{i + 1}} \|_{L^\infty}, \nonumber
  \end{align}
  where $R$ is from \eqref{DefR} and $C_p$ from \eqref{ZCon}. Similarly as we
  have argued several times before, we can again choose $\lambda$ and $\varepsilon$ to be
  suitably small (depending only $C_f, C_g, p$ and $R$) to get for some $C <
  1$ that
  \begin{align}
\label{YZCon}
    & \interleave Y^{n, n + 1}, Z^{n, n + 1} \interleave_{(t_i, t_{i + 1}]}  
    \\
    & \qquad  \le  C\|Y^{n - 1, n}_{t_{i + 1}} \|_{L^\infty} + (C + \eta) \|Y^{n, n
    + 1}_{t_{i + 1}} \|_{L^\infty}  + C \interleave Y^{n - 1, n}, Z^{n - 1, n}
    \interleave_{(t_i, t_{i + 1}]}  \nonumber
  \end{align}
  for all $n \in \mathbb{N}$ and $i = 1, \ldots, N - 1$, and some constant
  \begin{equation}
    \eta < C + 2 C_p . \label{DefEta}
  \end{equation}
  We have for forward-type integration of jumps that
  \begin{align*}
    \| \Delta Y^{n + 1}_t - \Delta Y_t^n \|_{L^\infty} & = \|(g_t (Y^n_{t +}) -
    g_{t} (Y_{t +}^{n - 1})) \Delta W_t \|_{L^\infty} \leq C_g | \Delta
    W_t |  \| Y^{n - 1, n}_{t +} \|_{L^\infty},
  \end{align*}
  while for the Marcus-type jumps it follows  that
  \begin{align*}
    \| \Delta Y^{n + 1}_t - \Delta Y_t^n \|_{L^\infty} & = \| \varphi (g_t
    \Delta W_t, Y^n_{t +}) - \varphi (g_t \Delta W_t, Y^{n - 1}_{t +}) -
    Y^n_{t +} + Y^{n - 1}_{t +} \|_{L^\infty}\\
    & \leq C_g | \Delta W_t | \exp (C_g | \Delta W_t |) \| Y^{n - 1,
    n}_{t +} \|_{L^\infty},
  \end{align*}
by  using Taylor and Gronwall arguments,  similarly as in
  \eqref{Jump-contraction-Y2}.  For both type of jumps and for any time $t$ we have $\| \Delta Y^{n + 1}_t -
  \Delta Y_t^n \|_{L^\infty} \leq \kappa (\|Y^{n - 1, n}_{t_{i + 1}}
  \|_{L^\infty} +\|Y^{n - 1, n} \|_{p, 2 ; (t_i, t_{i + 1}]})$ for
  \begin{equation}
    \kappa \assign C_g \|W\|_{q ; [0, T]} \vee C_g \exp (C_g \|W\|_{q ; [0,
    T]})  \|W\|_{q ; [0, T]} . \label{DefKi}
  \end{equation}
  We can combine the above inequality with \eqref{YZCon} to get
  \begin{align*}
    &\interleave Y^{n, n + 1}, Z^{n, n + 1} \interleave_{[t_i, t_{i + 1}]}\\
    &\qquad
    \leq  (C + \kappa) (\|Y^{n - 1, n}_{t_{i + 1}} \|_{L^\infty} +
    \interleave Y^{n - 1, n}, Z^{n - 1, n} \interleave_{(t_i, t_{i + 1}]})  
    +(C + \eta) \|Y^{n, n + 1}_{t_{i + 1}} \|_{L^\infty} 
  \end{align*}
  By repeatedly applying \eqref{YZCon} to the term $\interleave \cdot
  \interleave_{(t_i, t_{i + 1}]}$ on the right, one obtains
  \begin{align}
    & \interleave Y^{n, n + 1}, Z^{n, n + 1} \interleave_{[t_i, t_{i + 1}]} \label{YZCon-full}\\
    \leq & (C + \eta) \|Y^{n, n + 1}_{t_{i + 1}} \|_{L^\infty} + (C + \eta)
    (C + \kappa) \sum_{l = 1}^{n - 1} C^{l - 1} \|Y^{n - l, n - l + 1}_{t_{i +
    1}} \|_{L^\infty} \nonumber\\
    & + (C + \kappa) \sum_{l = 0}^{n - 1} C^l \|Y^{n - l - 1, n - l}_{t_{i +
    1}} \|_{L^\infty} + (C + \kappa) C^{n - 1} \interleave Y^1, Z^1
    \interleave_{(t_i, t_{i + 1}]} \nonumber
  \end{align}
  To show \eqref{Picard-contraction} by induction over $j = 1, \ldots, N$,  we start for $j = 1$ by noticing that $Y^{n, n + 1}_{t_N} = 0$ for all $n \in
  \mathbb{N}$, so \eqref{YZCon-full} simply becomes $\interleave Y^{n, n + 1},
  Z^{n, n + 1} \interleave_{[t_{N - 1}, T]} \leq (C + \kappa) C^{n - 1}
  \interleave Y^1, Z^1 \interleave_{(t_{N - 1}, t_N]}$.
  Assuming now that for some $j$ the inequality \eqref{Picard-contraction} holds
  for all $n \in \mathbb{N}$. To conclude that it holds for $j + 1$, notice that
  \eqref{YZCon-full} can be rewritten as
  \begin{align*}
    & \interleave Y^{n, n + 1}, Z^{n, n + 1} \interleave_{[t_{N - j - 1},
    t_{N - j}]}\\
    \leq & (C + \eta) \|Y^{n, n + 1} \|_{p, 2 ; [t_{N - j}, T]} + (C +
    \eta) (C + \kappa) \sum_{l = 1}^{n - 1} C^{l - 1} \|Y^{n - l, n - l + 1}
    \|_{p, 2 ; [t_{N - j}, T]}\\
    & + (C + \kappa) \sum_{l = 0}^{n - 1} C^l \|Y^{n - l - 1, n - l} \|_{p, 2
    ; [t_{N - j}, T]} + (C + \kappa) C^{n - 1} \interleave Y^1, Z^1
    \interleave_{(t_{N - j - 1}, t_{N - j}]} .
  \end{align*}
  By substituting according to \eqref{Picard-contraction} in the above, we obtain
  \begin{align*}
    & \interleave Y^{n, n + 1}, Z^{n, n + 1} \interleave_{[t_{N - j - 1},
    t_{N - j}]}\\
    \leq & 2^{(p + 2) (j - 1)} C^{n - j}  (C + \kappa)^j  (C + \eta)^j
    n^{j - 1} \interleave Y^1, Z^1 \interleave_{[t_{N - j}, T]}\\
    & + 2^{(p + 2) (j - 1)} C^{n - j - 1}  (C + \kappa)^{j + 1}  (C + \eta)^j
    n^j \interleave Y^1, Z^1 \interleave_{[t_{N - j}, T]}\\
    & + 2^{(p + 2) (j - 1)} C^{n - j - 1}  (C + \kappa)^{j + 1}  (C +
    \eta)^{j - 1} n^j \interleave Y^1, Z^1 \interleave_{[t_{N - j}, T]}\\
    & + (C + \kappa) C^{n - 1} \interleave Y^1, Z^1 \interleave_{(t_{N - j -
    1}, t_{N - j}]}\\
    \leq & 3 \times 2^{(p + 2) (j - 1)} C^{n - j - 1}  (C + \kappa)^{j +
    1}  (C + \eta)^j n^j \interleave Y^1, Z^1 \interleave_{[t_{N - j}, T]}\\
    & + (C + \kappa) C^{n - 1} \interleave Y^1, Z^1 \interleave_{(t_{N - j -
    1}, t_{N - j}]} .
  \end{align*}
  We use the fact $C^{n - j - 1}  (C + \kappa)^{j + 1}  (C + \eta)^j n^j >
  C^{n - j}  (C + \kappa)^j  (C + \eta)^{j - 1} n^{j - 1} > (C + \kappa) C^{n
  - 1}$ and apply Lemma \ref{p2Norm-property} twice to conclude that
  \begin{align*}
    & \interleave Y^{n, n + 1}, Z^{n, n + 1} \interleave_{[t_{N - j - 1},
    T]}\\
    \leq & 2^{p - 1}  (\interleave Y^{n, n + 1}, Z^{n, n + 1}
    \interleave_{[t_{N - j}, T]} + \interleave Y^{n, n + 1}, Z^{n, n + 1}
    \interleave_{[t_{N - j - 1}, t_{N - j}]})\\
    \leq & \nobracket 2^{(p + 2) j} C^{n - j - 1} (C + \kappa)^{j + 1} (C
    + \eta)^j n^j \interleave Y^1, Z^1 \interleave_{[t_{N - j - 1}, T]}) .
  \end{align*}
  This completes the proof.
\end{proof}

\section{Stability of Solution}\label{Chapter 4}
In this chapter, we study the stability of the solutions to the Forward-RBSDE
(\ref{Forward-RBSDE}) and the Marcus-RBSDE (\ref{Marcus-RBSDE}). Our aim is to
quantify by upper stability estimates the extend, by which small perturbations in the quantities $\xi$, $f$, $g$, and the rough driving path $W$
can affect the solution $(Y, Z)$. There are various ways to measure the
distance between c{\`a}gl{\`a}d paths, including the standard $p$-variation
metric and a Skorokhod-type $p$-variation metric, as introduced in
{\cite{chevyrev_canonical_2019}} and {\cite{friz_differential_2018}}. We work within the framework introduced by Chevyrev et al.
\cite{chevyrev_superdiffusive_2024}, who generalized classical
Skorokhod-type metrics by introducing a more general notion of paths, called
decorated paths, and defining a suitable Skorokhod-type metric on the
corresponding path space. Both the rough driver $W$ and the solution $Y$ can
be embedded into the space of decorated paths.

In Chapter~\ref{decorated-path}, we summarize the results of
\cite{chevyrev_superdiffusive_2024} and prove a new alternative
characterization of convergence of decorated paths in
Lemma~\ref{alternative-a-metric}, which may be of independent interest. We then provide further details on how the
RBSDE solution can be embedded into the space of decorated paths in
Chapter~\ref{rBSDEasDecoratedPath}. Finally, in
Chapter~\ref{Stability-solution}, we state and prove the stability of RBSDE
solutions.

\subsection{Decorated Paths and Skorokhod-type p-variation
Metric}\label{decorated-path}

We are going to work under the notion of decorated paths as introduced in
{\cite{chevyrev_superdiffusive_2024}}. They work with c{\`a}dl{\`a}g paths,
while we work with c{\`a}gl{\`a}d paths. Due to this difference in the setting, some of our
definitions require slight adaptations from theirs, but nonetheless, the proof ideas remain the same. Therefore,  we refer the reader to {\cite{chevyrev_superdiffusive_2024}} for proofs.

\begin{definition}\label{def:decopaths}
  Let $I = [a, b] \subset \mathbb{R}$ be a closed interval and $\Phi : I
  \rightarrow D (I) \assign D ([0, 1], \mathbb{R}^e)$ be a path that maps into
  the space of c{\`a}gl{\`a}d paths on $[0, 1]$. We say $t \in I$ is a
  stationary point of $\Phi$ if $\Phi (t) \equiv \tmop{const}$. Let $\Pi
  \subset I$ be a subset of $I$.\\
  Let $\bar{\mathcal{D}} (I)$ denote the space that consists of pairs $(\Phi,
  \Pi)$ with the following properties:
  \begin{enumeratealpha}
    \item the function $t \mapsto \Phi (t) (0)$ lies in $D (I)$;
    
    \item the set $\Pi$ is at most countable and contains all non stationary points of
    $\Phi$;
    
    \item for all $\varepsilon > 0$, there exist only finitely many points $t
    \in \Pi$ such that $| \Phi (t) (1) - \Phi (t) (0) | > \varepsilon$.
  \end{enumeratealpha}
\end{definition}

For closed interval $I,J$, let $\Lambda_{I;J}$ denote the set of strictly increasing bijections from $I$ to
$J$ and $\Lambda_{I}:=\Lambda_{I;I}$. Two c{\`a}gl{\`a}d paths $\varphi^1, \varphi^2 \in D (I)$ are called
reparameterization of each other if there exists a $\lambda \in \Lambda_I$
such that $\varphi^1 = \varphi^2 \circ \lambda$.

We can now define an equivalence class on $\bar{\mathcal{D}} (I)$. We say that
$(\Phi^1, \Pi^1), (\Phi^2, \Pi^2) \in \bar{\mathcal{D}} (I)$ are equivalent if
the functions
\[ s \mapsto \Phi^i (t) (s)  \hspace{0.17em} \mathds{1}_{\{s < 1\}} + \Phi^i
   (t +) (0)  \hspace{0.17em} \mathds{1}_{\{s = 1\}} \]
are reparameterizations of each other for all $t \in I$. In particular, this
implies $\Phi^1 (t -) (0) = \Phi^1 (t) (0) = \Phi^2 (t) (0) = \Phi^2 (t -)
(0)$ and $\Phi^1 (t +) (0) = \Phi^2 (t +) (0)$ for all $t \in I$. Notice that
the definition of equivalence classes on $\bar{\mathcal{D}} (I)$ is
independent of $\Pi^1$ and $\Pi^2$.

\begin{definition}
  Denote by $\mathfrak{D} (I) \assign \bar{\mathcal{D}} (I) / \sim$ the quotient space
  of equivalence classes on $\bar{\mathcal{D}} (I)$. We refer to elements of
  $\mathfrak{D} (I)$ as {\tmstrong{decorated paths}}.
\end{definition}

The intuition behind this definition is the following. The object of primary
interest is still the c{\`a}gl{\`a}d path $t\mapsto \Phi(t)(0)$. However, this
path only records the positions immediately before and after a jump, but not
the dynamics by which the jump is traversed. The additional path
$\Phi(t)\in D([0,1],\mathbb{R}^e)$ encodes precisely this missing information.
The set $\Pi$ records the times at which such additional information is
relevant. In particular, it is enough to require $\Pi$ to be at most
countable, since only jumps require a non-trivial decoration.

There are several natural ways to embed a c{\`a}gl{\`a}d path into the space of
decorated paths, one only has to specify how the path traverses its
jumps. Two basic choices are given by the embeddings
\hypertarget{embedding-i}{$\imath : D(I) \hookrightarrow \bar{\mathcal{D}}(I)$},
defined by
\begin{align}
  \label{embedding-i}
  (\imath h)(t)(s) \assign h(t),\footnotemark
  \qquad t\in I,\ s\in[0,1].
\end{align}
\footnotetext{One may also define
$(\imath h)(t)(s)=h(t)\mathds{1}_{s\in[0,1-\epsilon)}
+h(t+)\mathds{1}_{s\in[1-\epsilon,1]}$ for $s\in[0,1]$ and any
$0<\epsilon<1$. Although this gives a different element of
$\bar{\mathcal{D}}(I)$, it represents the same element in the quotient space
$\mathfrak{D}(I)$.}
and
\hypertarget{embedding-j}{$\jmath : D(I) \hookrightarrow \bar{\mathcal{D}}(I)$},
defined by
\begin{align} \label{embedding-j}
    (\jmath h)(t)(s)
\assign
(1-s)h(t+) + s h(t),
\qquad t\in I,\ s\in[0,1],
\end{align}
here we suppress the set $\Pi$, but whenever needed, it can
always be chosen as a countable set containing all discontinuities of $h$.
The embedding $\imath$ fills a jump by a constant path, while $\jmath$ fills it by the linear path connecting $h(t+)$ and $h(t)$. These two embeddings, however, do not yet fully exploit the flexibility of the decorated path framework, since
they only use the positions immediately before and after each jump. Later, in
Section~\ref{rBSDEasDecoratedPath}, we will use more refined decorations to
encode the dynamics of Marcus-type jumps.  \\

As preparation for defining a Skorokhod-type metric on decorated paths, we
associate with each decorated path a c{\`a}gl{\`a}d path by inserting small fictitious time intervals at the jump times. Let $(\Phi,\Pi)\in \bar{\mathcal{D}}([a,b])$ with
$\Pi=\{t_k\}_{k=1,\ldots,m}$. We define the
{\textbf{$\delta$-extension}} $\Phi^{\delta, \Pi}\in D([a,b+\delta])$ by adding
fictitious time of total length $\delta>0$ as follows. For $\Pi = \emptyset$ we simply define
\[ \Phi^{\delta, \Pi} (t) = \Phi (t) (0)  \quad \text{for } t \in [a, b]  \quad
   \text{and} \quad \Phi^{\delta, \Pi} (t) = \Phi (b)(0)  \quad \text{for } t \in (b,
   b + \delta]. \]
Otherwise, set
\(r = \sum_{k = 1}^m 2^{- k}  \text{ and }  r_k = \frac{2^{- k}
   \delta}{r}.\)
We define a  c{\`a}gl{\`a}d function
\begin{equation}
  \tau^{\delta, \Pi} : [a, b] \to [a, b + \delta], \qquad \tau^{\delta, \Pi} (t) = t +
  \sum_{k = 1}^m r_k  \hspace{0.17em} \mathds{1}_{\{t_k < t\}}, \label{tau}
\end{equation}
which is strictly increasing, and moreover  define $\Phi^{\delta}$ as
\begin{equation}
  \Phi^{\delta, \Pi}_s := \left\{\begin{array}{ll}
    \Phi (t) (0), & \text{if } s = \tau^{\delta, \Pi}_t  \text{ for some } t \in [a,
    b],\\
    \Phi (t_k)  ((s - \tau^{\delta, \Pi}_{t_k})   / r_k), \nobracket & \text{if }
    s \in (\tau^{\delta, \Pi}_{t_k}, \tau^{\delta, \Pi}_{t_k +} \nobracket]  \text{ for
    some } 1 \leq k \leq m.
  \end{array}\right. \label{delta-extention}
\end{equation}
It is common in the literature to refer to the additional path segments on the
intervals $(\tau^{\delta} (t_k), \tau^{\delta} (t_k +)]$, $1 \leq k \leq m$,
as \tmtextbf{excursions} of the paths. \\
For later purposes, we also introduce the continuous non-decreasing surjection
$c^{\delta, \Pi} : [a, b + \delta] \rightarrow [a, b]$ defined by $c^{\delta, \Pi}
\equiv (\iota \tmop{id}_{[a, b]})^{\delta, \Pi}$ in the sense of
\eqref{delta-extention}. One easily verifies that $c^{\delta, \Pi}$ admits the equivalent representation
\[ c^{\delta, \Pi} (t) := \inf \{ s \in [a, b] : \tau^{\delta} (s) \geqslant
   t \}, \]
that is, $c^\delta$ is the left inverse of $\tau^{\delta, \Pi}$, i.e. $c^{\delta, \Pi} \circ
\tau^{\delta, \Pi} = \tmop{id}_{[a, b]}.$ \\

For $(\Phi^1, \Pi^1), (\Phi^2, \Pi^2) \in \bar{\mathcal{D}} (I)$, we can now
introduce the (pseudo) metric
\[ \alpha_{\infty ; [a, b]} ((\Phi^1, \Pi^1), (\Phi^2, \Pi^2)) \assign
   \lim_{\delta \to 0} \inf_{\lambda \in \Lambda_{[a, b + \delta]}} \max
   \left\{ \| \lambda - \mathrm{id} \|_{\infty}, \hspace{0.27em} \| \Phi^{1,
   \delta} \circ \lambda - \Phi^{2, \delta} \|_{\infty ; [a, b + \delta]}
   \right\} . \]
It has been shown in Lemma 8.12 of {\cite{chevyrev_superdiffusive_2024}} that
this limit exists and is independent of various choices that we have made
above, such as the ordering of the jumps, the definition of $r_k$, and
especially the choice of $\Pi^1$ and $\Pi^2$. We will therefore suppress the
set $\Pi$ and speak only of $\Phi \in \bar{\mathcal{D}} ([a, b])$. It has also
been shown that $\alpha_{\infty ; [a, b]} (\Phi^1, \Phi^2) = 0$ if and only if
$\Phi^1 \sim \Phi^2$, which means that even though $\alpha_{\infty}$ is not a
metric on $\bar{\mathcal{D}} (I)$, it is a metric on the equivalence class
$\mathfrak{D} (I)$. In fact, by Theorem 8.14 of
{\cite{chevyrev_superdiffusive_2024}}, the space $\mathfrak{D} (I)$ equipped
with $\alpha_{\infty}$ is a complete separable metric space.

The authors of \cite{chevyrev_superdiffusive_2024} have shown that everything
holds likewise also when, instead of the uniform norm, one uses the
$p$-variation norm. More precisely, we define
\[
\bar{\mathfrak{D}}^{p\text{-var}}(I)
\assign
\left\{
(\Phi,\Pi)\in \bar{\mathcal{D}}(I)
\,:\,
\|\Phi^\delta\|_{p\text{-var}}<\infty
\right\},
\]
and the corresponding quotient space
\[
\mathfrak{D}^{p\text{-var}}(I)
\assign
\bar{\mathfrak{D}}^{p\text{-var}}(I)/\sim .
\]
On this space, one defines the $p$-variation type Skorokhod metric by
\begin{align} \label{def-skorokhod-metric}
    \alpha_{p ; [a, b]} ((\Phi^1, \Pi^1), (\Phi^2, \Pi^2)) \assign \lim_{\delta
   \to 0} \inf_{\lambda \in \Lambda_{[a, b + \delta]}} \max \left\{ \| \lambda
   - \mathrm{id} \|_{\infty}, \hspace{0.27em} \| \Phi^{1, \delta} \circ
   \lambda - \Phi^{2, \delta} \|_{p ; [a, b + \delta]} \right\} . 
\end{align}
Let $\sigma^p_{J 1}$ denote the $p$-variation type J1 metric on $D^p (I)$ (see
{\cite{friz_differential_2018}}) and $\sigma^p_{\tmop{SM} 1}$ denote the
$p$-variation type SM1 metric on $D^p (I)$ (see
{\cite{chevyrev_canonical_2019}}), in the special case $p = \infty$, this is
the classical Skorokhod J1 (see {\cite{billingsley_convergence_1999}}, Chap.
12 or {\cite{jacod_limit_2003}}, Chap. VI) and SM1 metric on $D (I)$ (see
{\cite{whitt_stochastic-process_2002}}, Chap. 12.3). There is a natural
connection between these metrics through the embeddings $\imath$ (see \eqref{embedding-i}) and $\jmath$ (see \eqref{embedding-j}). It
is straightforward to see that $(D^p (I), \sigma^p_{J 1})$ is isometric to
$(\imath D^p (I), \alpha_p)$ and $(D^p (I), \sigma^p_{\tmop{SM} 1})$ is
isometric to $(\jmath D^p (I), \alpha_p)$. \\
For readers familiar with \cite{chevyrev_canonical_2019}, notice that for
$\Phi \in \bar{\mathcal{D}}(I)$, if $\Phi(t)(\cdummy)\in C([0,1])$ and
$\Phi(t+)(0)=\Phi(t)(1)$ for all $t\in I$, then the family
$(\Phi(t))_{t\in I}$ plays the role of path functions in the sense of
\cite{chevyrev_canonical_2019}. The notion of decorated paths is, however, more
general, since it also contains the image space of the embedding $\imath$,
where depending on the exact definition, either $\Phi(t+)(0)=\Phi(t)(1)$ or $\Phi(t)(\cdummy)\in C([0,1])$ does not hold.

We next provide an alternative, but equivalent, formulation of the Skorokhod
metric $\alpha_p$. Readers familiar with the classical Skorokhod $M_1$
topology may recognize the conceptual similarity to the definition via
completed graphs, see \cite{whitt_stochastic-process_2002}. The following
result is due to \cite{chevyrev_superdiffusive_2024}. Although it is stated
there explicitly only for $p=\infty$ in
\cite[Lemma~8.13]{chevyrev_superdiffusive_2024}, the same proof applies
equally to all $p\in[1,\infty]$.

\begin{lemma}
  \label{lemma-alternative-alpha-p}
  Given $(x_1,\Pi^1),(x_2,\Pi^2)\in \mathfrak{D}^{p\text{-var}}$,
  $p\in[1,\infty]$, it holds for every $\delta>0$ that
  \begin{equation}
    \alpha_p((x_1,\Pi^1),(x_2,\Pi^2))
    =
    \inf_{\lambda\in\Lambda_{[0,T+\delta]}}
    \max\left\{
    \|x^{\delta,\Pi^1}_1\circ\lambda-x^{\delta,\Pi^2}_2\|_{p},
    |c^{\delta,\Pi^1}\circ\lambda-c^{\delta,\Pi^2}|_{\infty}
    \right\}.
    \label{alternative-alpha-p}
  \end{equation}
\end{lemma}

Unlike the definition of the Skorokhod metric $\alpha_p$ in
\eqref{def-skorokhod-metric}, the formulation above is stated for fixed
$\delta>0$ and does not require taking the limit $\delta\to0$. Based on this
result, we can characterize convergence of decorated paths
$(x_n,\Pi^n)$ to $(x,\Pi)$ in $\alpha_p$ in a form that is much easier to use
than the original definition of $\alpha_p$.

\begin{lemma}
  \label{alternative-a-metric}
  Let $p\in[1,\infty]$ and let
  $(x,\Pi),(x_n,\Pi^n)\in\mathfrak{D}^{p\text{-var}}([0,T])$ for
  $n\in\mathbb{N}$. Then $(x_n,\Pi^n)$ converges to $(x,\Pi)$ in the $\alpha_p$ metric
  if and only if, for every $\delta>0$, there exists a sequence
  $(\lambda^{n,\delta})_{n\in\mathbb{N}}\subset
  \Lambda_{[0,T+\delta]}$ such that
  \begin{equation}
    \lim_{n\to\infty}
    \max\left\{
    \|x^{\delta,\Pi}-x^{\delta,\Pi^n}_n\circ\lambda^{n,\delta}\|_{p},
    |c^{\delta,\Pi}-c^{\delta,\Pi^n}\circ\lambda^{n,\delta}|_{\infty}
    \right\}
    =0.
    \label{alternativ-alpha}
  \end{equation}
  Whenever \eqref{alternativ-alpha} holds for some $\delta>0$ and
  $(\lambda^{n,\delta})_{n\in\mathbb{N}}\subset
  \Lambda_{[0,T+\delta]}$, it is possible to construct for any other $\bar{\delta}>0$ a sequence of
  reparameterizations
  $(\lambda^{n,\bar{\delta}})_{n\in\mathbb{N}}\subset
  \Lambda_{[0,T+\bar{\delta}]}$ such that \eqref{alternativ-alpha} holds for this $\bar{\delta}$. Those can be constructed by setting
  \[
  \lambda^{n,\bar{\delta}}
  =
  \phi_n^{\bar{\delta},\delta}
  \circ
  \lambda^{n,\delta}
  \circ
  \phi^{\delta,\bar{\delta}},
  \]
  where
  $\phi_n^{\bar{\delta},\delta}\in
  \Lambda_{[0,T+\delta];[0,T+\bar{\delta}]}$ and
  $\phi^{\delta,\bar{\delta}}\in
  \Lambda_{[0,T+\bar{\delta}];[0,T+\delta]}$ denote the reparameterizations, which are satisfy the identities  
  \[
  y^{\delta,\Pi^n}
  = y^{\bar{\delta},\Pi^n}\circ\phi_n^{\bar{\delta},\delta}, \quad z^{\delta,\Pi}
  = z^{\bar{\delta},\Pi}\circ\phi^{\bar{\delta},\delta}, \quad \text{ for any }  (y,\Pi^n),(z,\Pi)\in\mathfrak{D}.
  \]
  In particular, this applies to
  $y=x_n$, $z=x$, and $y,z=\imath\mathrm{id}$.
\end{lemma}
\begin{remark}
  The existence of such a reparametrization $\phi^{\bar{\delta},\delta}$
  follows directly from the construction of the $\delta$-extension in
  \eqref{delta-extention}. Notice that $\phi^{\bar{\delta},\delta}$ depends on $\delta$, $\bar{\delta}$, and the set of stationary points $\Pi$, but not on $y$.  Moreover, $(\phi^{\bar{\delta},\delta})^{-1}
  =
  \phi^{\delta,\bar{\delta}}$.
\end{remark}
\begin{proof}
  Fix $\delta>0$. By Lemma~\ref{lemma-alternative-alpha-p}, the convergence
  $\alpha_p((x_n,\Pi^n),(x,\Pi))\to0$ is equivalent to
  \[
  \lim_{n\to\infty}
  \inf_{\lambda\in\Lambda_{[0,T+\delta]}}
  \max\left\{
  \|x_n^{\delta,\Pi^n}\circ\lambda-x^{\delta,\Pi}\|_p,\,
  \|c^{\delta,\Pi^n}\circ\lambda-c^{\delta,\Pi}\|_\infty
  \right\}
  =0.
  \]
  By the definition of the infimum, for each $n\in\mathbb{N}$ we can choose
  $\lambda^{n,\delta}\in\Lambda_{[0,T+\delta]}$ such that
  \[
  \max\left\{
  \|x_n^{\delta,\Pi^n}\circ\lambda^{n,\delta}-x^{\delta,\Pi}\|_p,\,
  \|c^{\delta,\Pi^n}\circ\lambda^{n,\delta}-c^{\delta,\Pi}\|_\infty
  \right\}
  \leq \frac{2}{n}.
  \]
  Letting $n\to\infty$ yields \eqref{alternativ-alpha}. For another $\bar{\delta}>0$, the construction of
  $(\lambda^{n,\bar{\delta}})_{n\in\mathbb{N}}$ follows by rewriting the
  previous estimate as
  \[
  \max\left\{
  \|x_n^{\bar{\delta},\Pi^n}\circ \phi_n^{\bar{\delta},\delta}
  \circ \lambda^{n,\delta}
  -x^{\bar{\delta},\Pi}\circ \phi^{\bar{\delta},\delta}\|_p,\,
  \|c^{\bar{\delta},\Pi^n}\circ \phi_n^{\bar{\delta},\delta}
  \circ \lambda^{n,\delta}
  -c^{\bar{\delta},\Pi}\circ \phi^{\bar{\delta},\delta}\|_\infty
  \right\}
  \leq \frac{2}{n}.
  \]
  Setting $\lambda^{n,\bar{\delta}}
  \assign
  \phi_n^{\bar{\delta},\delta}
  \circ \lambda^{n,\delta}
  \circ (\phi^{\bar{\delta},\delta})^{-1}
  =
  \phi_n^{\bar{\delta},\delta}
  \circ \lambda^{n,\delta}
  \circ \phi^{\delta,\bar{\delta}}$ 
  gives \eqref{alternativ-alpha} with $\bar{\delta}$ in place of $\delta$.
\end{proof}

We do not use the following corollary directly in the paper, but we believe it
may be a useful tool. We therefore include it for completeness.
\begin{corollary}
    \label{alternative-a-metric-sequential}
  Let $p\in[1,\infty]$ and let
  $(x,\Pi),(x_n,\Pi^n)\in\mathfrak{D}^{p\text{-var}}([a,b])$ for
  $n\in\mathbb{N}$. Then $(x_n,\Pi^n)$ converges to $(x,\Pi)$ in $\alpha_p$
  if and only if, for any sequence $(\delta^l)_{l\in\mathbb{N}}$ with
  $\lim_{l\to\infty}\delta^l=0$, there exists a sequence of
  reparameterizations $(\lambda^{n,\delta^l}\in\Lambda_{[0,T+\delta^l]})_{n,l\in\mathbb{N}}$
  such that
\begin{align} \label{alternative-ap-limit}
&\lim_{n\to\infty}\lim_{l\to\infty}
\max\left\{
\|x_n^{\delta^l,\Pi^n}\circ \lambda^{n,\delta^l}
-x^{\delta^l,\Pi}\|_{p},
|\lambda^{n,\delta^l}-\mathrm{id}|_{\infty}
\right\} \\
&\qquad =
\lim_{l\to\infty}\lim_{n\to\infty}
\max\left\{
\|x_n^{\delta^l,\Pi^n}\circ \lambda^{n,\delta^l}
-x^{\delta^l,\Pi}\|_{p},
|\lambda^{n,\delta^l} -\mathrm{id}|_{\infty}
\right\}
=0. \nonumber
\end{align}
\end{corollary}

\begin{proof}
First, notice that we can choose the reparametrizations
$(\lambda^{n,\delta^l}\in\Lambda_{[0,T+\delta^l]})$ as in
Lemma~\ref{alternative-a-metric}. This implies, in particular, that, for any
$k\in\mathbb{N}$,
\begin{align*}
   & \max\left\{
\|x_n^{\delta^l,\Pi^n}\circ \lambda^{n,\delta^l}
-x^{\delta^l,\Pi}\|_{p},
|\lambda^{n,\delta^l}-\mathrm{id}|_{\infty}
\right\}  \\
= {}& \max\left\{
\|x_n^{\delta^k,\Pi^n} \circ \lambda^{n,\delta^k}
  \circ \phi^{\delta^k,\delta^l}
-x^{\delta^k,\Pi}
  \circ \phi^{\delta^k,\delta^l}\|_{p},
|\phi^{\delta^l,\delta^k} \circ \lambda^{n,\delta^k}
\circ \phi^{\delta^k,\delta^l}
-\phi^{\delta^l,\delta^k} \circ \phi^{\delta^k,\delta^l}|_{\infty}
\right\} \\
= {}& \max\left\{
\|x_n^{\delta^k,\Pi^n} \circ \lambda^{n,\delta^k}
-x^{\delta^k,\Pi} \|_{p},
|\lambda^{n,\delta^k} -\mathrm{id}|_{\infty}
\right\},
\end{align*}
where in the last equality we used the invariance of the $p$-variation norm
under reparametrization and the bijectivity of
$\phi^{\delta^l,\delta^k}$. This observation explains why the two iterated
limits in \eqref{alternative-ap-limit} may be interchanged. We are only left to show that
\[
\lim_{l\to\infty}\lim_{n\to\infty} |c^{\delta^l,\Pi}-c^{\delta^l,\Pi^n}\circ\lambda^{n,\delta^l}-\lambda^{n,\delta^l} + \mathrm{id}|_{\infty} = 0.
\]
Indeed, applying Lemma~\ref{alternative-a-metric} then gives
\[
\lim_{n\to\infty}
\max\left\{
\|x_n^{\delta^l,\Pi^n}\circ \lambda^{n,\delta^l}
-x^{\delta^l,\Pi}\|_{p},
|c^{\delta^l,\Pi}-c^{\delta^l,\Pi^n}\circ\lambda^{n,\delta^l}|_{\infty}
\right\}
=0,
\]
and taking $\lim_{l\to\infty}$ yields the desired result. To prove the
remaining claim, observe that
\begin{align*}
|c^{\delta^l,\Pi}-c^{\delta^l,\Pi^n}\circ\lambda^{n,\delta^l} - \lambda^{n,\delta^l} + \mathrm{id}|_{\infty}
\leq 
|\lambda^{n,\delta^l}
- c^{\delta^l,\Pi^n}\circ\lambda^{n,\delta^l}|_{\infty}
+ |c^{\delta^l,\Pi}-\mathrm{id}|_{\infty}.
\end{align*}
The first term is bounded by
$|\mathrm{id}-c^{\delta^l,\Pi^n}|_{\infty}$. Hence, after taking
$\lim_{l\to\infty}$, the right-hand side vanishes.
\end{proof}

\subsection{Rough BSDE Solution as Decorated
Path}\label{rBSDEasDecoratedPath}

We have shown in the previous section that the rough driver $W$ can be
naturally lifted to the space of decorated paths via $\imath$ (or $\jmath$) by
adding constant (or linear) excursions as additional information at each
discontinuity. Measuring the lift $\imath W$ (or $\jmath W$) in $\alpha_p$ is
then equivalent to a $p$-variation-type Skorokhod J1 (or M1) metric.

As for the solution $(Y, Z)$ to the RBSDE, the $Z$-component does not play an
important role in this section, since its norm can be interpreted as a norm of
$\int Z \tmop{dB}$, which is almost surely continuous. Embedding the
$Y$-component into the space of decorated paths is more subtle. Indeed, even
in the simpler case of ODEs driven by bounded variation paths, a sequence of
drivers convergent in the J1 metric may fail to produce convergent solutions
in J1 or M1 (see Example 1.4 in {\cite{chevyrev_superdiffusive_2020}}), which
corresponds to the $\imath$- or $\jmath$-embedding. However, there is a
natural way to embed $Y$ that draws on the ``time-stretching'' idea commonly
used for Marcus-type equations.

Let $\Pi = \{t_k \}_{k = 1, \ldots, m} \subset [0, T]$ be a countable set
containing all discontinuities of $W$, and let $\mathcal{E}$ be an embedding
from $D (I)$ to $\bar{\mathcal{D}} (I)$. In this paper, $\mathcal{E}$ will be
chosen as either $\imath$ or $\jmath$ to relate to forward- or Marcus-type
jumps. We want to add fictitious time to the RBSDE and we begin by defining
$W^{\delta} \hspace{0.27em} = \hspace{0.27em} (\mathcal{E}W)^{\delta}$, where
the right-hand side is the $\delta$-extension from \eqref{delta-extention}
applied with respect to $\Pi$.

We then study the following (forward-type) RBSDE on $[0, T + \delta]$:
\begin{equation}
  \hat{Y}_t^{\delta} \hspace{0.27em} = \hspace{0.27em} \xi \hspace{0.27em} +
  \hspace{0.27em} \int_t^{T + \delta} \hat{f}^{\delta} (r, \hat{Y}^{\delta}_r,
  \hat{Z}^{\delta}_r) \hspace{0.17em} \mathrm{d} c^{\delta}_r \hspace{0.27em}
  + \hspace{0.27em} \int_t^{T + \delta} \hat{g}^{\delta} (r,
  \hat{Y}^{\delta}_r) \hspace{0.17em} \mathrm{d} W^{\delta}_r \hspace{0.27em}
  - \hspace{0.27em} \int_t^{T + \delta} \hat{Z}^{\delta}_r \hspace{0.17em}
  \mathrm{d} B^{\delta}_r, \label{cts_rBSDE}
\end{equation}
where $c^{\delta} = (\imath \hspace{0.17em} \mathrm{id})^{\delta}$,
$\hat{f}^{\delta} (r, \cdot, \cdot) = f (c^{\delta}_r, \cdot, \cdot)$,
$\hat{g}^{\delta} (r, \cdot) = g (c^{\delta}_r, \cdot)$, and $B^{\delta} =
(\imath B)^{\delta} = B_{c^{\delta}}$. Note that all $\delta$-extensions are
defined with respect to $\Pi$ and that for $c^\delta$ and $B^\delta$, we use constant excursions, since both $c$ and
$B$ are continuous. Under Assumption A, the
well-posedness of \eqref{cts_rBSDE} follows from Theorem
\ref{global_existenceuniqueness}.

We can now construct a decorated path \hypertarget{lift-Y}{$\mathbf{Y} : [0, T] \rightarrow D^p
([0, 1])$}, which naturally have $\hat{Y}^{\delta}$ as its $\delta$-extension,
for that we simply define $\mathbf{Y} (t)$ to be a linear reparameterization
of $Y^{\delta} \mid_{[ \tau^{\delta} (t), \tau^{\delta} (t +)]}$ for every $t\in [0,T]$.

In the next two theorems, we show rigorously how the RBSDE (\ref{cts_rBSDE})
can be seen as the ``time-stretched'' version of the Forward-RBSDE
(\ref{Forward-RBSDE}) or Marcus-RBSDE (\ref{Marcus-RBSDE}) depending on
the choice of $\mathcal{E}$ being $\imath$ or $\jmath$. The following implies, in particular, that
$\mathbf{Y} (t) (0) = Y^{\delta}_{\tau^{\delta} (t)} = Y_t$, $t \in [0, T]$, justifying the decorated path
$\mathbf{Y}$ being a lift of $Y$ in the space of decorated paths. The proofs
are inspired by Theorem 38 of {\cite{friz_general_2017}}.

\begin{theorem}
  \label{fictioustime_Marcus-RBSDE}Let $(Y, Z)$ be a solution of Marcus-RBSDE
  (\ref{Marcus-RBSDE}). We define
  \begin{align*}
    \hat{Y}^{\delta}_s = \left\{ \begin{array}{ll}
      Y_t & , \text{if } s = \tau^{\delta}_t  \text{ for some } t \in [0, T],\\
      \varphi (-g_{t_n} \Delta W_{t_n}, Y_{t_n}, (s - \tau^{\delta}_{t_k}) / r_k) & , \text{if } s \in (\tau^{\delta}_{t_k}, \tau^{\delta}_{t_k
      +} \nobracket]  \text{ for some } 1 \leq k \leq m,
    \end{array} \right.
  \end{align*}
  and $\hat{Z}^{\delta}_s \assign Z_{c^{\delta} (s)}$ for all $s \in [0, T + \delta]$. Then
  $(\hat{Y}^{\delta}, \hat{Z}^{\delta})$ is the unique solution to the RBSDE
  (\ref{cts_rBSDE}) with $W^{\delta} = (\jmath W)^{\delta}$.
  
  Conversely, let $(\hat{Y}^{\delta}, \hat{Z}^{\delta})$ be the solution to the
  RBSDE (\ref{cts_rBSDE}), the pair $(\hat{Y}^{\delta}_{\tau^{\delta}},
  \hat{Z}^{\delta}_{\tau^{\delta}})$ is the unique solution to the
  Marcus-RBSDE (\ref{Marcus-RBSDE}). 
\end{theorem}

\begin{proof}
  Notice that by construction it holds $Y_t = \hat{Y}^{\delta}_{\tau^{\delta}
  (t)}$ and $Z_t = \hat{Z}^{\delta}_{\tau^{\delta} (t)}$ for all $t \in [0,
  T]$. Furthermore, if $t \in \mathrm{Im} (\tau^{\delta}) \subset [0, T +
  \delta]$, then $\hat{Y}^{\delta}_t = Y_{c^{\delta}_t}$. The converse
  statement follows from the first result together with the uniqueness of the
  solution of (\ref{cts_rBSDE}).
  
  We start the proof of the first statement with the observation that for $t
  \in [0, T + \delta] \setminus \mathrm{Im} (\tau^{\delta})$ (i.e. for $t \in
  (\tau^{\delta} (t_k), \tau^{\delta} (t_k +)]$ for some $1 \le k \le m$) by
  definition of $\hat{Y}^{\delta}$ we have
  \begin{align}
    \hat{Y}^{\delta}_t
    &=
    \varphi\left(
    -g_{t_k} \Delta W_{t_k}, Y_{t_k +},
    \frac{t-\tau^{\delta}(t_k)}{r_k}
    \right) \nonumber\\
    &=
    Y_{t_k +}
    +
    \int_{\frac{t-\tau^{\delta}(t_k)}{r_k}}^1
    g_{t_k}
    \left(
    \varphi\left(
    -g_{t_k} \Delta W_{t_k}, Y_{t_k +}, q
    \right)
    \right)
    \Delta W_{t_k}\, dq \nonumber\\
    &=
    Y_{t_k +}
    +
    \int_t^{\tau^{\delta}(t_k +)}
    g_{t_k}
    \left(
    \varphi\left(
    -g_{t_k} \Delta W_{t_k}, Y_{t_k +},
    \frac{r-\tau^{\delta}(t_k)}{r_k}
    \right)
    \right)
    \frac{1}{r_k}\Delta W_{t_k}\, dr \nonumber\\
    &=
    Y_{t_k +}
    +
    \int_t^{\tau^{\delta}(t_k +)}
    \hat{g}^{\delta}_{\tau^{\delta}(t_k)}
    \left(\hat{Y}^{\delta}_r\right)
    dW^{\delta}_r .
    \label{cts_approx2}
\end{align}
  Then we show that $(\hat{Y}^{\delta}, \hat{Z}^{\delta})$ satisfies
  (\ref{cts_rBSDE}) for all $t \in \mathrm{Im} (\tau^{\delta})$.
  
  We now argue pathwise on a $\mathbb{P}$-full set on which the RBSDE holds and the relevant Young integrals are well-defined. We introduce the notation $a \approx_{\varepsilon} b$ meaning $|a - b| \le
  \varepsilon$. Due to the absolute continuity of the last term of the
  Marcus-RBSDE (\ref{Marcus-RBSDE}), for every $t \in [0, T + \delta]$ and
  every $\varepsilon > 0$ there exist $n \in \mathbb{N}$ and time points $\{t_i
  \}_{i = 0, \ldots, n}$ (with some being jump points) with
  \[ c^{\delta}_t = t_0 < t_1 < t_2 < \cdots < t_n < t_{n + 1} = T, \]
  such that
  \begin{align*}
    Y_{c^{\delta}_t} - \xi \approx_{\varepsilon} & \int_{c^{\delta}_t}^T f (r,
    Y_r, Z_r)  \hspace{0.17em} dr - \int_{c^{\delta}_t}^T Z_r  \hspace{0.17em}
    dB_r + \sum_{i = 0}^n \int_{t_i}^{t_{i + 1}} g_r (Y_r)  \hspace{0.17em}
    dW_r\\
    & + \sum_{i = 0}^n [\varphi (-g_{t_i +} \Delta W_{t_i}, Y_{t_i +}) -
    Y_{t_i +} - g_{t_i} (Y_{t_i +}) \Delta W_{t_i}] .
  \end{align*}
  By the MRS convergence of $\int g_r (Y_r)  \hspace{0.17em} dW_r$ (see
Proposition \ref{backwardYoung}), we can find partitions
\[
\mathcal{P}_i = \{t_i = t_0^i, \ldots, t_{n_i+1}^i = t_{i + 1} \}
\]
of $[t_i, t_{i + 1}]$ for $i = 0, \ldots, n$ such that
\[
\int_{t_i}^{t_{i + 1}} g_r (Y_r)  \hspace{0.17em} dW_r
\approx_{\frac{1}{n + 1} \varepsilon}
\sum_{j = 0}^{n_i} g_{t_{j + 1}^i}
(Y_{t_{j + 1}^i}) W_{t_j^i, t_{j + 1}^i}.
\]
Moreover, by the definition of MRS convergence, the partitions
$\mathcal{P}_i$ may be chosen such that the above estimate remains valid after
replacing $\mathcal{P}_i$ by any refinement; see also
\cite[Definition 1.1]{friz_differential_2018}. Thus,
\begin{align}
  Y_{c^{\delta}_t} - \xi \approx_{2 \varepsilon} & \int_{c^{\delta}_t}^T f (r,
  Y_r, Z_r)  \hspace{0.17em} \tmop{dr} - \int_{c^{\delta}_t}^T Z_r 
  \hspace{0.17em} \tmop{dB}_r + \sum_{i = 0}^n \sum_{j = 0}^{n_i} g_{t_{j +
  1}^i} (Y_{t_{j + 1}^i}) W_{t_j^i, t_{j + 1}^i} \nonumber\\
  & + \sum_{i = 0}^n [\varphi (-g_{t_i} \Delta W_{t_i}, Y_{t_i +}) -
  Y_{t_i +} - g_{t_i} (Y_{t_i +}) \Delta W_{t_i}] .  \label{cts_approx1}
\end{align}
  Equation \eqref{cts_approx2} implies in particular that
  \[ \varphi (-g_{t_i} \Delta W_{t_i}, Y_{t_i +}) = Y_{t_i +} +
     \int_{\tau^{\delta} (t_i)}^{\tau^{\delta}  (t_i +)}
     \hat{g}^{\delta}_{\tau^{\delta} (t_i)} (\hat{Y}^{\delta}_r)
     \tmop{dW}^{\delta}_r . \]
  Combining this with \eqref{cts_approx1} and applying a time-change result {\cite[Prop.V.1.5]{revuz_continuous_1999}} along
  with the definitions of $W^{\delta}$ and $\hat{Y}^{\delta}$ yields
  \begin{align*}
    Y_{c^{\delta}_t} - \xi \approx_{2 \varepsilon} & \int_t^{T + \delta}
    \hat{f}^{\delta} (r, Y_{c^{\delta}_r}, \hat{Z}^{\delta}_r) 
    \hspace{0.17em} dc^{\delta}_r - \int_t^{T + \delta} \hat{Z}^{\delta}_r 
    \hspace{0.17em} \tmop{dB}^{\delta}_r\\
    & + \sum_{i = 0}^n \int_{\tau^{\delta} (t_i)}^{\tau^{\delta}  (t_i +)}
    \hat{g}^{\delta}_{\tau^{\delta}  (t_i )} (\hat{Y}^{\delta}_r)
    \tmop{dW}^{\delta}_r - \sum_{i = 0}^n \hat{g}^{\delta}_{\tau^{\delta} 
    (t_i )} (\hat{Y}^{\delta}_{\tau^{\delta} (t_i +)})
    W^{\delta}_{\tau^{\delta} (t_i), \tau^{\delta}  (t_i +)}\\
    & + \sum_{i = 0}^n \sum_{j = 0}^{n_i} \hat{g}^{\delta}_{\tau^{\delta}
    (t^i_{j + 1})} (\hat{Y}^{\delta}_{\tau^{\delta} (t_{j + 1}^i)})
    W^{\delta}_{\tau^{\delta} (t_j^i), \tau^{\delta} (t_{j + 1}^i)} .
  \end{align*}
  For all $i = 0, \ldots, n$, note that $\tau^{\delta}  (t_i +) <
  \tau^{\delta} (t_1^i)$. Define the partition of $[\tau^{\delta} (t_i +),
  \tau^{\delta} (t_{i + 1})]$ by
  \[ \tilde{P}_i \assign \{ \tilde{t}_0^i, \tilde{t}_1^i, \ldots,
     \tilde{t}_{n_i + 1}^i \} \assign \{\tau^{\delta} (t_i +), \tau^{\delta}
     (t_1^i), \ldots, \tau^{\delta} (t_{n_i + 1}^i)\} . \]
  Then we can rewrite the above as
  \begin{align*}
    Y_{c^{\delta}_t} - \xi \approx_{2 \varepsilon} & \int_t^{T + \delta}
    \hat{f}^{\delta} (r, Y_{c^{\delta}_r}, \hat{Z}^{\delta}_r) 
    \hspace{0.17em} dc^{\delta}_r - \int_t^{T + \delta} \hat{Z}^{\delta}_r 
    \hspace{0.17em} \tmop{dB}^{\delta}_r + \sum_{i = 1}^n \int_{\tau^{\delta}
    (t_i)}^{\tau^{\delta}  (t_i +)} \hat{g}^{\delta}_{\tau^{\delta} (t_i)}
    (\hat{Y}^{\delta}_r) \tmop{dW}^{\delta}_r\\
    & + \sum_{i = 0}^n \bigg( \sum_{j = 0}^{n_i}
    \hat{g}^{\delta}_{\tilde{t}_{j + 1}^i} (\hat{Y}^{\delta}_{\tilde{t}_{j +
    1}^i}) \hspace{0.17em} W^{\delta}_{\tilde{t}_j^i, \tilde{t}_{j + 1}^i} +
    (\hat{g}^{\delta}_{\tilde{t}_1^i} (\hat{Y}^{\delta}_{\tilde{t}_1^i}) -
    \hat{g}^{\delta}_{\tilde{t}_0^i} (\hat{Y}^{\delta}_{\tilde{t}_0^i}))
    \hspace{0.17em} W^{\delta}_{\tau^{\delta} (t_i), \tau^{\delta} (t_i +)}
    \bigg)
  \end{align*}
  We can find refinements $\bar{P}_i \subset \tilde{P}_i$ such that
  \[ \hat{g}^{\delta}_{\tilde{t}_1^i} (\hat{Y}^{\delta}_{\bar{t}_1^i}) -
     \hat{g}^{\delta}_{\tilde{t}_0^i} (\hat{Y}^{\delta}_{\bar{t}_0^i}) \le
     \frac{\varepsilon}{n \|W\|_{q \text{-var} ; [0, T]}} \]
  for all $i = 1, \ldots, n$. This implies
  \begin{align*}
    Y_{c^{\delta}_t} - \xi \approx_{3 \varepsilon} & \int_t^{T + \delta}
    \hat{f}^{\delta} (r, Y_{c^{\delta}_r}, \hat{Z}^{\delta}_r) 
    \hspace{0.17em} dc^{\delta}_r - \int_t^{T + \delta} \hat{Z}^{\delta}_r 
    \hspace{0.17em} dB^{\delta}_r + \sum_{i = 0}^n \sum_{j = 0}^{n_i}
    \hat{g}^{\delta}_{\tilde{t}_{j + 1}^i} (\hat{Y}^{\delta}_{\tilde{t}_{j +
    1}^i}) W^{\delta}_{\tilde{t}_j^i, \tilde{t}_{j + 1}^i} .
  \end{align*}
  The right-hand side converges (in the RRS sense) to
  \[ \int_t^{T + \delta} \hat{f}^{\delta} (r, Y_{c^{\delta}_r},
     \hat{Z}^{\delta}_r)  \hspace{0.17em} dc^{\delta}_r - \int_t^{T + \delta}
     \hat{Z}^{\delta}_r  \hspace{0.17em} dB^{\delta}_r + \int_t^T
     \hat{g}^{\delta}_r (\hat{Y}^{\delta}_r)  \hspace{0.17em} dW^{\delta}_r .
  \]
  By construction, $Y_{c^{\delta}_t} - \xi$ is also the RRS limit, and by
  uniqueness we have
  \[ Y_{c^{\delta}_t} = \xi + \int_t^{T + \delta} \hat{f}^{\delta} (r,
     Y_{c^{\delta}_r}, \hat{Z}^{\delta}_r)  \hspace{0.17em} dc^{\delta}_r -
     \int_t^{T + \delta} \hat{Z}^{\delta}_r  \hspace{0.17em} dB^{\delta}_r +
     \int_t^T \hat{g}^{\delta}_r (\hat{Y}^{\delta}_r)  \hspace{0.17em}
     dW^{\delta}_r . \]
  Finally, it holds that $\hat{Y}^{\delta}_t = Y_{c^{\delta}_t}$ for all $t
  \in \mathrm{Im} (\tau^{\delta})$ and $c^\delta$ is constant on $[0, T + \delta]
  \setminus \mathrm{Im} (\tau^{\delta})$. Hence, for all $t \in \mathrm{Im}
  (\tau^{\delta})$ we have shown
  \[ Y_{c^{\delta}_t} = \hat{Y}^{\delta}_t = \xi + \int_t^{T + \delta}
     \hat{f}^{\delta} (r, \hat{Y}^{\delta}_r, \hat{Z}^{\delta}_r) 
     \hspace{0.17em} dc^{\delta}_r - \int_t^{T + \delta} \hat{Z}^{\delta}_r 
     \hspace{0.17em} dB^{\delta}_r + \int_t^T \hat{g}^{\delta}_r
     (\hat{Y}^{\delta}_r)  \hspace{0.17em} dW^{\delta}_r . \]
  For all $1 \le k \le m$, we get in particular
  \[ Y_{t_k +} = \hat{Y}^{\delta}_{\tau^{\delta}  (t_k +)} = \xi +
     \int_{\tau^{\delta}  (t_k +)}^{T + \delta} \hat{f}^{\delta} (r,
     \hat{Y}^{\delta}_r, \hat{Z}^{\delta}_r)  \hspace{0.17em} dc^{\delta}_r -
     \int_{\tau^{\delta}  (t_k +)}^{T + \delta} \hat{Z}^{\delta}_r 
     \hspace{0.17em} dB^{\delta}_r + \int_{\tau^{\delta}  (t_k +)}^T
     \hat{g}^{\delta}_r (\hat{Y}^{\delta}_r)  \hspace{0.17em} dW^{\delta}_r,
  \]
  by taking $\varepsilon \to 0$ for $\hat{Y}^{\delta}_{\tau^{\delta}  (t_k +
  \varepsilon)}$. For all $t \in [0, T + \delta] \setminus \mathrm{Im}
  (\tau^{\delta})$, i.e. $t \in (\tau^{\delta} (t_k), \tau^{\delta} (t_k +)]$
  for some $1 \le k \le m$, we can now add \eqref{cts_approx2} to the above to
  get
  \[ \hat{Y}^{\delta}_t = \xi + \int_{\tau^{\delta}  (t_k +)}^{T + \delta}
     \hat{f}^{\delta} (r, \hat{Y}^{\delta}_r, \hat{Z}^{\delta}_r) 
     \hspace{0.17em} \tmop{dc}^{\delta}_r - \int_{\tau^{\delta}  (t_k +)}^{T +
     \delta} \hat{Z}^{\delta}_r  \hspace{0.17em} dB^{\delta}_r + \int_t^T
     \hat{g}^{\delta}_r (\hat{Y}^{\delta}_r)  \hspace{0.17em} dW^{\delta}_r .
  \]
  This is the desired equation (\ref{cts_rBSDE}) for $t \in [0, T + \delta]
  \setminus \mathrm{Im} (\tau^{\delta})$, noting that both $c^{\delta}$ and
  $B^{\delta}$ are constant on $(\tau^{\delta} (t_k), \tau^{\delta} (t_k +)]$.
\end{proof}

By exactly the same argument, we can also get
\begin{theorem}
  \label{fictioustime_Forward-RBSDE}Let $(Y, Z)$ be a solution of
  Forward-RBSDE (\ref{Forward-RBSDE}). The pair $(Y_{c^{\delta}},
  Z_{c^{\delta}})$ is the unique solution to the RBSDE (\ref{cts_rBSDE}) with
  $W^{\delta} = (\imath W)^{\delta}$.
  
  Conversely, let $(\hat{Y}^{\delta}, \hat{Z}^{\delta})$ be the solution to the
  RBSDE (\ref{cts_rBSDE}), the pair $(\hat{Y}^{\delta}_{\tau^{\delta}},
  \hat{Z}^{\delta}_{\tau^{\delta}})$ is the unique solution to the
  Forward-RBSDE (\ref{Forward-RBSDE}).
\end{theorem}

\subsection{Stability of Solution Map}\label{Stability-solution}
Before we state and prove the continuity result, we shall first prove the following useful lemma.

\begin{lemma}
  \label{continuity_parametrization}Given any $p > p' \geq 0$, let $x : \Omega \times [0, T] \rightarrow
  \mathbb{R}^h$ be a process with continuous sample paths of finite
  {$p'$-variation}. Let $\lambda^l \in \Lambda_{[0, T +
  \delta^l]}$, $l \in \mathbb{N}$, be a sequence of reparameterizations with
  $\lim_{l \rightarrow \infty} \| \lambda^l - \tmop{id}_{[0,T+\delta^l]} \|_{\infty} = 0$. Then, it holds
  $\lim_{l \rightarrow \infty} \| x^{\delta^l} \circ \lambda^l - x^{\delta^l}
  \|_{p ; [0, T + \delta^l]} = 0$ almost surely.
  
  \begin{proof}
    We apply interpolation  {\cite[Prop.~5.5]{friz_multidimensional_2010}} to
    get
    \begin{eqnarray*}
      \| x^{\delta^l} \circ \lambda^l - x^{\delta^l} \|_{p ; [0, T +
      \delta^l]} & \leq & \| x^{\delta^l} \circ \lambda^l - x^{\delta^l}
      \|_{p - \varepsilon ; [0, T + \delta^l]}^{\frac{p - \varepsilon}{p}} \|
      x^{\delta^l} \circ \lambda^l - x^{\delta^l} \|_{0 ; [0, T +
      \delta^l]}^{\frac{\varepsilon}{p}}\\
      & \leq & 2 \| x^{\delta^l} \|_{p - \varepsilon ; [0, T +
      \delta^l]}^{\frac{p - \varepsilon}{p}} \| x^{\delta^l} \circ \lambda^l -
      x^{\delta^l} \|_{0 ; [0, T + \delta^l]}^{\frac{\varepsilon}{p}} .
    \end{eqnarray*}
    Now notice that due to the continuity of $x$ and by construction of the
    $\delta^l$-extension $x^{\delta^l}$, there exists reparameterization
    $\varphi^l \in \Lambda_{[0, T] ; [0, T + \delta^l]}$ such that $\|
    \varphi^l - \tmop{id} \|_{\infty} \leq \delta^l$ and $x^{\delta^l} = x
    \circ \varphi^l$. \ Using the fact that the $p$-variation stays invariant
    under reparameterization and $\| \cdot \|_0 \leq 2 \| \cdot
    \|_{\infty}$, we have
    \[ \lim_{l \rightarrow \infty} \| x^{\delta^l} \circ \lambda^l -
       x^{\delta^l} \|_{p ; [0, T + \delta^l]} \leq 4 \| x  \|_{p -
       \varepsilon ; [0, T]}^{\frac{p - \varepsilon}{p}} \lim_{l \rightarrow
       \infty} \| x \circ \varphi^l \circ \lambda^l \circ (\varphi^l)^{- 1} -
       x  \|_{\infty ; [0, T]}^{\frac{\varepsilon}{p}} = 0, \]
    where the second term converges to zero due to $\| \varphi^l \circ \lambda^l
    \circ (\varphi^l)^{- 1} - \tmop{id} \|_{\infty} \leq 2 \delta^l + \|
    \lambda^l - \tmop{id} \|_{\infty} \xrightarrow{l \rightarrow \infty} 0$ and
    uniform continuity of $x$.
  \end{proof}
\end{lemma}

We are going to show stability for RBSDE solutions under the following assumption. \\
\tmtextbf{Assumption B:}
\begin{enumeratealpha}
  \item Let $q \in [1, 2)$, $p > 2$ with $\frac{1}{p} + \frac{1}{q} > 1$.
  
  \item $W^k$ is in $ D^q ([0, T], \mathbb{R}^e)$, $k \in \mathbb{N} \cup
  \{\infty\}$, and satisfies $\sup_{k \in \mathbb{N} \cup \{\infty\}} \| W^k \|_{q ; [0,
  T]} < \infty$;
  
  \item $\xi^k$ is in $ L^{\infty}$, $k \in \mathbb{N} \cup \{\infty\}$, and satisfies
  $\sup_{k \in \mathbb{N} \cup \{\infty\}} \| \xi^k \|_{L^\infty} < \infty$;
  
  \item Generator functions $f^k : \Omega \times [0, T] \times \mathbb{R}^h \times \mathbb{R}^{h
  \times d} \rightarrow \mathbb{R}^h$, $k \in \mathbb{N} \cup \{\infty\}$, are
  adapted to 
  $(\mathcal{F}_t)_{t \in [0, T]}$. There exists
  some constant $C_f > 0$ such that $\mathbb{P}$-a.s.
  \begin{eqnarray}
    \sup_{t \in [0, T]} |f^k (t, 0, 0) | & \leq & C_f \text { and} \nonumber\\
    |f^k (t, y, z) - f^k (t, y', z') | & \leq & C_f  (|y - y' | + |z - z' |) 
   \text{ for all $k \in
  \mathbb{N} \cup \{\infty\}$, $t \in [0, T]$;} \nonumber
  \end{eqnarray}
  \item $\sup_{t,y,z} | f^k (t, y, z) - f^{\infty} (t, y, z) |$ converges in probability to zero as $k \to \infty$;
  \item It holds $g^k \in D^{p, 2} C_b^2 (\mathbb{R}^h, \mathcal{L} (\mathbb{R}^e,
  \mathbb{R}^h))$, $k \in \mathbb{N} \cup \{\infty\}$, and there exists some
  constant $C_g > 0$ such that for all $k \in \mathbb{N} \cup \{\infty\}$, we have
  \begin{align} \label{assumptionB-g}
      \sup_{t \in [0, T]} \||g^k |_{C^2_b} \|_{L^\infty} \leq C_g, \quad
     [[g^k]]_{p, 2 ; [0, T]} \leq C_g, \quad [[\tmop{Dg}^k]]_{p, 2 ; [0,
     T]} \leq C_g. 
  \end{align}
    Additionally it holds $[[g^\infty]]_{p', 2 ; [0, T]} < \infty$ for some $p' < p$;
  \item Both $\sup_{t \in [0, T]} | \tmop{Dg}^k_t - \tmop{Dg}^{\infty}_t
  |_{\infty}$ and $\sup_{y \in \mathbb{W}} \|g^k (y) - g^{\infty}  (y) \|_{p ;
  [0, T]}$ converge in probability to zero for $k \rightarrow \infty$;
  
  \item the path trajectories of $g^{\infty}_{\cdot} (y, \omega)$ and
  $\tmop{Dg}^{\infty}_{\cdot} (y, \omega)$ are $\mathbb{P}$-a.s. uniformly continuous in $t$ uniformly in $y$, i.e. $\mathbb{P}$-almost
  surely it holds $\lim_{n \rightarrow \infty} \sup_{t \in [0, T]} \sup_{y \in
  \mathbb{W}} | g^{\infty}_{\lambda^n (t)} (y) - g_t^{\infty} (y) | = 0$ and
  $\lim_{n \rightarrow \infty} \sup_{t \in [0, T]} \sup_{y \in \mathbb{W}} |
  \tmop{Dg}^{\infty}_{\lambda^n (t)} (y) - \tmop{Dg}_t^{\infty} (y) | = 0$ for
  any $(\lambda^n)_{n \in \mathbb{N}} \subset \Lambda_{[0, T]}$ with $\lim_{n
  \rightarrow \infty} \| \lambda - \tmop{id} \|_{\infty} = 0$.
\end{enumeratealpha}
\begin{remark}
    Notice that in Assumption B.f we require slightly higher regularity of
    $g^\infty$, namely $[[g^\infty]]_{p',2;[0,T]}<\infty$ for some $p'<p$.
    This is due to the application of Lemma
    \ref{continuity_parametrization} in \eqref{g-convergence}, which entails
    a slight loss of regularity.
\end{remark}

\begin{theorem}
  \label{stability-RBSDE-decorated}Provided that Assumption B is satisfied, let $(Y^k,
  Z^k)$ for $k \in \mathbb{N} \cup \{\infty\}$ be the solution to the
  forward-type (or Marcus-type) RBSDE
  \begin{eqnarray*}
    Y^k_t & = & \xi^k + \int_t^T f (r, Y^k_r, Z^k_r) \tmop{dr} + \int_t^T g_r
    (Y^k_{r +})  (\diamond) \tmop{dW}^k_r - \int_t^T Z^k_r \tmop{dB}_r .
  \end{eqnarray*}
  For $k \in \mathbb{N} \cup \{\infty\}$, let \hyperlink{lift-Y}{$\mathbf{Y^k}$} denote  the
  lift of $Y^k$ to the space of decorated paths, as constructed in Chapter~\ref{rBSDEasDecoratedPath}, and define \hyperlink{embedding-i}{$\mathbf{W^k} = \imath W^k$} (or \hyperlink{embedding-j}{$\mathbf{W^k} = \jmath W^k$}, respectively).\\ 
  If $\lim_{k \rightarrow \infty}
  \alpha_{q ; [0, T]} (\mathbf{W^k}, \mathbf{W^{\infty}}) = 0$ and $\xi^k \to \xi$ in probability, then we have for any $\varepsilon > 0$ that
  \begin{align*}
      \lim_{k \rightarrow
  \infty} \mathbb{P} (\alpha_{p ; [0, T]} (\mathbf{Y^k},
  \mathbf{Y^{\infty}}) > \varepsilon) = 0 \quad \text{and} \quad
  \lim_{k \rightarrow \infty} \mathbb{E} \bigg[ \int_0^T (Z^k_r -
  Z^{\infty}_r)^2 \tmop{dr} \bigg] = 0.
  \end{align*}
\end{theorem}

\begin{proof}
  We shall start by fixing for all decorated paths the same countable set
  $\Pi \subset [0, T]$ which contains the stationary points (in the sense of Def.~\ref{def:decopaths}) of all $W^k$, $k
  \in \mathbb{N} \cup \{\infty\}$.
  
  By Lemma \ref{alternative-a-metric}, the above statement is equivalent to the following: \\
  Given any sequence of reparameterization $(\lambda^{k, \delta} \in
  \Lambda_{[0, T + \delta]})_{k \in \mathbb{N}}$ such that
  \begin{equation}
  \lim_{k \rightarrow \infty} \| c^\delta \circ \lambda^{k, \delta}
     - c^\delta \|_{\infty} \quad \vee \quad \|W^{k, \delta} \circ
     \lambda^{k, \delta} - W^{\infty, \delta} \|_{q ; [0, T + \delta]} = 0,
     \label{reparametrization-convergence}
  \end{equation}
  then for all $\varepsilon > 0$ it holds
  \begin{eqnarray}
    \lim_{k \rightarrow \infty} \mathbb{P} (
    \|Y^{k, \delta} \circ \lambda^{k, \delta} - Y^{\infty, \delta} \|_{p ; [0,
    T + \delta]} > \varepsilon) = 0, &  &  \label{continuity-rBSDE}
  \end{eqnarray}
  \begin{equation}
    \lim_{k \rightarrow \infty} \mathbb{E} \bigg[ \int_0^T (Z^k_r -
    Z^{\infty}_r)^2 \tmop{dr} \bigg] = \lim_{k \rightarrow \infty} \mathbb{E}
    \bigg[ \int_0^{T + \delta} (Z^{k,
    \delta}_r - Z^{\infty, \delta}_r)^2 \tmop{dc}^{\delta}_r \bigg] =
    0, \label{continuity-rbsde-1}
  \end{equation}
  where for any $k \in \mathbb{N} \cup \{\infty\}$ the pair $(Y^{k, \delta},
  Z^{k, \delta})$ satisfies the RBSDE
  \begin{align*}
    Y^{k, \delta}_t = \xi^k + \int_t^{T + \delta} \hat{f}^{k, \delta}_r ( Y^{k, \delta}_r, Z^{k, \delta}_r) \tmop{dc}^{\delta}_r +
    \int_t^{T + \delta} \hat{g}^{k, \delta}_r  (Y^{k, \delta}_r)
    \tmop{dW}^{k, \delta}_r - \int_t^{T + \delta} Z^{k, \delta}_r
    \tmop{dB}^{\delta}_r.
  \end{align*}
  For each $k \in \mathbb{N} \cup \{\infty\}$, we define a Picard
  iteration sequence, by letting $Y^{k, 0, \delta} \equiv 0$ and $ Z^{k, 0, \delta} \equiv 0$, and then define iteratively ($Y^{k, n+1, \delta}, Z^{k, n+1, \delta}$) for  $n\ge 0$ by
  \begin{align*}
    Y^{k, n + 1, \delta}_t = & \xi^k + \int_t^{T + \delta} \hat{f}^{k,
    \delta}_r (Y^{k, n, \delta}_r, Z^{k, n, \delta}_r)
    \tmop{dc}^{\delta}_r + \int_t^{T + \delta} \hat{g}^{k, \delta}_r 
    (Y^{k, n, \delta}_r) \tmop{dW}^{k, \delta}_r\\
    & - \int_t^{T + \delta} Z^{k, n + 1, \delta}_r
    \tmop{dB}^{\delta}_r .
  \end{align*}
  In order to show the convergence
  (\ref{continuity-rBSDE}-\ref{continuity-rbsde-1}), we begin with the
  following inequalities
  \begin{align*}
    & \lim_{k \rightarrow \infty} \mathbb{P} (
    \|Y^{k, \delta} \circ \lambda^{k, \delta} - Y^{\infty, \delta} \|_{p ; [0,
    T + \delta]} > \varepsilon)\\
    \leq & \lim_{n \rightarrow \infty} \mathbb{P} \bigg( \|Y^{\infty, \delta} - Y^{\infty, n, \delta} \|_{p
    ; [0, T + \delta]} > \frac{\varepsilon}{3} \bigg)\\
    & + \lim_{n \rightarrow \infty} \lim_{k \rightarrow \infty} \mathbb{P}
    \bigg(  \|Y^{k, n, \delta} \circ \lambda^{k, \delta} - Y^{\infty, n, \delta} \|_{p ; [0, T + \delta]} >
    \frac{\varepsilon}{3} \bigg)\\
    & + \lim_{n \rightarrow \infty} \sup_{k \in \mathbb{N}} \mathbb{P}
    \bigg(  \|Y^{k, n, \delta} - Y^{k, \infty,
    \delta} \|_{p ; [0, T + \delta]} > \frac{\varepsilon}{3} \bigg), 
    \end{align*}
    \begin{align*}
    \lim_{k \rightarrow \infty} \mathbb{E} \bigg[ \int_0^T (Z^{k, \delta}_r
    - Z^{\infty, \delta}_r)^2 \tmop{d}_r \bigg] \leq & \lim_{n
    \rightarrow \infty} \mathbb{E} \bigg[ \int_0^{T + \delta} (Z^{\infty,
    \delta}_r - Z^{\infty, n , \delta}_r)^2 \tmop{dc}^{\delta}_r
    \bigg]\\
    & + \lim_{n \rightarrow \infty} \lim_{k \rightarrow \infty} \mathbb{E}
    \bigg[ \int_0^{T + \delta} (Z^{k, n , \delta}_r - Z^{\infty, n ,
    \delta}_r)^2 \tmop{dc}^{\delta}_r \bigg]
    \\ & + \lim_{n \rightarrow \infty} \sup_{k \in \mathbb{N}} \mathbb{E}
    \bigg[ \int_0^{T + \delta} (Z^{\infty, n, \delta}_r - Z^{\infty,
    \delta}_r)^2 \tmop{dc}^{\delta}_r \bigg].
  \end{align*}
  We show that all three terms on the right side of both inequalities are zero.\\
  We start with the first and third terms of each. By the Markov inequality, we
  have
  \begin{eqnarray}
    &  & \lim_{n \rightarrow \infty} \sup_{k \in \mathbb{N} \cup \{\infty\}}
    \mathbb{P} ( \|Y^{k, n , \delta} - Y^{k,
    \infty, \delta} \|_{p ; [0, T + \delta]} > \varepsilon) \nonumber\\
    & \leq & \lim_{n \rightarrow \infty} \sup_{k \in \mathbb{N} \cup
    \{\infty\}} \mathbb{P} (\|Y^{k, n, \delta} - Y^{k, \infty, \delta}
    \|_{p ; [0, T + \delta]} > \varepsilon) \nonumber\\
    & \leq & \lim_{n \rightarrow \infty} \sup_{k \in \mathbb{N} \cup
    \{\infty\}}  \frac{1}{\varepsilon}  \|Y^{k, n, \delta} - Y^{k,
    \infty, \delta} \|_{p, 2 ; [0, T + \delta]} . \nonumber
  \end{eqnarray}
  In (\ref{Picard-Inv}, \ref{Picard-bound}) in the proof of Theorem
  \ref{Picard-iteration} we have shown that
  \begin{align*}
      \interleave Y^{k, n, \delta}, Z^{k, n, \delta} \interleave_{[0, T
  + \delta]} &\leq R_k, \\
  \interleave Y^{k, \infty, \delta} - Y^{k,
  n, \delta}, Z^{k, \infty, \delta} - Z^{k, n, \delta} \interleave_{[0,
  T + \delta]} &\leq R_k C_{N_k} \sum_{l = n}^{\infty} l^{N^k - 1} C,
  \end{align*}
   holds for any $n \in \mathbb{N}$, 
  where $C<1$, and 
  \begin{align*}
  R_k &\assign 21 M_k^{1 - p}  (4 C_y)^{M_k + 1}  (1 \vee \| \xi^k
  \|_{L^\infty} \vee C_g \|W^k \|_{q ; [0, T]}), \\
  C_{N_k} &\assign 2^{(p +
  2) (N_k - 1)}  (C + \kappa_k)^{N_k}  (C + \eta)^{N_k - 1} C^{- N_k}
  \end{align*}
  with $\kappa_k \assign C_g \|W^k \|_{q ; [0, T]} \vee C_g \exp (C_g \|W^k \|_{q ;[0, T]})  \|W^k \|_{q ; [0, T]}$ and $M_k, N_k \in \NN$ bounded \linebreak by $M_k \leq 1 + \max
  \{|c_T |, \|W^k \|_{q ; [0, T]} \} / \varepsilon_1$ (choice of
  $\varepsilon_1$ only depending on $C_f, C_g, p$), \linebreak and $N_k \leq 1 + \max \{|c_T |, \|W^k
  \|_{q ; [0, T]} \} / \varepsilon_2$ (the choice $\varepsilon_2$ only depends
  on $C_f, C_g, R_k, p$).\\
  Notice that $\Upsilon = \sup_{k \in \mathbb{N} \cup \{\infty\}} \Upsilon_k$
  exists for all $\Upsilon \in \{ C_N, N, \kappa, R, M \}$ due to the
  assumptions $\sup_{k \in \mathbb{N} \cup \{\infty\}} \| \xi^k \|_{L^\infty} <
  \infty$ and $\sup_{k \in \mathbb{N} \cup \{\infty\}} \|W^k \|_{q ; [0, T]} <
  \infty$. So we have
  \begin{align}
    \sup_{k \in \mathbb{N} \cup \{\infty\}} (\|Y^{k, n, \delta} \|_{p, 2 ;
    [0, T + \delta]} +\|Z^{k, n, \delta} \|_{\tmop{BMO} ; [0, T +
    \delta]}) \leq & R < \infty,  \label{bound-n-picard}\\
    \sup_{k \in \mathbb{N} \cup \{\infty\}} \interleave Y^{k, \infty,
    \delta} - Y^{k, n, \delta}, Z^{k, \infty, \delta} - Z^{k, n,
    \delta} \interleave_{[0, T + \delta]} \leq & RC_N \sum_{l =
    n}^{\infty} l^{N - 1} C, \nonumber
  \end{align}
  hence $\lim_{n \rightarrow \infty} \sup_{k \in \mathbb{N} \cup \{\infty\}}
  \interleave Y^{k, \infty, \delta} - Y^{k, n, \delta}, Z^{k, \infty,
  \delta} - Z^{k, n, \delta} \interleave_{[0, T + \delta]} = 0$.\\
  We are only left to show the following iteratively over $n$
  \begin{align}
    \lim_{k \rightarrow \infty} \mathbb{P} \bigg(
     \|Y^{k, n, \delta} \circ \lambda^{k, \delta} -
    Y^{\infty, n, \delta} \|_{p ; [0, T + \delta]} > \frac{\varepsilon}{3}
    \bigg) = 0, \label{continuity-Y} \\
    \lim_{k \rightarrow \infty} \mathbb{E} \bigg[  \int_0^{T + \delta}
    (Z^{k, n, \delta}_r - Z^{\infty, n, \delta}_r)^2
    \tmop{dc}^{\delta}_r \bigg] = 0. \label{continuity-Z}
  \end{align}
  For $Y^{k, 0, \delta} \equiv Y^{\infty, 0, \delta} \equiv 0$ and
  $Z^{k, 0, \delta} \equiv Z^{\infty, 0, \delta} \equiv 0$ the convergences in
  (\ref{continuity-Y}) and (\ref{continuity-Z}) clearly hold. Assuming (\ref{continuity-Y}-\ref{continuity-Z}) hold for some $n \in
  \mathbb{N}$, we want to show they then also hold for $n + 1$. This means that it suffices to show \eqref{continuity-Z} together with the convergence in probability to zero as $k \rightarrow \infty$ for the
  following terms:
  \begin{flalign*}
    \tmop{a} . &  \bigg\| \int_{\lambda^{k, \delta}_{\cdot}}^{T + \delta} \hat{f}^{k, \delta}_r (Y^{k, n, \delta}_r,
    Z^{k, n, \delta}_r) \tmop{dc}^{\delta}_r - \int_{\cdot}^{T + \delta}
    \hat{f}^{\infty, \delta}_r (Y^{\infty, n, \delta}_r, Z^{\infty, n,
    \delta}_r) \tmop{dc}^{\delta}_r \bigg\|_{p ; [0, T + \delta]};&\\
    \tmop{b} . &  \bigg\| \int_{\lambda^{k, \delta}_{\cdot}}^{T + \delta} \hat{g}^{k, \delta} (r, Y^{k, n, \delta}_r)
    \tmop{dW}^{k, \delta}_r - \int_{\cdot}^{T + \delta} \hat{g}^{\infty,
    \delta} (r, Y^{\infty, n, \delta}_r) \tmop{dW}^{\infty, \delta}_r
    \bigg\|_{p ; [0, T + \delta]}; &\\
    \tmop{c} . &  \bigg\| \int_{\lambda^{k, \delta}_{\cdot}}^{T + \delta} Z^{k, n + 1, \delta}_r \tmop{dB}^{\delta}_r
    - \int_{\cdot}^{T + \delta} Z^{\infty, n + 1, \delta}_r
    \tmop{dB}^{\delta}_r \bigg\|_{p ; [0, T + \delta]}. &
  \end{flalign*}
  To show the convergence for \tmtextbf{a.}, we show the convergence of the following three terms:
  \begin{align*}
    &  \bigg\| \int_{\lambda^{k, \delta}_{\cdot}}^{T
    + \delta} \hat{f}^{k, \delta}_r (Y^{k, n, \delta}_r, Z^{k, n,
    \delta}_r) \tmop{dc}^{\delta}_r - \int_{\cdot}^{T + \delta}
    \hat{f}^{\infty, \delta}_r (Y^{\infty, n, \delta}_r, Z^{\infty, n,
    \delta}_r) \tmop{dc}^{\delta}_r \bigg\|_{p ; [0, T + \delta]}\\
    \leq &  \bigg\| \int_{\lambda^{k, \delta}_{\cdot}}^{T + \delta} \hat{f}^{k, \delta}_r (Y^{k, n, \delta}_r,
    Z^{k, n, \delta}_r) - \hat{f}^{\infty, \delta}_r (Y^{k, n,
    \delta}_r, Z^{k, n, \delta}_r) \tmop{dc}^{\delta}_r \bigg\|_{p ;
    [0, T + \delta]}\\
    & +  \bigg\| \int_{\lambda^{k, \delta}_{\cdot}}^{T + \delta} f^{\infty} (c^{\delta}_r, Y^{k, n,
    \delta}_r, Z^{k, n, \delta}_r) - f^{\infty} (c^{\delta}_r,
    Y^{\infty, n, \delta}_r, Z^{\infty, n, \delta}_r)
    \tmop{dc}^{\delta}_r \bigg\|_{p ; [0, T + \delta]}\\
    & +  \bigg\| \int_{\lambda^{k, \delta}_{\cdot}}^{\cdot} \hat{f}^{\infty, \delta}_r ( Y^{\infty, n,
    \delta}_r, Z^{\infty, n, \delta}_r) \tmop{dc}^{\delta}_r \bigg\|_{p
    ; [0, T + \delta]},
  \end{align*}
  we denote the three terms on the right-hand side by \textbf{a.1}, \textbf{a.2}, and \textbf{a.3}, respectively. We start with the term \textbf{a.3}. For any $t\in[0,T+\delta]$, since $c^\delta$ is
constant on the inserted fictitious intervals, the values of the integrand on
these intervals do not contribute to the integral. Hence
  \begin{align*}
      \int_{\lambda^{k, \delta}_{t}}^{t} \hat{f}^{\infty, \delta}_r ( Y^{\infty, n, \delta}_r, Z^{\infty, n, \delta}_r) \tmop{dc}^{\delta}_r = \int_{\lambda^{k, \delta}_{t}}^{t} f^{\infty}_{c^{\delta}_r} ( Y^{\infty, n}_{c^{\delta}_r}, Z^{\infty, n}_{c^{\delta}_r}) \tmop{dc}^{\delta}_r = \int_{c^{\delta} (\lambda^{k, \delta}_{t})}^{c^{\delta}_t} f^{\infty}_{r} ( Y^{\infty, n}_{r}, Z^{\infty, n}_{r}) \tmop{dr},
  \end{align*}
  where the second equality follows from the substitution rule for the
non-decreasing map $c^\delta$. Since $\lim_{k\to\infty}|c^\delta\circ\lambda^k-c^\delta|_\infty=0$ and the map
$t \mapsto \int_0^t \cdots \tmop{dr}$ being of
bounded variation, Lemma~\ref{continuity_parametrization} implies that the
last term converges to zero almost surely. \\
  We also have the convergence in probability to zero of the term \textbf{a.1}:
  \begin{align*}
    & \lim_{k \rightarrow \infty} \mathbb{P} \bigg( \bigg\|
    \int_{\lambda^{k, \delta}_{\cdot}}^{T + \delta} \hat{f}^{k, \delta}_r (
    Y^{k, n, \delta}_r, Z^{k, n, \delta}_r) - \hat{f}^{\infty, \delta}_r
    (Y^{k, n, \delta}_r, Z^{k, n, \delta}_r) \tmop{dc}^{\delta}_r
    \bigg\|_{p ; [0, T + \delta]} > \epsilon \bigg)\\
    \leq & \lim_{k \rightarrow \infty} \mathbb{P} \bigg(
    \int_0^{T + \delta} | \hat{f}^{k, \delta}_r ( Y^{k, n, \delta}_r,
    Z^{k, n, \delta}_r) - \hat{f}^{\infty, \delta}_r ( Y^{k, n,
    \delta}_r, Z^{k, n, \delta}_r) | \tmop{dc}^{\delta}_r > \epsilon \bigg)\\
    = & \lim_{k \rightarrow \infty}  \mathbb{P} \bigg( \sup_{r,y,z} | f^k
    (r, y,z) - f^{\infty}
    (r, y,z) |  > \frac{\epsilon}{T + \delta}  \bigg) \quad = \quad 0.
  \end{align*}
  To prove that the term \textbf{a.2} converges to zero in probability, we first note that
  \begin{align*}
    & \mathbb{P} \bigg(  \bigg\|
    \int_{\lambda^{k, \delta}_{\cdot}}^{T + \delta} \hat{f}^{\infty, \delta}
    (r, Y^{k, n, \delta}_r, Z^{k, n, \delta}_r) - \hat{f}^{\infty,
    \delta} (r, Y^{\infty, n, \delta}_r, Z^{\infty, n, \delta}_r)
    \tmop{dc}^{\delta}_r \bigg\|_{p ; [0, T + \delta]} > \varepsilon
    \bigg)\\
    \leq & \mathbb{P} \bigg(  \int_0^{T +
    \delta} |f (c^{\delta}_r, Y^{k, n, \delta}_r, Z^{k, n, \delta}_r)
    - f (c^{\delta}_r, Y^{\infty, n, \delta}_r, Z^{\infty, n, \delta}_r)
    | \tmop{dc}^{\delta}_r > \varepsilon \bigg)\\
    \leq & \mathbb{P} \bigg(  \int_0^{T +
    \delta} C_f |Y^{k, n, \delta}_r - Y^{\infty, n, \delta}_r | + C_f
    |Z^{k, n, \delta}_r - Z^{\infty, n, \delta}_r | \tmop{dc}^{\delta}_r
    > \varepsilon \bigg)\\
    \leq & \mathbb{P} \bigg(  \int_0^{T +
    \delta} C_f |Y^{k, n, \delta}_r - Y^{\infty, n, \delta}_r | \tmop{dc}^{\delta}_r
    > \frac{\varepsilon}{2} \bigg) + \frac{4 C_f^2}{\varepsilon^2} \mathbb{E} \bigg[ \int_0^T |Z^{k, n}_r -
    Z^{\infty, k}_r |^2 \tmop{dr} \bigg] 
  \end{align*}
  where the term involving $Z$ converges to zero by the induction assumption.
The term involving $Y$ requires an additional argument. We have
  \begin{align*}
     & \mathbb{P} \bigg(  \int_0^{T +
    \delta} C_f |Y^{k, n, \delta}_r - Y^{\infty, n, \delta}_r | \tmop{dc}^{\delta}_r
    > \frac{\varepsilon}{2} \bigg) \\
    \leq & \mathbb{P} \bigg(  \int_0^{T +
    \delta} |Y^{k, n,
    \delta}_r - (Y^{\infty, n, \delta}\circ (\lambda^{k, \delta})^{- 1})_r | 
    + |(Y^{\infty, n,
    \delta}\circ (\lambda^{k, \delta})^{- 1})_r - Y^{\infty, n, \delta}_r | 
    \tmop{dc}^{\delta}_r
    > \frac{\varepsilon}{2 C_f} \bigg) \\
    \leq & \mathbb{P} \bigg( \|Y^{k, n,
    \delta} - Y^{\infty, n, \delta} \circ (\lambda^{k, \delta})^{- 1} \|_{p;[0,T+\delta]} 
    + |\xi^k - \xi|> \frac{\varepsilon}{4 (T +
    \delta) C_f} \bigg) \\
    &+ \frac{16C_f^2 }{\varepsilon^2 } \mathbb{E} \bigg[  \int_0^{T +
    \delta} |(Y^{\infty, n,
    \delta}\circ (\lambda^{k, \delta})^{- 1})_r - Y^{\infty, n, \delta}_r | \tmop{dc}^{\delta}_r \bigg],
  \end{align*}
the first term converges to zero by the induction assumption and Assumption~B.
For the second term, we distinguish between the forward and Marcus jump cases.
In the forward jump case, we have
$Y^{\infty,n,\delta}=Y^{\infty,n} \circ {c^\delta}$ by Theorem \ref{fictioustime_Forward-RBSDE}. Consider the set $N:=\{r \in [0,T + \delta]: c^\delta_r \notin \mathcal{C}(W^\infty)\}$, it holds for all $r \in [0,T+\delta]\setminus N$ that
\begin{align*}
    \lim_{k \to \infty}|Y^{\infty, n,
    \delta}\circ (\lambda^{k, \delta})^{- 1})_r - Y^{\infty, n, \delta}_r| = \lim_{k \to \infty}|(Y^{\infty,n} \circ {c^\delta} \circ (\lambda^{k, \delta})^{- 1})_r - Y^{\infty, n}\circ {c^\delta_r}| = 0,
\end{align*}
due to $\lim_{k \to \infty}|c^\delta \circ (\lambda^{k, \delta})^{-1} - c^\delta|_\infty = \lim_{k \to \infty}|c^\delta - c^\delta \circ \lambda^{k, \delta}|_\infty = 0$ and the continuity of $Y^{\infty, n}$ at all such $r$. Notice that the set $N$ is a Nullset in the measure induced by $c^\delta$, hence by dominated convergence (of the Lebesgue–Stieltjes integral and expectation) we have
\begin{align*}
    &\lim_{k \to \infty}\mathbb{E} \bigg[  \int_0^{T +
    \delta} |(Y^{\infty, n,
    \delta}\circ (\lambda^{k, \delta})^{- 1})_r - Y^{\infty, n, \delta}_r | \tmop{dc}^{\delta}_r \bigg] \\
    =& \mathbb{E} \bigg[   \int_0^{T +
    \delta} \lim_{k \to \infty} |(Y^{\infty, n,
    \delta}\circ (\lambda^{k, \delta})^{- 1})_r - Y^{\infty, n, \delta}_r | \tmop{dc}^{\delta}_r \bigg] = 0.
\end{align*}
In the Marcus jump case, recall the reparametrization
$\phi^{\bar{\delta},\delta}\in
\Lambda_{[0,T+\delta];[0,T+\bar{\delta}]}$ from
Lemma~\ref{alternative-a-metric}. Applying the lemma gives
\begin{align*}
    & |(Y^{\infty, n,\delta}\circ (\lambda^{k,\delta})^{-1})_r
    - Y^{\infty, n,\delta}_r| \\
    ={}&
    |(Y^{\infty,n,\bar{\delta}} \circ \phi^{\bar{\delta},\delta}
    \circ (\phi^{\bar{\delta},\delta})^{-1}
    \circ (\lambda^{k,\bar{\delta}})^{-1}
    \circ (\phi^{\delta,\bar{\delta}})^{-1})_r
    - (Y^{\infty,n,\bar{\delta}}\circ \phi^{\bar{\delta},\delta})_r| \\
    ={}&
    |(Y^{\infty,n,\bar{\delta}}
    \circ (\lambda^{k,\bar{\delta}})^{-1}
    \circ \phi^{\bar{\delta},\delta})_r
    - (Y^{\infty,n,\bar{\delta}}\circ \phi^{\bar{\delta},\delta})_r|.
\end{align*}
Together with $c^{\delta}=c^{\bar{\delta}}\circ\phi^{\bar{\delta},\delta}$,
the substitution rule gives
\begin{align*}
    &\mathbb{E} \bigg[  \int_0^{T+\delta}
    |(Y^{\infty, n,\delta}\circ (\lambda^{k,\delta})^{-1})_r
    - Y^{\infty, n,\delta}_r | \tmop{dc}^{\delta}_r \bigg] \\
    ={}&
    \mathbb{E} \bigg[  \int_0^{T+\bar{\delta}}
    |(Y^{\infty, n,\bar{\delta}}\circ (\lambda^{k,\bar{\delta}})^{-1})_r
    - Y^{\infty, n,\bar{\delta}}_r |
    \tmop{dc}^{\bar{\delta}}_r \bigg]
\end{align*}
for arbitrary $\bar{\delta}>0$. Hence
\begin{align*} &\lim_{k \to \infty}\mathbb{E} \bigg[ \int_0^{T + \delta} |(Y^{\infty, n, \delta}\circ (\lambda^{k, \delta})^{- 1})_r - Y^{\infty, n, \delta}_r | \tmop{dc}^{\delta}_r \bigg] \\ =& \mathbb{E} \bigg[ \lim_{k \to \infty} \lim_{\delta \to 0} \int_0^{T + \delta} |(Y^{\infty, n, \delta}\circ (\lambda^{k, \delta})^{- 1})_r - Y^{\infty, n, \delta}_r | \tmop{dc}^{\delta}_r \bigg] \\ \leq& \mathbb{E} \bigg[ \lim_{k \to \infty} \lim_{\delta \to 0} (T + \delta) \|(Y^{\infty, n, \delta}\circ (\lambda^{k, \delta})^{- 1})_r - Y^{\infty, n, \delta}_r \|_\infty \bigg] = 0, \end{align*}
where we used Lemma~\ref{continuity_parametrization} in the last line.
Combining the estimates for \textbf{a.1}, \textbf{a.2}, and \textbf{a.3}, we obtain convergence in probability of a.
  
  For showing the convergence for \tmtextbf{b.}, we start by applying Proposition \ref{stability_youngIntergral} to get
  \begin{align*}
    &   \bigg\|
    \int_{\lambda^{k, \delta}_{\cdot}}^{T + \delta} \hat{g}^{k, \delta} (r,
    Y^{k, n, \delta}_r) \tmop{dW}^{k, \delta}_r - \int_{\cdot}^{T +
    \delta} \hat{g}^{\infty, \delta} (r, Y^{\infty, n, \delta}_r)
    \tmop{dW}^{\infty, \delta}_r \bigg\|_{p ; [0, T + \delta]} \\
    = &  \bigg\| \int^{T +
    \delta}_{\cdot} \hat{g}^{k, \delta} (\lambda^{k, \delta}_r, Y^{k, n,
    \delta}_{\lambda^{k, \delta}_r}) \tmop{dW}^{k, \delta}_{\lambda^{k, \delta}_r} -
    \int_{\cdot}^{T + \delta} \hat{g}^{\infty, \delta} (r, Y^{\infty, n,
    \delta}_r) \tmop{dW}^{\infty, \delta}_r \bigg\|_{p ; [0, T +
    \delta]} \\
    \leq &   \|g^k
    (\lambda^{k, \delta}  \circ c^{\delta} , Y^{k, n, \delta}_{\lambda^{k, \delta}})
    - g^{\infty} (c^{\delta} , Y^{\infty, n, \delta})\|_{p ; [0, T +
    \delta]} \|W^{k, \delta} \circ \lambda^{k, \delta} \|_{q ; [0, T +
    \delta]} \\
    & +  (\| \hat{g}^{\infty,
    \delta} (r, Y^{\infty, n, \delta}_r)\|_{p ; [0, T + \delta]} +
    \sup_{t \in [0, T]} \||g^{\infty} |_{C^2_b} \|_{L^\infty})\|W^{k, \delta}
    \circ \lambda^{k, \delta} - W^{\infty, \delta}_r \|_{q ; [0, T + \delta]},
  \end{align*}
  the second term converges to $0$ as $k\to \infty$ by assumption. To show that the first term converges in probability to $0$, first notice that for all $k \in \mathbb{N} \cup \{
  \infty \}$ and $j=0,1,2$ it holds
  \begin{align*}
  \sup_{s \in [0, T + \delta]} |\tmop{D}^j g^k  (\lambda^{k, \delta}  \circ c^{\delta}_s, \cdot) |_{\infty} &= \sup_{s \in [0, T]} |\tmop{D}^j g^k_s
  (\cdot) |_{\infty}, \\
      \sup_{y \in \mathbb{W}} \|\tmop{D}^j g^k 
  (\lambda^{k, \delta}  \circ c^{\delta} , y)\|_{p ; [0, T + \delta]} &= \sup_{y
  \in \mathbb{W}} \|\tmop{D}^j g^k_{\cdot} (\omega, y)\|_{p ; [0, T]}.
  \end{align*}
  Then by applying Lemma \ref{composition} to
  $g^k (\lambda^{k, \delta}  \circ c^{\delta} , \cdot)$ and $g^{\infty} 
  (c^{\delta} , \cdot)$, we get the inequality 
  \begin{align}
    &  \|g^k (\lambda^{k, \delta}  \circ
    c^{\delta} , Y^{k, n, \delta}_{\lambda^{k, \delta}}) - g^{\infty}
    (c^{\delta} , Y^{\infty, n, \delta})\|_{p ; [0, T + \delta]} \nonumber\\
    \leq &  2 C_g \| Y^{k, n,
    \delta}_{\lambda^{k, \delta}} - Y^{\infty, n, \delta} \|_{p ; [0, T +
    \delta]} \label{stability-g-1}\\
    & +  C_g (\| Y^{k, n, \delta}  \|_{p ;
    [0, T]} + \| Y^{\infty, n, \delta} \|_{p ; [0, T]}) (\| Y^{k, n,
    \delta}_{\lambda^{k, \delta}} - Y^{\infty, n, \delta} \|_{p ; [0, T +
    \delta]} + \|\xi^k - \xi\|_\infty)\label{stability-g-2}\\
    & +  \sup_{y \in \mathbb{W}} \|g^k
    (c^{\delta}  \circ \lambda^{k, \delta} , y) - g^{\infty}  (c^{\delta} , y)
    \|_{p ; [0, T + \delta]} \label{stability-g-3}\\
    & +  \sup_{s \in [0, T]} \big|
    g^k_{c^{\delta}  \circ \lambda^{k, \delta}_s} - g^{\infty}_{c^{\delta}_s}
    \big|_{\tmop{Lip}} \|Y^{\infty, n, \delta} \|_{p ; [0, T + \delta]} \label{stability-g-4}
  \end{align}
  Take $k$ to $\infty$, \eqref{stability-g-1} converges in probability to $0$ by the induction assumption. The terms (\ref{stability-g-3}-\ref{stability-g-4}) converge in probability to $0$ by Assumption B and Lemma \ref{continuity_parametrization}
  \begin{align} \label{g-convergence}
     &  \sup_{y \in \mathbb{W}} \|g^k
    (c^{\delta}  \circ \lambda^{k, \delta} , y) - g^{\infty}  (c^{\delta} , y)
    \|_{p ; [0, T + \delta]} \\
    \leq & \sup_{y \in \mathbb{W}} \|g^k (y) - g^{\infty}  (y) \|_{p ; [0,
    T]} +  \sup_{y \in \mathbb{W}}
    \|g^{\infty}_{c^{\delta}  \circ \lambda^{k, \delta} } (y) -
    g^{\infty}_{c^{\delta} } (y) \|_{p ; [0, T + \delta]} \xrightarrow{k \rightarrow \infty} 0, \nonumber\\
     &  \sup_{s \in [0, T]} \big|
    g^k_{c^{\delta}  \circ \lambda^{k, \delta}_s} - g^{\infty}_{c^{\delta}_s}
    \big|_{\tmop{Lip}} \|Y^{\infty, n, \delta} \|_{p ; [0, T + \delta]} \nonumber \\
    \leq & \sup_{s \in [0, T]} \big| \tmop{Dg}^k  {- \tmop{Dg}^{\infty}_s} 
    \big|_{\infty} +  \sup_{s \in [0, T]} \big|
    \tmop{Dg}^{\infty}_{c^{\delta}  \circ \lambda^{k, \delta}_s} -
    \tmop{Dg}^{\infty}_{c^{\delta}_s} \big|_{\infty} \|Y^{\infty, n} \|_{p
    ; [0, T ]} \xrightarrow{k \rightarrow \infty} 0. \nonumber
  \end{align}
  We are left to show the convergence \eqref{stability-g-2}. We have for any $\varepsilon, A > 0$ that
  \begin{align*}
    & \mathbb{P} ( C_g (\| Y^{k, n, \delta} 
    \|_{p ; [0, T]} + \| Y^{\infty, n, \delta} \|_{p ; [0, T]}) (\| Y^{k, n,
    \delta}_{\lambda^{k, \delta}} - Y^{\infty, n, \delta} \|_{p ; [0, T +
    \delta]} + \|\xi^k - \xi\|_\infty) > \varepsilon)\\
    \leq & \sup_{k \in \mathbb{N} \cup \{ \infty \}} \mathbb{P} \bigg( 
    \| Y^{k, n, \delta}  \|_{p ; [0, T]} + \| Y^{\infty, n, \delta} \|_{p ; [0, T]} > \frac{A}{C_g} \bigg) \\
    &+\mathbb{P} \bigg(  (\| Y^{k, n,
    \delta}_{\lambda^{k, \delta}} - Y^{\infty, n, \delta} \|_{p ; [0, T +
    \delta]} + \|\xi^k - \xi\|_\infty) > \frac{\varepsilon}{A} \bigg),
  \end{align*}
  the second term converges to zero as $k \rightarrow \infty$ by induction
  assumption and the first converges to zero as we take $A \rightarrow \infty$,
  since $\sup_{k \in \mathbb{N} \cup \{ \infty \}}  \mathbb{E} [\| Y^{k, n,
  \delta}  \|_{p ; [0, T]}] < \infty$ by (\ref{bound-n-picard}).
  
  To show the convergence \eqref{continuity-Z}, we start by defining
  \[ M^{k, n + 1, \delta} \assign \xi^k + \int_0^{T + \delta} \hat{f}^{k,
     \delta} (r, Y^{k, n, \delta}_r, Z^{k, n, \delta}_r)
     \tmop{dc}^{\delta}_r + \int_0^{T + \delta} \hat{g}^{k, \delta}  (r,
     Y^{k, n, \delta}_r) \tmop{dW}^{k, \delta}_r \]
  for any $k \in \mathbb{N} \cup \{\infty\}$, that is $M^{k, n + 1,
  \delta} = Y^{k, n + 1, \delta}_0 + \int_0^{T + \delta} Z^{k, n + 1,
  \delta}_r \tmop{dB}^{\delta}_r$ by construction. We show
  \begin{enumerateroman}
    \item $\displaystyle \sup_{k \in \mathbb{N}}
    \mathbb{E} [|M^{k,n+1,\delta}|^4] < \infty,$
    
    \item $\lim_{k \rightarrow \infty} \mathbb{P} (
    |M^{k, n + 1, \delta} - M^{\infty, n + 1, \delta} | > \varepsilon)$=0
    for all $\varepsilon > 0$,
  \end{enumerateroman}
  Condition i) implies uniform integrability of
  $\{|M^{k,n+1,\delta}|^2:k\in\mathbb{N}\}$, which together with ii), by
  applying Vitali's convergence theorem, implies that
  \begin{equation}
    \lim_{k \rightarrow \infty} \mathbb{E} [ (M^{k,
    n + 1, \delta} - M^{\infty, n + 1, \delta})^2] = 0.
    \label{continuity-M}
  \end{equation}
  We postpone the proof of i) and ii) for later. By It{\^o} isometry and all $\mathcal{F}_0$-measurable random variable (in particular $Y^{k, n + 1,
  \delta}_0$) being almost surely constant, we get the
  convergence \eqref{continuity-Z}
  \begin{align*}
     & \mathbb{E} \bigg[ \int_0^{T + \delta} (Z^{k, n + 1, \delta}_r -
    Z^{\infty, n + 1, \delta}_r)^2 \tmop{dc}^{\delta}_r \bigg]\\
    \leq &  \mathbb{E} \bigg[ \big(Y^{k,
    n + 1, \delta}_0 + Y^{\infty,
    n + 1, \delta}_0 \big)^2 + \bigg(\int_0^{T + \delta} (Z^{k, n + 1, \delta}_r -
    Z^{\infty, n + 1, \delta}_r) \tmop{dB}^{\delta}_r \bigg)^2 \bigg]\\
    \leq &  \mathbb{E} \bigg[ \bigg( Y^{k,
    n + 1, \delta}_0 + Y^{\infty,
    n + 1, \delta}_0 + \int_0^{T + \delta} (Z^{k, n + 1, \delta}_r -
    Z^{\infty, n + 1, \delta}_r) \tmop{dB}^{\delta}_r \bigg)^2 \bigg]\\
     = & \mathbb{E} [ (M^{k, n + 1, \delta} -
    M^{\infty, n + 1, \delta})^2] \xrightarrow{k \rightarrow \infty} 0,
  \end{align*}
  where in the last equality, we can pull the limit inside due to i). \\
  For i), we get
  \begin{align*}
    \sup_{k\in\mathbb{N}}
    \mathbb{E}\bigg[
      \bigg|\int_0^{T+\delta}Z_r^{k,n+1,\delta}
      \tmop{dB}^{\delta}_r\bigg|^4
    \bigg]
    \lesssim{}&\sup_{k\in\mathbb{N}}
    \mathbb{E}\bigg[
      \bigg(\int_0^{T+\delta}|Z_r^{k,n+1,\delta}|^2
      \tmop{dc}^{\delta}_r\bigg)^2
    \bigg]\\
    \lesssim{}&\sup_{k\in\mathbb{N}}
    \|Z^{k,n+1,\delta}\|_{\tmop{BMO};[0,T+\delta]}^4<\infty.
  \end{align*}
  Here we applied the Burkholder--Davis--Gundy inequality in the first
  inequality, the standard energy estimate for BMO martingales in the second and
  (\ref{bound-n-picard}) in the last. For the reader's convenience, we
  briefly recall the short proof of the energy estimate. Set
  $A_t^k\assign\int_0^t|Z_r^{k,n+1,\delta}|^2\tmop{dc}_r^\delta$, by
  integration by parts and the tower property we get
  \[
    \mathbb{E}[(A_{T+\delta}^k)^2]=2\mathbb{E}\bigg[\int_0^{T+\delta}(A_{T+\delta}^k-A_t^k)\tmop{dA}_t^k\bigg]=2\mathbb{E}\bigg[\int_0^{T+\delta}\mathbb{E}[A_{T+\delta}^k-A_t^k\mid\mathcal{F}_t]\tmop{dA}_t^k\bigg].
  \]
  Hence, by the definition of the BMO norm, we get the desired energy estimate
  \[
    \mathbb{E}[(A_{T+\delta}^k)^2]\leq
    2\|Z^{k,n+1,\delta}\|_{\tmop{BMO};[0,T+\delta]}^2
    \mathbb{E}[A_{T+\delta}^k]
    \leq2\|Z^{k,n+1,\delta}\|_{\tmop{BMO};[0,T+\delta]}^4.
  \]
  Moreover, the uniform in $k$ bound of $\|Y^{k,n+1,\delta}\|_{p,2;[0,T+\delta]}$ from (\ref{bound-n-picard}) and the uniform bound on the terminal
  conditions together imply $\sup_{k\in\mathbb{N}}
    |Y_0^{k,n+1,\delta}|<\infty$. By the definition of $M^{k,n+1,\delta}$, condition i) follows.
  As for ii), we have done all the hard work in a. and b. The trivial bound
  $| x_{0, T} | \leq \| x \|_{p ; [0, T]} + | x_0 |$ gives us
  \[  \bigg| \int_0^{T + \delta} \hat{f}^{k,
     \delta} (r, Y^{k, n, \delta}_r, Z^{k, n, \delta}_r)
     \tmop{dc}^{\delta}_r - \int_0^{T + \delta} \hat{f}^{\infty, \delta}
     (r, Y^{\infty, n, \delta}_r, Z^{\infty, n, \delta}_r)
     \tmop{dc}^{\delta}_r \bigg| \leq \text{a.1}. \]
  This converges to zero in probability as $k\to\infty$, as shown above. The same holds true for the Young integral term. Combining them
  gives us the convergence in ii).
  
  Finally, to show the claim for \tmtextbf{c.}, it holds
  \begin{align*}
    & \bigg\| \int_{\lambda^{k, \delta}_{\cdot}}^{T +
    \delta} Z^{k, n + 1, \delta}_r \tmop{dB}^{\delta}_r -
    \int_{\cdot}^{T + \delta} Z^{\infty, n + 1, \delta}_r
    \tmop{dB}^{\delta}_r \bigg\|_{p ; [0, T + \delta]} 
    \\
    \leq &  \bigg\| \int_{\lambda^{k, \delta}_{\cdot}}^{\cdot} Z^{\infty, n + 1, \delta}_r \tmop{dB}^{\delta}_r
    \bigg\|_{p ; [0, T + \delta]} 
    +\bigg\| \int_{\lambda^{k, \delta}_{\cdot}}^{T +
    \delta} (Z^{k, n + 1, \delta}_r - Z^{\infty, n + 1, \delta}_r)
    \tmop{dB}^{\delta}_r \bigg\|_{p ; [0, T + \delta]}.
  \end{align*}
  The first term converges almost surely by Lemma \ref{continuity_parametrization}, while for the second term we have
  \begin{eqnarray*}
    &  & \mathbb{P} \bigg( \bigg\| \int_{\lambda^{k, \delta}_{\cdot}}^{T +
    \delta} (Z^{k, n + 1, \delta}_r - Z^{\infty, n + 1, \delta}_r)
    \tmop{dB}^{\delta}_r \bigg\|_{p ; [0, T + \delta]} >
    \frac{\varepsilon}{2} \bigg)\\
    & \leq & \frac{4}{\varepsilon^2} \mathbb{E} \bigg[ \bigg\|
    \int_{\cdot}^{T + \delta} (Z^{k, n + 1, \delta}_r - Z^{\infty, n + 1,
    \delta}_r) \tmop{dB}^{\delta}_r \bigg\|_{p ; [0, T + \delta]}^2
    \bigg]\\
    & \leq & \frac{4 C^p}{\varepsilon^2} \mathbb{E} \bigg[ \int_0^T
    (Z^{k, n + 1, \delta}_r - Z^{\infty, n + 1, \delta}_r)^2
    \tmop{dc}^{\delta}_r \bigg] \xrightarrow{k \rightarrow \infty, l
    \longrightarrow \infty} 0
  \end{eqnarray*}
  by applying Markov and Burkholder inequalities. 
\end{proof}

Under the stronger assumption that $W^k$ converges to $W^\infty$ directly
in $q$-variation, one also obtains the convergence of $Y^k$ to $Y^\infty$
directly in $p$-variation, without embedding the paths into the space of
decorated paths. Of course, convergence in variation topology is rather
restrictive for discontinuous paths, since it essentially requires the jump
times of $W^k$ and $W^\infty$ to be aligned. Nevertheless, we believe that
the following result may be useful even in the special case $W^k=W^\infty$,
as the terminal conditions $\xi^k$ and the coefficients $f^k$ and $g^k$ are
still allowed to vary.

\begin{corollary}
  In the setting of Theorem~\ref{stability-RBSDE-decorated}, suppose in
  addition that
  \[
    \lim_{k \rightarrow \infty}
    \|W^k-W^\infty\|_{q;[0,T]}=0.
  \]
  Then, for every $\varepsilon>0$,
  \[
    \lim_{k \rightarrow \infty}
    \mathbb{P}\left(
      \|Y^k-Y^\infty\|_{p;[0,T]}>\varepsilon
    \right)=0
    \quad\text{and}\quad
    \lim_{k \rightarrow \infty}
    \mathbb{E}\left[
      \int_0^T (Z^k_r-Z^\infty_r)^2\,\tmop{dr}
    \right]=0.
  \]
\end{corollary}

\begin{proof}
  In this case the reparameterizations in
  \eqref{reparametrization-convergence} can simply be chosen as
  $\lambda^{k,\delta}=\tmop{id}_{[0,T+\delta]}$. The result then follows
  directly from \eqref{continuity-rBSDE} and
  \eqref{continuity-rbsde-1} with this choice of $\lambda^{k,\delta}$.
\end{proof}

\section{BDSDE}

One of our main motivations for studying the RBSDEs is to study the new type
of backward Doubly SDEs (BDSDEs) as presented in \eqref{Marcus-BDSDE}. In
spirit, BDSDE can be seen as an annealed or randomized version of the RBSDE.
In the RBSDE, the rough path $W$ is treated as a frozen realization of the
stochastic noise that drives the BDSDE; by randomizing $W$, one recovers the
BDSDE. Section \ref{Chap-measurable-selection} introduces some measurable
selection results in the spirit of
{\cite{friz_controlled_2024,stricker_calcul_1978}}. These results allow us to
find a version of the RBSDE solution which is in some sense measurable with
respect to the rough driver $W$ (see Theorem \ref{measurable-RBSDE-solution}).
The exact definition of a BDSDE solution is presented in Section
\ref{solution-bdsde}. We show that RBSDE solutions can be seen as conditional
solutions to the BDSDEs (see Proposition \ref{quenched-BDSDE}), from which we
can deduce uniqueness of the BDSDE. We continue by showing that if we
randomize the measurable version of the RBSDE solution from Section
\ref{Chap-measurable-selection}, the resulting process indeed solves the
BDSDE, hence proving well-posedness of this new type of BDSDEs (see Theorem
\ref{Existence-BDSDE}).

\subsection{Measurable Selection}\label{Chap-measurable-selection}

On a filtered probability space  $(\Omega^1, \mathcal{F}^1, (\mathcal{F}^1_t)_{t \in [0, T]},
\mathbb{P}^1)$ satisfying the usual condition,
we denote by $\tmop{Prog}$  the progressive
$\sigma$-field (\cite[Def.I.4.7] {revuz_continuous_1999})
\begin{align*}
  \tmop{Prog} = \{A \in \mathcal{F}^1 \otimes \mathcal{B}([0, T]) \mid
  \mathds{1}_A \text{ is } (\mathcal{F}^1_t)_{t \in [0, T]}
  \text{-progressively measurable} \} .
\end{align*}
 
Let $(U, \mathcal{U})$ be a measurable space. Given two processes $X, \bar{X}$, beings maps from
$(\Omega^1 \times [0, T]) \times U$ into some Polish measurable space, we say
$\bar{X}$ is a \textbf{$\tmop{Prog} \otimes \, \mathcal{U}$-measurable version}
of $X$, if $\bar{X}$ is a $\tmop{Prog} \otimes \, \mathcal{U}$-measurable and $\bar{X} (\cdot, u)$ is a modification  (cf.\   \cite[Def.I.1.7]{protter_stochastic_2004})
of $X (\cdot, u)$ for any $u \in
U$, meaning that for any $t$ and $u$, $X (t,\cdot, u)=\bar{X}(t,\cdot, u)$ holds a.s..
If, moreover, $X$ and $\bar{X}$ are both c\`adl\`ag (or c{\`a}gl{\`a}d) in  $t$ for a.e. $\omega$ and all $u$, then
$\bar{X} (\cdot, u)$ is even indistinguishable from $X (\cdot, u)$ for any $u \in
U$.

The following measurable selection  results are similar to those on {\cite{friz_controlled_2024}}, building on classical work by Stricker and Yor {\cite{stricker_calcul_1978}}. We simply adapt
their results to our setting.

\begin{proposition}
  \label{limit-prog-measurable}Given processes $X^n : (\Omega^1 \times [0, T])
  \times U \longrightarrow \mathbb{R}^h$, $n \in \mathbb{N}$, such that every
  $X^n$ has a $\tmop{Prog} \otimes \, \mathcal{U}$-measurable version and  $X^n (\cdot, u)$ is $\mathbb{P}^1$-a.s. c{\`a}gl{\`a}d, for
  each $u \in U$. Let
  $X: \Omega^1 \times [0, T] \times U \longrightarrow \mathbb{R}^h$ be a map
  (without any measurability assumption) and assume that for each $u \in U$,
  $X^n (\cdot, u)$ converges in $\mathbb{P}^1$-probability to $X (\cdot, u)$
  uniformly in time. Then $X$ has a $\tmop{Prog} \otimes
  \mathcal{U}$-measurable version, which is $\mathbb{P}^1$-a.s. c{\`a}gl{\`a}d
  for every $u \in U$.
\end{proposition}

\begin{proof}
  The statement follows by the same proof as in Lemma 4.6 in
  {\cite{friz_controlled_2024}} by replacing the optional $\sigma$-field with
  the progressive $\sigma$-field $\tmop{Prog}$.
\end{proof}

Let $(\Omega^2, \mathcal{F}^2, \mathbb{P}^2)$ to be complete probability
space. Define $(\Omega, \mathcal{F}, \mathbb{P})$ as the product space
$(\Omega^1, \mathcal{F}^1, \mathbb{P}^1) \otimes (\Omega^2, \mathcal{F}^2,
\mathbb{P}^2)$ with the filtration $(\mathcal{F}_t)_{t \in [0, T]} \assign
(\mathcal{F}^1_t \otimes \mathcal{F}^2)_{t \in [0, T]}$.

\begin{proposition}
  \label{measurable-selection}Let $X^n : \Omega^1 \times \Omega^2 \times [0,
  T] \longrightarrow \mathbb{R}^h$ be $\mathcal{F}_t$-progressively measurable
  and $X : \Omega^1 \times \Omega^2 \times [0, T] \longrightarrow
  \mathbb{R}^h$ such that for $\mathbb{P}^2$-a.e. $\omega^2$ the process $X^n
  (\cdot, \omega^2)$ is $\mathbb{P}^1$-a.s. c{\`a}gl{\`a}d and converges in
  $\mathbb{P}^1$-probability to $X (\cdot, \omega^2)$ uniformly in time. Then
  $X$ has a $\mathcal{F}_t$-progressively measurable version $\tilde{X}$ such
  that $\tilde{X}$ is $\mathbb{P}$-a.s. c{\`a}gl{\`a}d and $X^n$ converge in
  $\mathbb{P}$-probability to $\tilde{X}$ uniformly in time.
\end{proposition}

\begin{proof}
  Define $C \subset \Omega^2$ to be the collection of all $\omega^2 \in \Omega^2$, where $X^n (\cdot, \omega^2)$ is not c{\`a}gl{\`a}d for some $n\in \mathbb{N}$ or $X^n (\cdot, \omega^2)$ does not converge to $X (\cdot, \omega^2)$ uniformly in time in $\mathbb{P}^1$-probability. Further define the process $\widetilde{X^n}$ as
  \begin{align*}
    \widetilde{X^n} (\omega^1, \omega^2) = \left\{ \begin{array}{ll}
      X^n (\omega^1, \omega^2) & , \tmop{if} \omega^2 \notin C\\
      0 & , \text{if } \omega^2 \in C
    \end{array} . \right.
  \end{align*}
  The process $\widetilde{X^n}$ is $\mathcal{F}_t$-progressively measurable
  since $C$ is a $\mathbb{P}^2$-Nullset and $\mathcal{F}^2$ is completed.
  Furthermore, by construction, we have for every $\omega^2$ the process
  $\widetilde{X^n} (\cdot, \omega^2)$ is c{\`a}gl{\`a}d $\mathbb{P}^1$-a.s.
  and converges to $X (\cdot, \omega^2)$ uniformly in time in
  $\mathbb{P}^1$-probability, so by Proposition \ref{limit-prog-measurable}
  there exists a ${\tmop{Prog} \otimes \mathcal{F}^2 } $-measurable version
  $\tilde{X}$ of $X$ such that for every $\omega^2$ we have $\tilde{X} (\cdot,
  \omega^2)$ is $\mathbb{P}^1$-a.s. c{\`a}gl{\`a}d and
  \begin{equation}
    \mathbb{E}^{\mathbb{P}^1} [\sup_{t \in [0, T]} | \widetilde{X^n_t} (\cdot,
    \omega^2) - \widetilde{X_t} (\cdot, \omega^2) | \wedge 1] \xrightarrow{n
    \rightarrow \infty}  0. \label{tildeXnTOtildeX}
  \end{equation}
  We get $\tilde{X}$ is $\mathbb{P}$-a.s. c{\`a}gl{\`a}d by simply applying
  Fubini. By a simple monotone class, we can see that ${\tmop{Prog}
  \otimes \mathcal{F}^2 } $-measurable is the same as
  $(\mathcal{F}_t)$-progressively measurable. This implies that $\sup_{t \in
  [0, T]} | \widetilde{X^n_t} - \widetilde{X_t} |$ is
  $\mathcal{F}_T$-measurable, then by Fubini, dominated convergence and
  (\ref{tildeXnTOtildeX}) we have
  \[ \mathbb{E}^{\mathbb{P}} [\sup_{t \in [0, T]} | \widetilde{X^n_t} -
     \widetilde{X_t} | \wedge 1] =\mathbb{E}^{\mathbb{P}^2}
     [\mathbb{E}^{\mathbb{P}^1} [\sup_{t \in [0, T]} | \widetilde{X^n_t} -
     \widetilde{X_t} | \wedge 1]] \xrightarrow{n \rightarrow \infty}  0. \]
  Since $X^n = \widetilde{X^n}$ holds $\mathbb{P}$-a.s. we get the desired
  convergence. 
\end{proof}

We consider the filtered probability space $(\Omega^B, \mathcal{F}^B,
(\mathcal{F}^B_t)_{t \in [0, T]}, \mathbb{P}^B)$, which supports a
$d$-dimensional Brownian motion $B$ and the filtration $(\mathcal{F}^B_t)_{t
\in [0, T]}$ is given by the usual filtration of $B$. Let $(U, \mathcal{U})$
denote another measurable space and $L : (U, \mathcal{U}) \rightarrow (D^q
([0, T]), \mathfrak{D}_T)$ be a measurable map from $(U, \mathcal{U})$ to
$(D^q, \mathfrak{D}_T)$, $q < 2$, with $\mathfrak{D}_T$ denoting the smallest
$\sigma$-algebra with respect to which all coordinate projections are
measurable (cf. Theorem 12.5, {\cite{billingsley_convergence_1999}}). We
consider RBSDEs in the form
\begin{eqnarray}
  Y_t^u & = & \xi (\cdot, u) + \int_t^T f (r, u, Y_r^u, Z_r^u) \tmop{dr} -
  \int_t^T Z_r^u \tmop{dB}_r + \int_t^T g_r (u, Y^u_{r +})  (\diamond)
  \tmop{dL}_r (u),  \label{RBSDE-measurable}
\end{eqnarray}
where we use the shorthand notation $\int_t^T g_r (u,Y^u_{r +}) \hspace{0.17em}
(\diamond) \tmop{dL}_r (u)$ to signal that we are showing results for both
forward- and Marcus-type RBSDEs. We have shown in Theorem
\ref{global_existenceuniqueness} that (under suitable conditions on the
coefficients) for any $u \in U$, there exists a unique solution $(Y^u, Z^u)$
to the above RBSDE. We define the process $(Y, Z)$ as maps from $\Omega^B
\times [0, T] \times U$ to $\mathbb{R}^h \times \mathbb{R}^{h \times d}$
given by $Y (t, \omega^B, u) = Y^u (t, \omega^B)$ and $Z (t, \omega^B, u) =
Z^u (t, \omega^B)$. We show in the next theorem that $(Y, Z)$ has a
$\tmop{Prog} \otimes \mathfrak{D}_T$-measurable version $(\tilde{Y}, \tilde{Z})$. \\
We precede the theorem with the following lemma, where we show a measurable
selection result for the It{\^o} representation theorem. 

\begin{lemma}
  \label{measurable-selection-Martingale-repre}Let $M : \Omega^B
\times [0, T] \times U \to \mathbb{R}^h$ be a $\mathcal{F}^B \otimes
\mathcal{B} ([0, T]) \otimes \mathcal{U}$-measurable process such that each
$M^u \assign M (\cdot, \cdot, u)$ is a zero mean $L^2$-martingale with respect
to the Brownian filtration $(\mathcal{F}_t^B)_t$. For each $u \in U$, let $H^u = (H^{u, i, j})_{i \leq h, j \leq d} :
\Omega^B \times [0, T] \to \mathbb{R}^{h \times d}$ be the unique progressively measurable (even
predictable) process $H^u \in L^2
(\tmop{dt} \otimes \mathbb{P}^B)$ given by It{\^o} representation such that 
$M_t^{u, i} = \int_0^t H_s^u
\tmop{dB}_s = \sum_{j = 1}^d H_s^{u, i, j} \tmop{dB}^j_s$, $i \leq h$,
for all $t \in [0, T]$. \\
There exists $\tilde{H} :
  (\Omega^B \times [0, T]) \times U \to \mathbb{R}^{h \times d},$ which is a
  $\tmop{Prog} \otimes \, \mathcal{U}$-measurable version of $H (u, \cdot, \cdot) \assign H^u$, $u\in U$. In particular it holds for all $u \in U$ and all $t \in [0, T]$
  that
  \[ M_t^{u, i} = \int_0^t \tilde{H}_s (u, \cdot, \cdot) \tmop{dB}_s = \sum_{j = 1}^d \int_0^t
     \tilde{H}_s^{i, j} (u, \cdot, \cdot) \tmop{dB}^j_s . \]
\end{lemma}

\begin{proof}(We thank Peter Imkeller for suggesting the idea of the proof.)
  We show the statement for $h = d = 1$. By
  in {\cite[Proposition~2, in slight (multivariate) generalization)]{stricker_calcul_1978}}, there is a $\tmop{Prog}
  \otimes \mathcal{U}$-measurable map denoted by $[M, B] : \Omega^B \times [0,
  T] \times U \to \mathbb{R}_+$ such that $[M, B]$ is a $\tmop{Prog} \otimes
  \mathcal{U}$-measurable version of $[M^u, B] \mid_{u = \cdot}$, which is equal (up to indistinguishability) to
  \[ \bigg[ \int_0^t H_s^u \tmop{dB}_s, B \bigg] = \int_0^t H_s^u \tmop{ds} .
  \]
  This shows that, for each $u \in U$, $t \mapsto [M, B] (\omega, t, u)$ is
  for a.e. $\omega$ in the Cameron-Martin space $\mathcal{H}_T$ of
  absolutely continuous paths $h$ on $[0, T]$ with $\partial h \in L_T^2 \assign L^2 ([0, T])$.
  From
  \[ \partial : \mathcal{H}_T \to L_T^2 ; \quad f \mapsto F \assign \bigg( t
     \mapsto \lim_{n \to \infty}  \frac{f (t) - f ((t - 1 / n) \vee 0)}{1 / n}
     \bigg) \]
  we obtain that $\tilde{H} \assign \partial [M, B]$ is $\tmop{Prog} \otimes
  \mathcal{U}$-measurable from applying Proposition \ref{limit-prog-measurable} to the limit in the previous definition. Recall that by definition it holds $H^u=H (u, \cdot, \cdot)$, we can conclude that \~{H} is a $\tmop{Prog} \otimes
  \mathcal{U}$-measurable version of $H$. 
\end{proof}

\begin{theorem}
  \label{measurable-RBSDE-solution}Let $\xi : \Omega^B \times U \rightarrow
  \mathbb{R}^h$ be $\mathcal{F}^B \otimes \mathcal{U}$-measurable, $f :
  (\Omega^B \times [0, T]) \times U \times \mathbb{R}^h \times \mathbb{R}^{h
  \times d} \rightarrow \mathbb{R}^h$ be $\tmop{Prog} \otimes \mathcal{U}
  \otimes \mathcal{B}_h \otimes \mathcal{B}_{h \times d}$-measurable and $g :
  (\Omega^B \times [0, T]) \times U \times \mathbb{R}^h \rightarrow
  \mathcal{L} (\mathbb{R}^e, \mathbb{R}^h)$ be $\tmop{Prog} \otimes
  \mathcal{U} \otimes \mathcal{B}_h$-measurable. Further assume that for any
  $u \in U$ the functions $\xi^u (\cdot) \assign \xi (\cdot, u)$, $f^u (\cdot)
  \assign f (u, \cdot)$ and $g^u (\cdot) \assign g (u, \cdot)$ satisfy the
  Assumption A. \\
  Let $(Y^u, Z^u)$ denote the solution to the RBSDE
  \eqref{RBSDE-measurable} and define $Y (t, \omega^B, u) = Y^u (t, \omega^B)$
  and $Z (t, \omega^B, u) = Z^u (t, \omega^B)$. Then $(Y, Z)$ has a
  $\tmop{Prog} \otimes \, \mathcal{U}$-measurable version $(\tilde{Y},
  \tilde{Z})$, such that $\tilde{Y} (\cdot, u)$ is $\mathbb{P}^B$-a.s.
  c{\`a}gl{\`a}d for any $u \in U$. In particular, $Y(\cdot,u)$ and $\tilde{Y}(\cdot,u)$ are indistinguishable for all $u \in U$.

\end{theorem}
\begin{remark}
  One can simply choose $(U, \mathcal{U})$ to be $(D^p, \mathfrak{D}_T)$
  and the map $L$ to be the identity map, then we would have the measurable
  dependency of $(Y, Z)$ directly to the rough driver $W \in U = D^p$, in the
  sense that $(Y, Z)$ has a $\tmop{Prog} \otimes \mathfrak{D}_T$-measurable
  version. 
\end{remark}

\begin{proof}
  Recall from Theorem \ref{Picard-iteration} that $(Y^u, Z^u)$ is the limit of
  the Picard iteration \\$(Y^{u, n}, Z^{u, n})_{n \in \NN}$, defined as $Y^{u,
  0} \equiv 0$, $Z^{u, 0} \equiv 0$ and then for every $u \in U$
  iteratively as
  \begin{align}
    Y^{u, n + 1}_t  =  \xi^u + \int_t^T f^u (r, Y^{u, n}_r, Z^{u, n}_r)
    \tmop{dr} + \int_t^T g^u_r (Y^{u, n}_{r +})  (\diamond) \tmop{dL}_r (u) -
    \int_t^T Z_r^{u, n + 1} \tmop{dB}_r .  \label{Picard-W-RBSDE}
  \end{align}
  We show by induction that $(Y^{\cdot, n}, Z^{\cdot, n})$ has $\tmop{Prog}
  \otimes \mathcal{U}$-measurable version. \
  
  For $Y^{\cdot, 0} \equiv 0$, $Z^{\cdot, 0} \equiv 0$ this obviously holds. We abuse the notation a little bit and let $(Y^{\cdot, n},
  Z^{\cdot, n}_r)$ denote its $\tmop{Prog} \otimes \, \mathcal{U}$-measurable
  version. Using the $\tmop{Prog} \otimes \, \mathcal{U}$-measurability of
  $(Y^{\cdot, n}, Z^{\cdot, n})$ and applying Lemma 8.5 in
  \cite{friz_controlled_2024} and Lemma \ref{Measurable-Young} we get that \linebreak
  $\int f^u (r, Y^{u, n}_r, Z^{u, n}_r) \tmop{dr} \mid_{u = \cdot}$ and $\int
  g^u_r (Y^{u, n}_{r +}) \tmop{dL}_r (u) \mid_{u = \cdot}$ have $\tmop{Prog}
  \otimes \mathfrak{D}_T$-measurable versions (again denoted the same). As for
  ${\sum_{0 \leq r < \cdot}}  [\varphi (g^u_r \Delta L (u)_r, Y^{u, n}_{r
  +}) - Y^{u, n}_{r +} - g^u_r (Y^{u, n}_{r +}) \Delta L (u)_r] \mid_{u =
  \cdot}$, it follows directly from continuity of the ODE solution with
  respect to the initial condition (cf. \cite[Theorem III.13.II]{walter_ordinary_1998}) that each summand is $\mathcal{F}^B_r
  \otimes \mathfrak{D}_T$-measurable. By the absolute continuity of the sum
  $\sum_{0 \leq r < t} (\cdot)$, we get that the sum $\sum_{0 \leq r
  < t} (\cdot)$ is $\mathcal{F}^B_t \otimes \mathcal{U}$-measurable, hence by
  the left-continuity of ${\sum_{0 \leq r < \cdot}}  (\cdot)$, we get the
  $\tmop{Prog} \otimes \, \mathcal{U}$-measurability.
  
  We are left to show that $Z^{\cdot, n + 1}$ and $\int  Z^{u, n + 1}_r
  \tmop{dB}_r \mid_{u = \cdot}$ have $\tmop{Prog} \otimes
  \mathcal{U}$-measurable versions. Recall that $Z^{u, n + 1}$ is given by the
  martingale representation
  \begin{align*}
    \int_0^t Z^{u, n + 1}_r \tmop{dB}_r = \mathbb{E} \big[ -Y^{u,n+1}_0 + \xi^u + \int_0^T
    f^u (r, Y^{u, n}_r, Z^{u, n}_r) \tmop{dr} + \int_0^T g^u_r (Y^{u, n}_{r
    +}) (\diamond) \tmop{dL}_r (u) \vert \mathcal{F}^B_t \big] .
  \end{align*}
  We have shown that everything inside the conditional expectation has
  $\tmop{Prog} \otimes \, \mathcal{U}$ - measurable versions, then by Lemma
  \ref{measurable-selection-Martingale-repre}, $Z^{\cdot, n + 1}$ also has a
  $\tmop{Prog} \otimes \, \mathcal{U}$-measurable version. \\
  We can now rewrite \eqref{Picard-W-RBSDE} as a ``forward'' equation
  \begin{align*}
    Y^{\cdot, n + 1}_t - Y^{\cdot, n + 1}_0  =&  - \int_0^t f^u (r, Y^{u, n}_r, Z^{u, n}_r) \tmop{dr}
    \mid_{u = \cdot} - \int_0^t g^u_r (Y^{u, n}_{r
    +}) (\diamond) \tmop{dL}_r (u) \mid_{u =
    \cdot} \\
    &+ \int_0^t Z^{u, n + 1}_r \tmop{dB}_r \mid_{u = \cdot} 
  \end{align*}
  to see that $Y^{\cdot, n + 1}$ has $\tmop{Prog} \otimes \, \mathcal{U}$-measurable version, which is also c{\`a}gl{\`a}d. 
  
  By Theorem \ref{Picard-iteration} we have for all $u \in U$ that $Y^{u, n}
  \rightarrow Y^u$ in $\| \cdot \|_{p, 2}$ and $Z^{u, n} \rightarrow Z^u$ in
  $\tmop{BMO}$, then by the estimates $\mathbb{E} [\sup_{t \in [0, T]} | Y^{u,
  n}_t - Y^u_t |^2]^{\frac{1}{2}} \leq \| Y^{u, n} - Y^u \|_{p, 2}$ and
  $\mathbb{E} \big[ \int_0^T (Z_r^{u, n} - Z_r^u)^2 \tmop{dr} \big]
  \leq \| Z^{u, n} - Z^u \|_{\tmop{BMO}}$ we get in particular $Y^{u, n}
  \rightarrow Y^u$ uniform in time in $\mathbb{P}^B$-probability and $Z^{u, n
  + 1} \rightarrow Z^u$ in $\tmop{dt} \otimes \mathbb{P}^B$-probability. By
  Lemma \ref{limit-prog-measurable} and Proposition 1 in
  {\cite{stricker_calcul_1978}} we get that $Y^{\cdot}$ and $Z^{\cdot}$ both
  have $\tmop{Prog} \otimes \, \mathcal{U}$-measurable version $(\tilde{Y},
  \tilde{Z})$ and $\tilde{Y} (\cdot, u)$ is $\mathbb{P}^B$-a.s. càglàd for any $u \in
  U$.
\end{proof}

\subsection{Solution to BDSDE}\label{solution-bdsde}

In this section, we study the BDSDE \eqref{Marcus-BDSDE} and we start by
specifying the probabilistic setup. Let $(\Omega^B, \mathcal{F}^B,
\mathbb{P}^B)$ and $(\Omega^L, \mathcal{F}^L, \mathbb{P}^L)$ denote complete
probability spaces, which support respectively a $d$-dimensional Brownian
motion $B$ and a process $L : (\Omega^L, \mathcal{F}^L) \rightarrow (D^q,
\mathfrak{D}_T)$ of finite $q$-variation. One can for example take $(\Omega^B,
\mathcal{F}^B, \mathbb{P}^B)$ to be the classical Wiener space and $B$ to be
canonical process. As for $(\Omega^L, \mathcal{F}^L, \mathbb{P}^L)$ one can
take $(\Omega^L, \mathcal{F}^L)$ to be $(D^q, \mathfrak{D}_T)$, $L$ to be the
canonical process and $\mathbb{P}^L$ to be the measure that is uniquely
determined by the canonical process $L$ being
\begin{enumeratenumeric}
  \item a L{\'e}vy process $L$ with characteristic triplet $(0,\gamma,\nu)$, such that the so-called Blumenthal--Getoor index $\beta$, defined as
\begin{align*}
\beta := \inf \left\{ \alpha>0 \,\middle|\, \int_{|x| < 1} |x|^{\alpha}\,\nu(dx) < \infty \right\},
\end{align*}
is smaller than $2$. Such process is almost surely of finite $q$-variation for any $\beta < q < 2$ by \cite[Theorem 2]{monroe_gamma_1972}. This includes compound Poisson
processes, Gamma processes, $\alpha$-stable L{\'e}vy processes with
$\alpha<2$, and a large class of tempered stable L{\'e}vy processes (see
\cite[Chap.~3.7]{cont_financial_2004} and the referene therein).
  
  \item a fractional Brownian motion with Hurst $h \in (\frac{1}{2}, 1)$;
  
  \item a sum of an independent L{\'e}vy process and a fractional Brownian motion,
  both of the above types.
\end{enumeratenumeric}
Whether $L$ has independent increments with respect to $\mathbb{P}^L$ plays an
important role later in Corollary \ref{BDSDE-independentInc} to specify the
filtration of BDSDE solutions, one can check easily that only the first out of
the above example has independent increments.

We work on the product space
\[ (\Omega, \mathcal{F}, \mathbb{P}) \assign (\Omega^B, \mathcal{F}^B,
   \mathbb{P}^B) \otimes (\Omega^L, \mathcal{F}^L, \mathbb{P}^L), \]
on which by construction the (lifted) map $B (\omega^B, \omega^L) \equiv B
(\omega^B)$ is a Brownian motion and $L (\omega^B, \omega^L) \equiv L
(\omega^L)$ is a stochastic process, which is independent of $B$ and of finite
$q$-variation for some $q < 2$. We introduce the forward filtration
$\mathcal{F}^B_t = \sigma \{B_s, 0 \leq s \leq t\} \vee \mathcal{N}^B$
generated by $B$ (on $\Omega^B$) and completed with the
$\mathbb{P}^B$-negligible sets $\mathcal{N}^B$ and backward filtration
$\mathcal{F}^L_{t, T} = \sigma \{L_T - L_s, t \leq s \leq T\} \vee
\mathcal{N}^L$ generated by L (on $\Omega^L$) and completed with the
$\mathbb{P}^L$-negligible sets $\mathcal{N}^L$. One can naturally lift
$(\mathcal{F}^B_t)_t$ and $(\mathcal{F}^L_{t, T})_t$ to filtrations on the
product space $\Omega$ with $\bar{\mathcal{F}}^B_t =\mathcal{F}^B_t \otimes \{
\emptyset, \Omega^L \}$ and $\bar{\mathcal{F}}^L_{t, T} = \{ \emptyset,
\Omega^B \} \otimes \mathcal{F}^L_{t, T}$. Finally, we define on $\Omega$ the
initially enlarged filtration $(\mathcal{F}^B_t \otimes \mathcal{F}^L_{0,
T})_{t \in [0, T]}$, \ and the two sides ``filtration''
$\mathcal{G}_t =\mathcal{F}^B_t \otimes \mathcal{F}^L_{t, T}$ often used in
BDSDE literature, which is actually not a filtration.

\

Let $\xi : \Omega \rightarrow \mathbb{R}$ be $\mathcal{F}^B_T \otimes \mathcal{F}^L_{0, T}$-measurable, $f :
\Omega^B \times [0, T] \times \Omega^L \times \mathbb{R}^h \times \mathbb{R}^{h\times d}
\rightarrow \mathbb{R}^h$ be $\tmop{Prog} \otimes \mathcal{F}^L_{0, T} \otimes \mathcal{B} \otimes
\mathcal{B}_d  $-measurable and $g : \Omega^B \times [0, T]
\times \Omega^L \times \mathbb{R}^h \rightarrow \mathbb{R}^h$ be $\tmop{Prog} \otimes
\mathcal{F}^L_{0, T} \otimes \mathcal{B}$-measurable, we are interested in the
\tmtextbf{BDSDE} \eqref{Marcus-BDSDE} as seen in the introduction
\begin{eqnarray}
  Y_t & = & \xi + \int_t^T f (r, Y_r, Z_r) \tmop{dr} - \int_t^T Z_r
  \tmop{dB}_r + \int_t^T g_r (Y_{r +})  (\diamond) \tmop{dL}_r . \nonumber
\end{eqnarray}
Here, the first two integrals are standard Lebesgue integrals and It\^{o}
integral. The last integral $\int g (Y) \tmop{dL}$ is a pathwise defined
backward Young integral (see Proposition \ref{randomizedYoung}). The term
$\int g (Y) \diamond \tmop{dL}$ is again a short hand notation for
\begin{align}
  \int_t^T g_r (Y_{r +}) \diamond \tmop{dL}_r =  \int_t^T g_r (Y_{r +})
  \tmop{dL}_r + \sum_{t \leq r < T} [\varphi (-g_r \Delta L_r, Y_{r +}) - Y_{r +}
  - g_r (Y_{r +}) \Delta L_r] .  \label{BDSDE-jumps}
\end{align}
In the following, we provide a solution theory for this new type of BDSDEs,
which accommodates a wide range of process $L$, even allowing jumps, while
``only'' requiring it to be of finite $q$-variation with $q < 2$. Of course,
due to the Brownian motion notoriously only being of finite $2 +
\varepsilon$-variation, we can not study the classical BDSDE introduced by
Pardoux and Peng in {\cite{pardoux_backward_1994}}. But we can obtain the
results about the well-posedness of BDSDEs by Jing
{\cite{jing_nonlinear_2012}} (under slightly different assumptions), where the
author studies the case of $L$ being a fractional Brownian motion with Hurst
parameter in $(\frac{1}{2}, 1)$. In that paper, the author makes sense of $\int
g (Y) \tmop{dL}$ as a backward Russo--Vallois integral, which, given enough
regularity of the integrand agrees with the backward Young integral in our
paper, this has been shown in Section 4 of {\cite{zahle_integration_2001}} for
forward integral, but for backward integral the argument is the same.

In the next definition, we introduce the notion of BDSDE solution.

\begin{definition}
  \label{Def-BDSDE-solution}Let $Y : [0, T] \times \Omega \rightarrow
  \mathbb{R}$ and $Z : [0, T] \times \Omega \rightarrow \mathbb{R}$ be a pair
  of $(\mathcal{F}^B_t \otimes \mathcal{F}^L_{0, T})_{t \in [0,
  T]}$-progressively measurable processes, such that for $\mathbb{P}^L$-a.e.
  $\omega^L$ the norms $\| Y (\cdot, \omega^L) \|_{p, 2}$ and $\|Z(\cdot, \omega^L)\|_{\tmop{BMO}}$ are finite\footnote{For $(Y, Z)$ satisfying the
  above conditions, all integrals in (\ref{Marcus-BDSDE}) are intrinsically
  well-defined.}. The pair $(Y, Z)$ is called a solution to the BDSDE if it
  satisfies the integral equation (\ref{Marcus-BDSDE}). 
\end{definition}

\begin{remark}
  The solution pair $(Y, Z)$ in general does not need to be adapted to two
  sides ``filtration'' $\mathcal{G}_t =\mathcal{F}^B_t \otimes
  \mathcal{F}^L_{t, T}$. This property only holds under additional measurability assumptions on
$\xi,f,g$, together with the assumption that $L$ has independent increments,
see Corollary~\ref{BDSDE-independentInc}.
\end{remark}

We start by showing the uniqueness of the BDSDE solution and its connection to
RBSDEs if the solution exists. The existence of the BDSDE solution will be
shown in Theorem \ref{Existence-BDSDE}. But before we show the following
useful lemma. \

\begin{lemma}
  \label{quenched-lebesgue}
  Let $x : [0, T] \times \Omega \rightarrow \mathbb{R}$ be a measurable
  function, then it holds $\mathbb{P}$-a.s.
  \begin{align*}
    \int x_r(\omega^B,\omega^L) \tmop{dr}
    =
    \left.
    \int x_r(\omega^B,W) \tmop{dr}
    \right|_{W=\omega^L}.
  \end{align*}
  Let $x : [0, T] \times \Omega \rightarrow \mathbb{R}$ be
  $(\mathcal{F}^B_t \otimes \mathcal{F}^L_{0,T})_{t \in [0,T]}$-progressively
  measurable, then $\mathbb{P}$-a.s.
  \begin{align*}
    \int x_r(\omega^B,\omega^L) \tmop{dB}_r
    =
    \left.
    \int x_r(\omega^B,W) \tmop{dB}_r
    \right|_{W=\omega^L}.
  \end{align*}
\end{lemma}

\begin{proof}
  Let $\pi = (\pi^n)_{n \in \NN}$ be a sequence of time partitions on $[0,T]$
  with $|\pi^n| \to 0$ as $n \to \infty$. We have
  \[
  \sum_{t_i^n \in \pi^n}
  x_{t_i^n}(\omega^B,\omega^L)(t^n_{i+1}-t^n_i)
  =
  \sum_{t_i^n \in \pi^n}
  x_{t_i^n}(\omega^B,W)(t^n_{i+1}-t^n_i)
  |_{W=\omega^L}.
  \]
  Taking $n\to\infty$ yields the first assertion. The second follows by the
  same argument.
\end{proof}

\begin{proposition}
  \label{quenched-BDSDE}Let $\xi : \Omega \rightarrow \mathbb{R}^h$ be
  $\mathcal{F}^B_{T} \otimes \mathcal{F}^L_{0, T}$-measurable, $f : (\Omega^B \times [0, T]) \times \Omega^L \times
  \mathbb{R}^h \times \mathbb{R}^{h \times d} \rightarrow \mathbb{R}^h$ be
  $\tmop{Prog} \otimes \mathcal{F}^L_{0, T} \otimes \mathcal{B}_h \otimes
  \mathcal{B}_{h \times d}$-measurable and $g : (\Omega^B \times [0, T])
  \times \Omega^L \times \mathbb{R}^h \rightarrow \mathcal{L} (\mathbb{R}^e,
  \mathbb{R}^h)$ be $\tmop{Prog} \otimes \mathcal{F}^L_{0, T} \otimes
  \mathcal{B}_h$-measurable. Assume that for $\mathbb{P}^L$-a.e. $\omega^L \in \Omega^L$,
  the functions $\xi (\cdot, \omega^L)$, $f
  (\omega^L, \cdot)$ and $g (\omega^L, \cdot)$ satisfy the Assumption A.
  Then the BDSDE has at most one solution in the sense of Definition
  \ref{Def-BDSDE-solution}. Furthermore let $(Y, Z)$ be a solution of the
  BDSDE, then for $\mathbb{P}^L$-a.e. $\omega^L \in \Omega^L$ the pair $(Y
  (\cdot, \omega^L), Z (\cdot, \omega^L))$ is a solution to the RBSDE
  \begin{eqnarray*}
    Y^{L (\omega^L)}_t & = & \xi (\cdot, \omega^L) + \int_t^T f (r, \cdot,
    \omega^L, Y^{L (\omega^L)}_r, Z^{L (\omega^L)}_r) \tmop{dr} - \int_t^T
    Z^{L (\omega^L)}_r \tmop{dB}_r\\
    &  & + \int_t^T g_r (\cdot, \omega^L, Y^{L (\omega^L)}_{r +})  (\diamond)
    \tmop{dL}_r (\omega^L) .
  \end{eqnarray*}
\end{proposition}

\begin{proof}
  Let $(Y^1, Z^1), (Y^2, Z^2)$ both be solutions to the BDSDE
  \eqref{Marcus-BDSDE}. Then by Lemma \ref{quenched-lebesgue} and
  \ref{randomizedYoung} and similar arguments for the sum in (\ref{BDSDE-jumps})
  we get for $\mathbb{P}^L$-a.e. $\omega^L \in \Omega^L$ that
  \begin{align*}
    Y^i_t (\cdot, \omega^L) = & \xi (\cdot, \omega^L) + \int_t^T f (r, W, Y
    _r (\cdot, W), Z_r (\cdot, W)) \tmop{dr} \mid_{W = \omega^L} - \int_t^T
    Z_r (\cdot, W) \tmop{dB}_r \mid_{W = \omega^L}\\
    & + \int_t^T g_r (\cdot, W, Y_{r +} (\cdot, W))  (\diamond) \tmop{dL}_r
    (W) \mid_{W = \omega^L} .
  \end{align*}
  Then by the uniqueness of the solution to the above RBSDE (see Theorem
  \ref{global_existenceuniqueness}) we get for $\mathbb{P}^L$-a.e. $\omega^L
  \in \Omega^L$ that $Y^1 (\cdot, \omega^L) = Y^2 (\cdot, \omega^L)$ holds $\mathbb{P}^B$-a.s. for all $t$ and $Z^1 
  (\cdot, \omega^L) = Z^2  (\cdot, \omega^L)$ holds $\tmop{dt} \otimes
  \mathbb{P}^B$-a.s., then by Fubini we get that $Y^1 = Y^2$ $\mathbb{P}$-a.s. for all $t$ and $Z^1
  = Z^2$ hold $\tmop{dt} \otimes \mathbb{P}$-almost surely. 
\end{proof}

Given the uniqueness of the solution, we only need to construct a solution to the
BDSDE \eqref{Marcus-BDSDE} to get well-posedness of it. We apply Theorem~\ref{measurable-RBSDE-solution} with
$(U,\mathcal{U})=(\Omega^L,\mathcal{F}^L_{0,T})$ to obtain a pair
$(\tilde{Y},\tilde{Z})$ of $\mathrm{Prog}\otimes\mathcal{U}$-measurable
processes such that, for every $u\in U$, the pair
$(\tilde{Y}(\cdot,u),\tilde{Z}(\cdot,u))$ solves
\eqref{RBSDE-measurable}. We then randomize this process by setting
\[
(Y^L_t(\omega),Z^L_t(\omega))
\assign
(\tilde{Y}_t(\omega^B,\omega^L),
\tilde{Z}_t(\omega^B,\omega^L)),
\qquad t\in[0,T].
\]
By construction, $(Y^L,Z^L)$ is
$(\mathcal{F}^B_t\otimes\mathcal{F}^L_{0,T})_{t\in[0,T]}$-progressively
measurable. \\
We show in the next theorem that
$(Y^L, Z^L)$ is indeed a solution to the BDSDE (\ref{Marcus-BDSDE}). To further motivate this solution, we also show that $(Y^L, Z^L)$ is the limit
of the Picard iteration $(Y^n, Z^n)$, $n \in \mathbb{N}$, of the BDSDE
(\ref{Marcus-BDSDE}). We define $Y^0 \equiv 0$, $Z^0 \equiv 0$ and then
iteratively define $(Y^{n+1}, Z^{n+1})$ in the following way. We define
\begin{align*}
  Y^{n+1}_t =\mathbb{E} \bigg[ \xi + \int_t^T f (r, \omega, Y^n_r, Z^n_r)
  \tmop{dr} + \int_t^T g_r (Y^n_{r+}) (\diamond) \tmop{dL}_r \mid
  \mathcal{F}^B_t \otimes \mathcal{F}^L_{0,T} \bigg],
\end{align*}
here the process $Y^{n+1}$ is by definition
$(\mathcal{F}^B_t \otimes \mathcal{F}^L_{0,T})_{t\in[0,T]}$-adapted, and we
shall always work with its c{\`a}gl{\`a}d and hence
$(\mathcal{F}^B_t \otimes \mathcal{F}^L_{0,T})_{t\in[0,T]}$-progressively
measurable version. By the extended It{\^o} representation theorem
(cf. \cite[Theorem 4.2]{amendinger_martingale_2000}), there exists a unique
$(\mathcal{F}^B_t \otimes \mathcal{F}^L_{0,T})_{t\in[0,T]}$-progressively
measurable process $Z^{n+1}$ such that
\begin{align*}
  \int_0^t Z^{n+1}_r \tmop{dB}_r
  =
  \mathbb{E}\bigg[
  \xi
  + \int_0^T f(r,\omega,Y^n_r,Z^n_r)\tmop{dr}
  + \int_0^T g_r(Y^n_{r+})(\diamond)\tmop{dW}
  \mid
  \mathcal{F}^B_t \otimes \mathcal{F}^L_{0,T}
  \bigg].
\end{align*}
\begin{theorem}
  \label{Existence-BDSDE}Under the same assumption as in Proposition
  \ref{quenched-BDSDE}, there exists a pair of $(\mathcal{F}^B_t \otimes \mathcal{F}^L_{0,T})_{t \in [0, T]}$-progressively measurable process
  $(Y^L, Z^L)$, which is the unique solution to the BDSDE
  (\ref{Marcus-BDSDE}). Furthermore it holds $Y^{n + 1} \rightarrow Y^L$
  uniform in time in $\mathbb{P}$-probability and $Z^{n + 1} \rightarrow Z^L$
  in $\tmop{dt} \otimes \mathbb{P}$-probability.
\end{theorem}

\begin{proof}
  To show that the $(Y^L, Z^L)$ constructed above indeed solves the BDSDE is
  straightforward. We have by construction that for $\mathbb{P}^L$-a.e.
  $\omega^L$ we have
  \begin{eqnarray*}
    Y^L_t (\cdot, \omega^L) & = & \xi (\cdot, \omega^L) + \int_t^T f (r,
    \omega^L, Y _r (\cdot, \omega^L), Z_r (\cdot, \omega^L)) \tmop{dr} -
    \int_t^T Z_r (\cdot, \omega^L) \tmop{dB}_r\\
    &  & + \int_t^T g_r (\cdot, \omega^L, Y_{r +} (\cdot, \omega^L)) 
    (\diamond) \tmop{dL}_r (\omega^L) .
  \end{eqnarray*}
  Hence by the same argument as Proposition \ref{quenched-BDSDE} (only the
  other way around) we get that $(Y^L, Z^L)$ solves the BDSDE
  \eqref{Marcus-BDSDE}.
  
  Now let $Y^{\cdot, n + 1}$ and $Z^{\cdot, n + 1}$ denote the $\tmop{Prog}
  \otimes \mathcal{F}^L_{0,T}$-measurable version of the iteration for the RBSDE in
  Theorem \ref{measurable-RBSDE-solution}. Then one can prove iteratively
  using again the same argument as Proposition \ref{quenched-BDSDE} that
  $Y^{\cdot, n + 1} = Y^{n + 1}$ and $Z^{\cdot, n + 1} = Z^{n + 1}$ holds
  $\tmop{dt} \otimes \mathbb{P}$-a.s., and from Theorem
  \ref{measurable-RBSDE-solution} we know that $Y^{n + 1} (\cdot, \omega^L) =
  Y^{u, n + 1} \mid_{u = \omega^L} \rightarrow Y^L (\cdot, \omega^L)$ uniform
  in time in $\mathbb{P}^B$-probability and $Z^{n + 1} (\cdot, \omega^L) =
  Z^{u, n + 1} \mid_{u = \omega^L} \rightarrow Z^L (\cdot, \omega^L)$ in
  $\tmop{dt} \otimes \mathbb{P}^B$-probability for any $\omega^L \in
  \Omega^L$. By Fubini we get the desired convergence result.
\end{proof}

We will see in the next corollary that under additional measurability assumptions
on $\xi, f$ and $g$ and assuming $L$ to have independent increment, the solution
pair $(Y_t, Z_t)$ is independent of the information on $L$ prior to time $t$.
We omit the proof since it is analogous to the proof of Proposition 1.2 in
{\cite{pardoux_backward_1994}}.

\begin{corollary}
  \label{BDSDE-independentInc}
  Under the assumptions of Proposition~\ref{quenched-BDSDE}, assume further
  that $\xi$ is $\mathcal{G}_T$-measurable and that, for every
  $(y,z)\in\mathbb{R}^h \times \mathbb{R}^{h \times d}$, the processes
  $f(\cdot,y,z)$ and $g(\cdot,y)$ are adapted to
  $(\mathcal{G}_t)_{t \in [0,T]} = (\mathcal{F}^B_t \otimes
  \mathcal{F}^L_{t,T})_{t \in [0,T]}$. If $L$ has independent increments,
  then the unique solution $(Y,Z)$ to the BDSDE \eqref{Marcus-BDSDE} is
  adapted to $(\mathcal{G}_t)_{t \in [0,T]}$.
\end{corollary}

\appendix
\section{Backward Young Integral}

Here, we construct the backward Young integral similarly to the (forward)
Young integral in {\cite{friz_differential_2018}}. In the following appendix, we will be working with both c{\`a}dl{\'a}g and
c{\`a}gl{\`a}d paths. We will therefore introduce the notations $\Delta^- x_t
\assign x_t - x_{t -}$ and $\Delta^+ x_t \assign x_{t +} - x_t$ denoting the
jump of $x$ at time $t$ from the left or respectively from the right.

Let $\mathbb{W}, \mathbb{V}, \mathbb{U}$ be finite dimensional Banach spaces, we introduce
the space of paths from $[0, T]$ to $\mathbb{V}$ of finite $p$-variation as
$V^p ([0, T], \mathbb{V})$, here we do not assume the path to have any right-
or left continuity, so naturally $C^p \subset D^p \subset V^p$. We denote by
$\mathcal{L} (\mathbb{W}, \mathbb{V})$ the space of linear operators from
$\mathbb{W}$ to $\mathbb{V}$, which is again a Banach space when equipped with
the operator norm.

We differentiate between two types of convergence of the Riemann-Stieltjes
sums, namely the convergence in MRS and in RRS sense, for details see
Definition 1.1 in {\cite{friz_differential_2018}}.

\begin{proposition}
  \label{BackwardsYoung}Let $x \in V^p ([0, T], \mathcal{L} (\mathbb{V},
  \mathbb{W}))$ and $y \in V^q ([0, T], \mathbb{V})$ with $\frac{1}{p} +
  \frac{1}{q} > 1$. Let $\pi = (\pi^n)_{n \in \NN}$ be a sequence of time
  partitions on $[0, T]$ of the form $\pi^n = (0 = t^n_0 < t^n_1 < \cdots <
  t^n_N = T)$ with $| \pi^n | \to 0$ as $n \to \infty$. Then the limit
  \begin{equation}
    \lim_{n \to \infty}  \sum_{t_i^n \in \pi^n} x_{t^n_{i + 1}} y_{t^n_i,
    t^n_{i + 1}} = : \int_0^T x_r d \overleftarrow{y_r} \label{backwardYoung}
  \end{equation}
  exists in RRS sense and is called the backward Young integral of $x$
  integrated against $y$. We further have the estimate
  \begin{equation}
    \bigg| \int_s^t x_r d \overleftarrow{y_r} - x_t y_{s, t} \bigg| \leq
    C_{p, q} \| x \|_{p ; (s, t]} \| y \|_{q ; [s, t)} . \label{YoungEstimate}
  \end{equation}
  If additionally $y$ is c{\`a}gl{\`a}d or $x$ is c{\`a}dl{\`a}g, then the
  convergence in (\ref{backwardYoung}) holds in MRS sense.
\end{proposition}

\begin{proof}
  Simply apply Theorem 2.2 in {\cite{friz_differential_2018}} to the germ
  $\Xi_{s, t} \assign x_t y_{s, t}$.
\end{proof}

\begin{corollary}
  \label{integral_q-var}Let $x \in V^p ([0, T], \mathcal{L} (\mathbb{V},
  \mathbb{W}))$ and $y \in V^q ([0, T], \mathbb{V})$ with $\frac{1}{p} +
  \frac{1}{q} > 1$ and further assume $y$ to be c{\`a}gl{\`a}d. Then the
  backward Young integral $\int_0^{\cdot} xd \overleftarrow{y_r}$ is a
  c{\`a}gl{\`a}d path of finite $q$-variation, in particular, it holds 
  \begin{align} \label{Young-estimate-p}
    \bigg\| \int_s^t x_r d \overleftarrow{y_r} \bigg\|_{q ; [0, T]} \leq
    C_{p, q} \| x \|_{p ; [0, T]} \| y \|_{q ; [0, T]} + (\| x \|_{p ; [0, T]}
    + | x_T |) \| y \|_{q ; [0, T]}
  \end{align}
\end{corollary}

\begin{proof}
  By the estimate (\ref{YoungEstimate}) it holds
  \begin{align}
    \bigg| \int_s^t x_r d \overleftarrow{y_r} \bigg| & \leq C_{p, q} \| x
    \|_{p ; (s, t]} \| y \|_{q ; [s, t)} + | x_t | \| y \|_{q ; [s, t]} \\
    & \leq C_{p, q} \| x \|_{p ; (s, t]} \| y \|_{q ; [s, t)} + (\| x \|_{p ;
    (s, t]} + | x_T |) \| y \|_{q ; [s, t]} . \nonumber
  \end{align}
  For $s \mapsto t$, we have that the right-hand side converges to zero by a
  similar argument as Lemma 7.1 in {\cite{friz_differential_2018}} and we have
  shown the left continuity of the integral.\\
  Proposition 5.8 in {\cite{friz_multidimensional_2010}}
  imply that $\omega(s,t) \assign \| y \|^q_{q ; [s, t]}$ is control, then we get
  (\ref{Young-estimate-p}) by applying Proposition 5.10 in {\cite{friz_multidimensional_2010}}.
\end{proof}

The following shows that for left-continuous $y$, the
choice of $x,x^+,x^-$ is irrelevant. 

\begin{lemma}
  \label{backwardYoungJump}Let $x \in V^p ([0, T], \mathcal{L} (\mathbb{V},
  \mathbb{W}))$ and $y \in V^q ([0, T], \mathbb{V})$ with $\frac{1}{p} +
  \frac{1}{q} > 1$ and further assume $y$ to be c{\`a}gl{\`a}d. Then the
  following backward Young integrals are equal
  \[ \int_0^T x_r d \overleftarrow{y_r} = \int_0^T x^+_r d \overleftarrow{y_r}
     = \int_0^T x^-_r d \overleftarrow{y_r} . \]
\end{lemma}

\begin{proof}
  Apply Theorem 2.11 in {\cite{friz_differential_2018}} for $g (s) \assign |
  \Delta^{\pm} x_s |^p$ and $\omega (s, t) \assign \| y \|^q_{q ; [s, t]}$.
\end{proof}

The following lemma shows the associativity of the Young integral.

\begin{lemma}
  \label{associativityYoung} Let $x \in V^p ([0, T], \mathcal{L} (\mathbb{W},
  \mathbb{U}))$, $y \in V^p ([0, T], \mathcal{L} (\mathbb{V},
  \mathbb{W}))$ and $y \in V^q ([0, T], \mathbb{V})$ with $\frac{1}{p} +
  \frac{1}{q} > 1$. Then the (backward) Young integral is associative, i.e.
  $\int_0^T x_r d (\overleftarrow{\int_0^r y_s d \overleftarrow{z_s}}) =
  \int_0^T x_r y_r d \overleftarrow{z_r}$, if $z$ is c{\`a}gl{\`a}d or $x, y$ are both c{\`a}dl{\`a}g with $y$ and $z$ not sharing any
    common discontinuity points.
\end{lemma}

\begin{proof}
  First notice that the integral on the r.h.s. and l.h.s. both exist due to Proposition \ref{BackwardsYoung} and Corollary \ref{integral_q-var}. 
  Therefore, for an arbitrary sequence of time
  partitions $\pi = (\pi^n)_{n \in \NN}$ on $[0, T]$ with $| \pi^n | \to 0$ as
  $n \to \infty$, it holds
  \begin{align*}
    & \int_0^T x_r d \bigg( \overleftarrow{\int_0^r y_s d
    \overleftarrow{z_s}} \bigg) - \int_0^T x_r y_r d \overleftarrow{z_r}\\
    = & \lim_{n \to \infty}  \sum_{t_i^n \in \pi^n} x_{t^n_{i + 1}}  \bigg(
    \int_{t^n_i}^{t^n_{i + 1}} y_r d \overleftarrow{z_r} - y_{t^n_i} z_{t^n_i,
    t^n_{i + 1}} \bigg)\\
    \leq & C (| x_0 | + \| x \|_{p - var ; [0, T]}) \lim_{n \to \infty} 
    \sum_{t_i^n \in \pi^n} \| y \|_{p ; (t^n_i, t^n_{i + 1}]} \| z \|_{q ;
    [t^n_i, t^n_{i + 1})}\\
    \leq & C \lim_{n \to \infty} \sup_{t_i^n \in \pi^n} (\| y \|^{\varepsilon}_{p
    ; (t^n_i, t^n_{i + 1}]} \| z \|^{\varepsilon}_{q ; [t^n_i, t^n_{i + 1})}) 
    \sum_{t_i^n \in \pi^n} \omega (t^n_i, t^n_{i + 1})\\
    \leq & C \omega (0, T) \lim_{n \to \infty} \sup_{t_i^n \in \pi^n} (\| y
    \|^{\varepsilon}_{p ; (t^n_i, t^n_{i + 1}]} \| z \|^{\varepsilon}_{q ; [t^n_i,
    t^n_{i + 1})}),
  \end{align*}
  here $\omega (s, t) \assign \| y \|^{1 - \varepsilon}_{p ; [s, t]} \| z \|^{1
  - \varepsilon}_{q ; [s, t]}$ defines a control (Exercise 1.10 \& Proposition
  5.8, {\cite{friz_multidimensional_2010}}).\\
  Under condition 2, if $y$ and $z$ do not share any common discontinuity
  points, then the last term obviously converges to $0$.\\
  Under condition 1, we apply Lemma \ref{BackwardsYoung} and rewrite the above to see that it is zero:
  \begin{align*}
    \int_0^T x_r d \bigg( \overleftarrow{\int_0^r y_s d \overleftarrow{z_s}}
    \bigg) - \int_0^T x_r y_r d \overleftarrow{z_r} = & \int_0^T x^-_r d
    \bigg( \overleftarrow{\int_0^r y^-_s d \overleftarrow{z_s}} \bigg) -
    \int_0^T x^-_r y^-_r d \overleftarrow{z_r}\\
    \leq & C \omega (0, T) \lim_{n \to \infty} \sup_{t_i^n \in \pi^n} (\| y^-
    \|^{\varepsilon}_{p ; (t^n_i, t^n_{i + 1}]} \| z \|^{\varepsilon}_{q ; [t^n_i,
    t^n_{i + 1})}),
  \end{align*}
  here $\| y^- \|_{p ; [s, t]}$ is left continuous by the same argument as
  Lemma 7.1 in {\cite{friz_differential_2018}}.
\end{proof}

The following lemma compares backward with forward Young integrals.

\begin{lemma}
  \label{backwardforwardYoung}Let $x \in V^p ([0, T], \mathcal{L} (\mathbb{V},
  \mathbb{W}))$ and $y \in V^q ([0, T], \mathbb{V})$ with $\frac{1}{p} +
  \frac{1}{q} > 1$. If $x$ is c{\`a}dl{\`a}g, then for any $t \in [0, T]$ it
  holds
  \begin{align*}
    \int_0^t x_r d \overleftarrow{y_r} - \int_0^t x_r \tmop{dy}_r = \sum_{0 <
    r \leq t} \Delta^- x_r \Delta^- y_r .
  \end{align*}
  If $y$ is c{\`a}gl{\`a}d, then for any $t \in [0, T]$ it holds
  \begin{align*}
    \int_0^t x_r d \overleftarrow{y_r} - \int_0^t x_r \tmop{dy}_r = \sum_{0
    \leq r < t} \Delta^+ x_r \Delta^+ y_r .
  \end{align*}
  If $x$ is c{\`a}dl{\`a}g and $y$ is c{\`a}gl{\`a}d, the
  forward and backward Young integrals are equal.
\end{lemma}

\begin{proof}
  We only show it for the case where $x$ is c{\`a}dl{\`a}g, the proof is the
  same for $y$ being c{\`a}gl{\`a}d. Let $\pi = (\pi^n)_{n \in \NN}$ be a sequence of time partitions on $[0, t]$ with $| \pi^n | \to 0$ as $n \to \infty$, such that $\int_0^t x_r \tmop{dy}_r = \lim_{n \to \infty}  \sum_{t_i^n \in \pi^n}
     x_{t^n_i} y_{t^n_i, t^n_{i + 1}}$ holds in RRS sense, such sequence exists due to Proposition 2.4 in
  {\cite{friz_differential_2018}}.
  Since $x$ is c{\`a}dl{\`a}g, we have that
  \[ \int_0^t x_r d \overleftarrow{y_r} = \lim_{n \to \infty}  \sum_{t_i^n \in
     \pi^n} x_{t^n_{i + 1}} y_{t^n_i, t^n_{i + 1}} \infixand \int_0^t y_r
     \tmop{dx}_r = \lim_{n \to \infty}  \sum_{t_i^n \in \pi^n} y_{t^n_i}
     x_{t^n_{i + 1}, t^n_{i + 1}} \]
  for the same sequence $\pi$, the first convergence follows by Proposition
  \ref{BackwardsYoung} and the second by Theorem 2.2 and Proposition 2.4 in
  {\cite{friz_differential_2018}}. Combining the previous convergence implies
  \begin{align*}
    \int_0^t x_r d \overleftarrow{y_r} - \int_0^t x_r \tmop{dy}_r & = \lim_{n
    \to \infty}  \sum_{t_i^n \in \pi^n} x_{t^n_{i + 1}, t^n_{i + 1}} y_{t^n_i,
    t^n_{i + 1}}\\
    & = \lim_{n \to \infty}  \sum_{t_i^n \in \pi^n} x_{t^n_{i + 1}} y_{t^n_{i
    + 1}} - x_{t^n_i} y_{t^n_i} - y_{t^n_i} x_{t^n_{i + 1}, t^n_{i + 1}} -
    x_{t^n_i} y_{t^n_{i + 1}, t^n_{i + 1}}\\
    & = x_T y_T - x_0 y_0 - \int_0^t y_r \tmop{dx}_r - \int_0^t x_r
    \tmop{dy}_r .
  \end{align*}
  Finally, applying the product formula for (forward) Young integral in
  {\cite{friz_differential_2018}} results to $x_T y_T - x_0 y_0 - \int_0^t y_r \tmop{dx}_r - \int_0^t x_r \tmop{dy}_r =
    \sum_{0 < r \leq T} \Delta^- x \Delta^- y$.
\end{proof}

The following lemma states the difference between integrating against a
c{\`a}gl{\`a}d path $y$ or it's limit from the right $y^+$.
\begin{lemma}
  \label{dAdA+}Let $x \in V^p ([0, T], \mathcal{L} (\mathbb{V}, \mathbb{W}))$
  and $y \in V^q ([0, T], \mathbb{V})$ with $\frac{1}{p} + \frac{1}{q} > 1$
  and further assume $y$ to be c{\`a}gl{\`a}d and $x$ to be c{\`a}dl{\`a}g.
  Then for any $t \in [0, T]$ it holds
  \begin{align*}
    \int_0^t \tmop{xdy}^+ = \int_0^t \tmop{xdy} + x_t \Delta^+ y_t - x_0
    \Delta^+ y_0 = \int_0^t xd \overleftarrow{y } + x_t \Delta^+ y_t - x_0
    \Delta^+ y_0 .
  \end{align*}
\end{lemma}

\begin{proof}
  The second equality follows directly from the previous lemma, we only need
  to show the first equality. Let $\pi = (\pi^n)_{n \in \NN}$ be a sequence of partitions on $[0, t]$ with $| \pi^n | \to 0$, such that $\int_0^t x_r \tmop{dy} = \lim_{n \to \infty}  \sum_{t_i^n \in \pi^n}
     x_{t^n_i} y_{t^n_i, t^n_{i + 1}}$ holds in RRS sense, such sequence exists due to Proposition 2.4 in
  {\cite{friz_differential_2018}}. By the same proposition, it also holds that
  \begin{align*}
    \int_0^t x_r \tmop{dy}^+_r & = \lim_{n \to \infty}  \sum_{t_i^n \in \pi^n}
    x_{t^n_i} y^+_{t^n_i, t^n_{i + 1}} = \lim_{n \to \infty}  \sum_{t_i^n \in \pi^n} x_{t^n_i} y_{t^n_i,
    t^n_{i + 1}} + x_{t^n_i} \Delta^+ y_{t^n_{i + 1}} - x_{t^n_i} \Delta^+
    y_{t^n_i}\\
    & = \int_0^t x_r \tmop{dy} - \lim_{n \to \infty}  \sum_{t_i^n \in \pi^n}
    x_{t^n_i, t^n_{i + 1}} \Delta^+ y_{t^n_{i + 1}} + \lim_{n \to \infty} 
    \sum_{t_i^n \in \pi^n} x_{t^n_{i + 1}} \Delta^+ y_{t^n_{i + 1}} -
    x_{t^n_i} \Delta^+ y_{t^n_i}\\
    & = \int_0^t x_r \tmop{dy}_r + x_t \Delta^+ y_t - x^+_0 \Delta^+ y_0,
  \end{align*}
  where for the last equality, the second term converges to zero due to the
  mild sewing lemma for pure jumps (Theorem 2.11) in
  {\cite{friz_differential_2018}} and the last term is a telescope sum.
\end{proof}

We prove a stability result for the backward Young integral.

\begin{proposition}
  \label{stability_youngIntergral}Let $x^1, x^2 \in V^p ([0, T], \mathcal{L}
  (\mathbb{V}, \mathbb{W}))$ and $y^1, y^2 \in V^q ([0, T], \mathbb{V})$ with
  $\frac{1}{p} + \frac{1}{q} > 1$ and define $\Delta x \assign x^1 - x^2$ and
  $\Delta y \assign y^1 - y^2$, then
  \begin{align*}
 \lefteqn{\big\| \int_0^. x^1_r d \overleftarrow{y_r^1} - \int_0^. x^2_r d
    \overleftarrow{y_r^2} \big\|_{p ; [0, T]}} \\ 
    & \leq   (\| \Delta x \|_{p ;
    [0, T]} + | \Delta x_T |) \| y^1 \|_{q ; [0, T]} + (\| x^2 \|_{p ; [0, T]}
    + | x_T^2 |) \| \Delta y \|_{q ; [0, T]} .
  \end{align*}
\end{proposition}

\begin{proof}
  For $0 \leq s \leq u \leq t \leq T$, we define $\Delta_{s, t} \assign x^1_t
  y^1_{s, t} - x^2_t y^2_{s, t}$, then we have
  \begin{align*}
    \delta \Delta_{s, u, t} \assign \Delta_{s, t} - \Delta_{s, u} - \Delta_{u,
    t} = \Delta x_{u, t} y^1_{s, u} + x^2_{u, t} \Delta y_{s, u} .
  \end{align*}
  Now for the controls $\omega^{1, 1} (s, t) \assign \| y^1 \|^q_{q ; [s,
  t]}$, $\omega^{1, 2} (s, t) \assign \| \Delta x \|^p_{p ; [s, t]}$,
  $\omega^{2, 1} (s, t) \assign \| \Delta y \|^q_{q ; [s, t]}$ and $\omega^{2,
  2} (s, t) \assign \| x^2 \|^p_{p ; [s, t]}$, it holds
  \begin{align*}
    | \delta \Delta_{s, u, t} | & \leq \| y^1 \|_{q ; [s, u]} \| \Delta x
    \|_{p ; [u, t]} + \| \Delta y \|_{q ; [s, u]} \| x^2 \|_{p ; [u, t]}\\
    & = \omega^{1, 1} (s, u)^{\frac{1}{q}} \omega^{1, 2} (u, t)^{\frac{1}{p}}
    + \omega^{2, 1} (s, u)^{\frac{1}{q}} \omega^{2, 2} (u, t)^{\frac{1}{p}} .
  \end{align*}
  Let $\pi = (\pi^n)_{n \in \NN}$ be a sequence of time partitions on $[0, T]$
  of the form $\pi^n = (0 = t^n_0 < t^n_1 < \cdots < t^n_N = T)$ with $| \pi^n
  | \to 0$ as $n \to \infty$. Proposition \ref{BackwardsYoung} implies that
  \begin{align*}
    \int_0^. x^1_r d \overleftarrow{y_r^1} - \int_0^. x^2_r d
    \overleftarrow{y_r^2} = \lim_{n \to \infty}  \sum_{t^n_i \in \pi^n}
    \Delta_{t^n_i, t^n_{i + 1}} .
  \end{align*}
  By the general sewing theorem (Theorem 2.5,
  {\cite{friz_differential_2018}}), we know that
  \begin{align*}
    \bigg| \int_0^. x^1_r d \overleftarrow{y_r^1} - \int_0^. x^2_r d
    \overleftarrow{y_r^2} - \Delta_{0, T} \bigg| \leq \| \Delta x \|_{p ; [0,
    T]} \| y^1 \|_{q ; [0, T]} + \| x^2 \|_{p ; [0, T]} \| \Delta y \|_{q ;
    [0, T]}.
  \end{align*}
  Then we get the desired estimate by the same argument as in Corollary \ref{integral_q-var}.
\end{proof}

We have the following measurable selection result for the Young integral.
\begin{lemma}
  \label{Measurable-Young}Let $(\Omega^1, \mathcal{F}^1, (\mathcal{F}^1_t)_{t
  \in [0, T]}, \mathbb{P}^1)$ be a filtered probability space and $(U,
  \mathcal{U})$ be some measurable space. Given some $p, q > 0$ satisfying
  $\frac{1}{p} + \frac{1}{q} > 1$, let $x : \Omega^1 \times [0, T]) \times U
  \rightarrow \mathcal{L} (\mathbb{V}, \mathbb{W})$ and $y : ([\Omega^1 \times
  0, T]) \times U \rightarrow \mathbb{V}$ be $\tmop{Prog} \otimes
  \mathcal{U}$-measurable processes such that for every $u \in U$ the process
  $(x^u_t)_{t \in [0, T]}$ is of finite p-variation and $(y^u_t)_{t \in [0,
  T]}$ is of finite $q$-variation and furthermore c{\`a}gl{\`a}d. Then there
  exists a $\tmop{Prog} \otimes \, \mathcal{U}$-measurable processes $I :
  ([\Omega^1 \times 0, T]) \times U \rightarrow \mathbb{V}$ such that for
  every $u \in U$, $I^u$ is indistinguishable from $\int_0^{\cdot} x^u_r d
  \overleftarrow{y^u_r}$.
\end{lemma}

\begin{proof}
  Given any sequence of finite time partitions $(\pi^n)_{n \in \NN}$ with
  vanishing meshsize, we have by Proposition \ref{BackwardsYoung} and
  Corollary \ref{integral_q-var} that, for any $u \in U$ and any $t \in [0,
  T]$ that $\int_0^t x^u_r d \overleftarrow{y^u_r} = \lim_{n \to \infty} 
  \sum_{t_i^n \in \pi^n} x^u_{t \wedge t^n_{i + 1}} y^u_{t \wedge t^n_i, t
  \wedge t^n_{i + 1}}$ almost surely. Obviously, for each $n$ the sum is
  $\mathcal{F}_t \otimes \mathcal{U}$-measurable. Then by Proposition 1 in
  {\cite{stricker_calcul_1978}} and the fact that $\int_0^{\cdot} x^u_r d
  \overleftarrow{y^u_r}$ is c{\`a}gl{\`a}d from Corollary
  \ref{integral_q-var}, there exists a $\tmop{Prog} \otimes
  \mathcal{U}$-measurable version $I$ of $\int_0^{\cdot} x^u_r d
  \overleftarrow{y^u_r}$.
\end{proof}

Replacing the parameter space $(U, \mathcal{U})$ with an actual
probability space yields.
\begin{proposition}
  \label{randomizedYoung}Let $(\Omega^1, \mathcal{F}^1, (\mathcal{F}^1_t)_{t
  \in [0, T]}, \mathbb{P}^1)$ and $(\Omega^2, \mathcal{F}^2,
  (\mathcal{F}^2_t)_{t \in [0, T]}, \mathbb{P}^2)$ be two filtered probability
  spaces and let $(\Omega, \mathcal{F}, (\mathcal{F}_t)_{t \in [0, T]},
  \mathbb{P})$ denotes the product space. Given some $p, q > 0$ satisfying
  $\frac{1}{p} + \frac{1}{q} > 1$, let $x : [0, T] \times \Omega \rightarrow
  \mathcal{L} (\mathbb{V}, \mathbb{W})$ and $y : [0, T] \times \Omega
  \rightarrow \mathbb{V}$ be $\mathcal{F}_t$-progressively measurable
  processes such that $x$ and $y$ are of finite p-variation and $q$-variation
  with $y$ additionally assumed to be c{\`a}gl{\`a}d. The pathwise defined
  backward Young integral $\int_0^{\cdot} x_r d \overleftarrow{y_r}$ is then
  $\mathcal{F}_t$-progressively measurable and it holds
  \begin{equation}
    \int_0^T x_r d\overleftarrow{y_{r}} = \int^T_0 x_r (\cdot, u) d\overleftarrow{y_{r}} (\cdot, u)
    \mid_{u = \omega^2}, \label{quenched-Young}
  \end{equation}
  where the integral on the r.h.s. is the measurable version obtained
  in Lemma \ref{Measurable-Young}.
\end{proposition}

\begin{proof}
  Given any sequence of finite time partitions $(\pi^n)_{n \in \NN}$ with
  vanishing meshsize. We have $\mathbb{P}$-a.s. that
  \[  \sum_{t_i^n \in \pi^n} x_{t \wedge t^n_{i + 1}} (\omega^1, \omega^2)
     y_{t \wedge t^n_i, t \wedge t^n_{i + 1}} (\omega^1, \omega^2) =
     \sum_{t_i^n \in \pi^n} x_{t \wedge t^n_{i + 1}} (\omega^1, u) y_{t \wedge
     t^n_i, t \wedge t^n_{i + 1}} (\omega^1, u) \mid_{u = \omega^2}, \]
  where the left side converges $\mathbb{P}$-a.s. to $\int^t_0 x  d
  \overleftarrow{y}$ and the right side converges for fixed $\omega^2 \in
  \Omega^2$ to measurable version in Lemma \ref{Measurable-Young}, hence
  proving \eqref{quenched-Young}. The $\mathcal{F}_t$-progressively
  measurability follows by the same argument as Lemma \ref{Measurable-Young}.
\end{proof}

\section{An Extension of It{\^o}'s Formula}\label{ItoFormula}

Our arguments require a version of It{\^o}'s formula that is applicable to processes given by sums of (continuous) local martingales and  c{\`a}gl{\`a}d processes of
finite $q$-variation for $q < 2$. Since such processes are
(pathwise) of quadratic variation, so even though the standard It{\^o}'s formula
is not applicable and the results in {\cite{friz_rough_2024}} are only for
continuous processes, one can still adapt ideas for the pathwise It{\^o}'s formula  from F{\"o}llmer {\cite{follmer_calcul_1981}} to our setting. Another
option could be to adapt the It{\^o}'s formula for weak Dirichlet processes by {\cite{bandini_weak_2024}}.  We want to note,
that yet neither {\cite{follmer_calcul_1981}} nor {\cite{bandini_weak_2024}} provides the result as required directly, as they do not
work with (backward) Young integrals and their results are stated for c{\`a}dl{\'a}g
(instead of c{\`a}gl{\`a}d) processes. For completeness, this appendix thus provides a suitable adaption of It{\^o} formula, although arguments may be folklore.

Let $\pi = (\pi^n)_{n \in \NN}= \{0 = t^n_0 < t^n_1 < \cdots < t^n_N = T\}_{n \in \NN}$ be a sequence of time partitions on $[0, T]$. A
c{\`a}dl{\'a}g process $x$ (with values in $\mathbb{R}$) is then said to have
quadratic variation along $\pi$ if the sequence of measures $\sum_{t_i^n \in
\pi^n} (x_{t_{i + 1}^n} - x_{t_i^n})^2 \delta_{t_i^n}$, where $\delta_t$
denotes the Dirac measure at point $t$, converges weakly to a Radon measure
$\mu^{\pi}$ such that $t \mapsto [x]^{\pi}_c (t) \assign \mu^{\pi} ([0, t]) -
\sum_{0 < s \leq t} (\Delta^- x_s)^2$ is a continuous and increasing function.
The quadratic variation of $x$ along $\pi$ is then defined as $[x]^{\pi} (t)
\assign \mu^{\pi} ([0, t]) = [x]^{\pi}_c (t) + \sum_{0 < s \leq t} (\Delta^-
x_s)^2$. In the case of $x$ being c{\`a}gl{\`a}d, we adapt the definition
accordingly to $[x]^{\pi}_c (t) \assign \mu^{\pi} ([0, t]) - \sum_{0 \leq s <
t} (\Delta^+ x_s)^2$ and $[x]^{\pi} (t) \assign \mu^{\pi} ([0, t]) =
[x]^{\pi}_c (t) + \sum_{0 \leq s < t} (\Delta^+ x_s)^2$.

Now let $x = (x^1, \cdots, x^n)$ be a c{\`a}dl{\`a}g function taking values
in $\mathbb{R}^n$. We say that $x$ is of quadratic variation along $\pi$ if
the processes $x^i$, $x^i + x^j$ are of quadratic variation along $\pi$ for
all $1 \leq i, j \leq n$. In this case, we define
\begin{align*}
  {}[x^i, x^j] (t) & = \frac{1}{2}  ([x^i + x^j] (t) - [x^i] (t) - [x^j] (t))
  = [x^i, x^j]^c (t) + \sum_{0 \leq s < t} \Delta^+ x^i_s \Delta^+ x^j_s .
\end{align*}
We start by showing the following result about the quadratic variation of the
sum of a path of $q$-variation for $q < 2$ and a (continuous) local
martingale.

\begin{lemma}
  \label{quadratic_variation_M+A}Let $A$ be a c{\`a}dl{\'a}g process of finite
  $q$-variation for $q < 2$ and $M$ be a continuous local martingale. Let
  $\pi$ be a sequence of time partitions such that the
  sum $S^{\pi^n} (t) \assign \sum_{t_i^n \in \pi^n} (M_{t_{i + 1}^n \wedge t}
  - M_{t_i^n \wedge t})^2$ converges to the stochastic quadratic variation
  $[M] (t)$ almost surely for all $t \in [0, T]$, then $[M]^{\pi} = [M]$ and $[M + A]^{\pi} = [M] +
  \sum_{0 < s \leq \cdot} (\Delta^- A)^2$ hold almost surely for all $t \in [0, T]$.
\end{lemma}

\begin{proof}
  The convergence of $S^{\pi^n}$ to $[M]$ implies the weak convergence of the
  related measure $\mu_n \assign \sum_{t_i^n \in \pi^n} (M_{t_{i + 1}^n} -
  M_{t_i^n})^2 \delta_{t_i^n}$ to $\mu$ with $\mu ([0, t]) \assign [\pi] (t)$ (cf. \cite[Theorem 2.1]{billingsley_convergence_1999}). By
  definition of the pathwise quadratic variation, we have $[M]^{\pi} (t) = [M]
  (t)$ a.s. for all $t$. For the second equation, we define $Y = M + A$ and write
  \begin{align*}
    \sum_{t_i^n \in \pi^n} (Y_{t_{i + 1}^n \wedge t} - Y_{t_i^n \wedge t})^2 =
    & \sum_{t_i^n \in \pi^n} (M_{t_{i + 1}^n \wedge t} - M_{t_i^n \wedge t})^2
    + \sum_{t_i^n \in \pi^n} (A_{t_{i + 1}^n \wedge t} - A_{t_i^n \wedge
    t})^2\\
    & + 2 \sum_{t_i^n \in \pi^n} (M_{t_{i + 1}^n \wedge t} - M_{t_i^n \wedge
    t})  (A_{t_{i + 1}^n \wedge t} - A_{t_i^n \wedge t}) .
  \end{align*}
  The first sum converges to $[M]_t$ almost surely. For the second sum, we
  define $J^{\varepsilon}$ to be the set of jumps in $A$ which are larger than
  $\varepsilon$, this set is finite for any $\varepsilon$, since $A$ has finite $q$-variation. Denote by
  $\bar{J}^{\varepsilon}$ the set of the other jumps of $A$. Then we have
  \begin{align*}
    \sum_{t_i^n \in \pi^n} (A_{t_{i + 1}^n \wedge t} - A_{t_i^n \wedge t})^2 =
    \sum_{t_i, J^{\varepsilon}} (A_{t_{i + 1}^n \wedge t} - A_{t_i^n \wedge t})^2
    + \sum_{t_i, \bar{J}^{\varepsilon}} (A_{t_{i + 1}^n \wedge t} - A_{t_i^n
    \wedge t})^2,
  \end{align*}
  where $\sum_{t_i, J^{\varepsilon}}$ denotes $\sum_{t_i^n \in \pi^n}
  \mathds{1}_{\{J^{\varepsilon} \cap (t_i, t_{i + 1}] \neq \emptyset\}}$ for $A$
  c{\`a}dl{\'a}g. The first term converges to $\sum_{s \in J^{\varepsilon}}
  (\Delta^- A_s)^2$ as $n \to \infty$, and we can bound the second term by
  $\sum_{t_i, \bar{J}^{\varepsilon}} (A_{t_{i + 1}^n \wedge t} - A_{t_i^n \wedge
  t})^2 \leq \varepsilon^{2 - q} \| A \|^q_{q - var} \to 0$. So for $n \to
  \infty$ and $\varepsilon \to 0$ it holds $\sum_{t_i^n \in \pi^n} (A_{t_{i +
  1}^n \wedge t} - A_{t_i^n \wedge t})^2 \to \sum_{0 < s \leq t} (\Delta^-
  A_s)^2$. For the third sum, we can apply H{\"o}lder inequality to get
  \begin{align*}
    &\sum_{t_i^n \in \pi^n} (M_{t_{i + 1}^n \wedge t} - M_{t_i^n \wedge t}) 
    (A_{t_{i + 1}^n \wedge t} - A_{t_i^n \wedge t})\\
    \leq & \sup_{t_i^n \in \pi^n} | M_{t_{i + 1}^n} - M_{t_i^n} |^{\frac{q}{q
    - 1} - 2 - \delta} \big( \sum_{t_i^n \in \pi^n} | M_{t_{i + 1}^n \wedge
    t} - M_{t_i^n \wedge t} |^{2 + \delta} \big)^{\frac{q - 1}{q}} \big(
    \sum_{t_i^n \in \pi^n} | A_{t_{i + 1}^n \wedge t} - A_{t_i^n \wedge t} |^q
    \big)^{\frac{1}{q}}\\
    \leq & \sup_{t_i^n \in \pi^n} | M_{t_{i + 1}^n} - M_{t_i^n} |^{\frac{q}{q
    - 1} - 2 - \delta} \| M \|_{2 + \delta - var}^{\frac{q - 1}{q (2 +
    \delta)}} \| A \|_{q - var}^q
  \end{align*}
  for some $\delta > 0$ with $2 + \delta < \frac{q}{q - 1}$. Due to the
  continuity of $M$, this term converges to $0$ as $n \to \infty$. In total we
  have $\sum_{t_i^n \in \pi^n} (Y_{t_{i + 1}^n \wedge t} - Y_{t_i^n \wedge
  t})^2 \to [M]_t + \sum_{0 < s \leq t} (\Delta^- A_s)^2$. As in the proof
  for the first equation, this implies the weak convergence of the related
  measure and we get $[Y]^{\pi} = [M] + \sum_{0 < s \leq \cdot} (\Delta^-
  A_s)^2$.
\end{proof}

\begin{remark}
  With an analogous proof, one can show that for $A$ being c{\`a}gl{\`a}d, we
  have $[M + A]^{\pi} = [M] + \sum_{0 \leq s < \cdot} (\Delta^+ A_s)^2$ almost
  surely.
\end{remark}


\begin{theorem}
  \label{Ito}Let $A$ be a c{\`a}gl{\`a}d process of finite $q$-variation for
  $q < 2$ and $M$ be a continuous local martingale. We define $Y = A + M$,
  then for any $f \in C^2  (\mathbb{R}, \mathbb{R})$ we have
  \begin{align*}
    f (Y_t) = & f (Y_0) + \int_0^t f' (Y^+_s) \tmop{dA}_s + \int_0^t f' (Y_s)
    \tmop{dM}_s + \int_0^t f'' (Y_s) d [M]_s.
  \end{align*}
  Equivalently, this It\^o-formula can be written with backward Young integrals as
  \begin{align*}
    f (Y_t) = & f (Y_0) + \int_0^t f' (Y_s) d \overleftarrow{A_s} + \int_0^t
    f' (Y_s) \tmop{dM}_s + \int_0^t f'' (Y_s) d [M]_s\\
    & + \sum_{0 \leq s < t} (f (Y_{s +}) - f (Y_s) - f' (Y_s) \Delta^+ A_s) .
  \end{align*}
\end{theorem}

\begin{remark}
  \label{ItoJump}The absolute convergence of the sum can be shown
  by Taylor's formula
  \begin{align*}
    & \bigg\| \sum_{0 \leq s < t} (f (Y_{s +}) - f (Y_s) - f' (Y_s) \Delta^+
    Y_s) \bigg\|\\
    \leq & | f'' |_{\infty } \sum_{0 \leq s < t} (\Delta^+ Y_s)^2 =
    | f'' |_{\infty } \sum_{0 \leq s < t} (\Delta^+ A_s)^2 \leq
    | f'' |_{\infty } \| A \|_{q, [0, t)} .
  \end{align*}
\end{remark}

\begin{remark}
  \label{multidimIto}Similar to the paper {\cite{follmer_calcul_1981}} by
  F{\"o}llmer, for $Y$ now being $n$-dimensional and $f \in C (\mathbb{R}^n,
  \mathbb{R})$, the It{\^o}'s formula is the same as the one-dimensional case,
  just with more cumbersome notation. We will therefore omit the proof and
  only state the formula:
  \begin{align}
    f (Y_t) = & f (Y_0) + \sum_{i = 1}^n \int_0^t D_i f (Y_s) d
    \overleftarrow{A^i_s} + \sum_{i = 1}^n \int_0^t D_i f (Y^i_s) dM^i_s +
    \sum_{i, j = 1}^n \int_0^t D_i D_j f (Y_s) d [M^i, M^j]_s \nonumber\\
    & + \sum_{0 \leq s < t} \big( f (Y_{s +}) - f (Y_s) - \sum_{i = 1}^n D_i
    f (Y_s) \Delta^+ A^i_s \big) . 
  \end{align}
\end{remark}

\begin{proof}
  For any sequences of time partitions $\bar{\pi}$ on $[0, T]$ with vanishing
  mesh $| \bar{\pi}^n | \to 0$, we know that $\sum_{t_i^n \in \bar{\pi}^n} f'
  (Y^+_{t_i})  (M_{t_{i + 1}^n \wedge \cdot} - M_{t_i^n \wedge \cdot}) \to
  \int_0^{\cdot} f' (Y) \tmop{dM}$ and $\sum_{t_i^n \in \bar{\pi}^n} (M_{t_{i
  + 1}^n \wedge \cdot} - M_{t_i^n \wedge \cdot})^2 \to [M]$ converge both in
  ucp (see {\cite{protter_stochastic_2004}}, Theorem II.21 and II.22), then along a subsequence $\pi \subset \bar{\pi}$ the convergence also
  holds almost surely for all $t$. By Lemma \ref{quadratic_variation_M+A} we have $Y$ is
  of quadratic variation along $\pi$ and one can easily see from the
  definition of pathwise quadratic variation that $Y^+$ is also of quadratic
  variation with $[Y^+]^{\pi} (t) = [M] (t) + \sum_{0 < s \leq \cdot}
  (\Delta^+ A_s)^2$, since $A^+_{s -} = A_s$. Now for any $f \in C^2 
  (\mathbb{R}, \mathbb{R})$ we can apply the It{\^o}'s formula from
  {\cite{follmer_calcul_1981}} and get
  \begin{align}
    f (Y^+_t) = & f (Y^+_0) + \int_0^t f' (Y^+) d^{\pi} Y^+ + \frac{1}{2} 
    \int_0^t f'' (Y^+) d [M] \nonumber\\
    & + \sum_{0 < s \leq t} (f (Y^+_s) - f (Y_s) - f' (Y_s) (Y^+_s - Y_s)),
    \nonumber
  \end{align}
  where $\int_0^t f' (Y^+) d^{\pi} Y^+ = \lim_{n \to \infty}  \sum_{t_i^n \in
  \pi^n} f' (Y^+_{t_i})  (Y^+_{t_{i + 1} \wedge t} - Y^+_{t_i \wedge t})$.
  The above equation can be easily transformed into
  \begin{align}
    f (Y_t) = & f (Y_0) + \int_0^t f' (Y^+) d^{\pi} Y^+ + \frac{1}{2} 
    \int_0^t f'' (Y^+) d [M] \label{ItoProof}\\
    & + \sum_{0 \leq s < t} (f (Y^+_s) - f (Y_s) - f' (Y_s) (A^+_s - A_s)) -
    f' (Y^+)_T \Delta^+ A_T + f' (Y^+)_0 \Delta^+ A_0.\nonumber  
  \end{align}
  By Proposition 2.4 in {\cite{friz_differential_2018}} and Lemma \ref{dAdA+} we obtain
  \begin{eqnarray}
   \lefteqn{\int_0^t f' (Y^+) \tmop{dA} + f' (Y^+)_T \Delta^+ A_T - f' (Y^+)_0
    \Delta^+ A_0} \nonumber\\
   &= \int_0^t f' (Y^+) \tmop{dA} & = \lim_{n \to \infty}  \sum_{t_i^n \in
    \pi^n} f' (Y^+_{t_i})  (A^+_{t_{i + 1} \wedge t} - A^+_{t_i \wedge t}), 
    \label{Ito_dA+}\\ 
    \text{and}
    &\int_0^t f'' (Y^+)  d [M] &= \int_0^t f'' (Y) d [M] .
    \label{ItoIntQua} 
  \end{eqnarray}
  Now by our specific choice of $\pi$ and equations (\ref{Ito_dA+}) and
  (\ref{ItoIntQua}), we have
  \begin{align}
    \int_0^t f' (Y^+) d^{\pi} Y^+ &=  \lim_{n \to \infty}  \sum_{t_i^n \in
    \pi^n} f' (Y^+_{t_i})  (M^+_{t_{i + 1} \wedge t} - M^+_{t_i \wedge t}) + f' (Y^+_{t_i})  (A^+_{t_{i +
    1} \wedge t} - A^+_{t_i \wedge t}) \label{ItodYpi}\\
      = &\int_0^t f' (Y) \tmop{dM} + \int_0^t f' (Y^+) \tmop{dA} + f' (Y^+)_T
    \Delta^+ A_T - f' (Y^+)_0 \Delta^+ A_0 .  \nonumber
  \end{align}
  So, combining (\ref{ItoProof}), (\ref{ItoIntQua}) and (\ref{ItodYpi}) yields the desired variant of It{\^o}'s formula. Applying Lemma \ref{backwardforwardYoung} and then Lemma
  \ref{backwardYoungJump} yields the second formula claimed. 
\end{proof}

\textbf{Acknowledgments:} The authors acknowledge funding by the
Deutsche Forschungsgemeinschaft (DFG, German Research Foundation) – CRC/TRR 388
"Rough Analysis, Stochastic Dynamics and Related Fields" – Project ID 516748464. We thank Peter Friz for valuable discussions. Y.S. thanks Joscha Diehl for his hospitality during the visit to Greifswald. 

\bibliographystyle{amsplain} 
\bibliography{library}

\end{document}